\documentclass[reqno,fleqn]{amsart}
\usepackage{amsfonts,amsthm,amssymb}
\usepackage{lipsum}
\usepackage[utf8]{inputenc}
\usepackage[english]{babel}

\usepackage{lmodern,textcomp}
\usepackage{amsmath}
\usepackage{mathtools}
\usepackage{latexsym}
\usepackage{tikz}
\usepackage{fancyvrb}
\usepackage{epsfig}
\usepackage{pstricks,slashbox,multirow}
\usepackage{graphicx}
\usepackage{rotating}
\usetikzlibrary{positioning}
\usepackage{subcaption}
\usepackage{datetime}
\newtheorem{theorem}{Theorem}

\newtheorem{lemma}[theorem]{Lemma}

\newtheorem{cor}[theorem]{Corollary}
\newtheorem{definition}[theorem]{Definition}

\newtheorem{prm}{Problem}

\newtheorem{rem}[theorem]{Remark}

\newtheorem{note}{Note}

\title[A study on Type-2 isomorphic $C_n(R)$: Part 1: Type-2 isomorphic $C_n(R)$ w.r.t. $m$ = 2]{A study on Type-2 isomorphic circulant graphs. \\ Part 1: Type-2 isomorphic circulant graphs $C_n(R)$ w.r.t. $m$ = 2}

\author{\sc Vilfred Kamalappan} 
\address{Department of Mathematics, Central University of Kerala, Periye, Kasaragod, Kerala, India - 671 316.}
\email{vilfredkamal@gmail.com}  



\date{}

\begin{document} 

\begin{abstract} This article is the first part of a detailed study on Type-2 isomorphic circulant graphs having ten parts \cite{v2-1}-\cite{v2-10}. Circulant graphs $C_n(R)$ and $C_n(S)$ are said to be \emph{Adam's isomorphic} if there exist some $a\in \mathbb{Z}_n^*$ such that $S = a R$ under arithmetic reflexive modulo $n$ \cite{ad67}. In 1970, Elspas and Turner \cite{eltu} raised a question on the isomorphism of $C_{16}(1,3,7)$ and $C_{16}(2,3,5)$. The author \cite{v96}, in 1996, defined Type-2 isomorphism of $C_n(R)$ w.r.t. $m$ $\ni$ $m$ = $\gcd(n, r) > 1$, $r\in R$ and $r,n\in\mathbb{N}$ and proved that $C_{16}(1,3,7)$ and $C_{16}(2,3,5)$ are Type-2 isomorphic w.r.t. $m$ = 2; studied Type-2 isomorphic w.r.t. $m$ = 2 in \cite{v13, v20}, and with Wilson \cite{vw1} - \cite{vw3} obtained families of such graphs for $m$ = 3,5,7; in \cite{v25}, modified the definition to $m$ is a divisor of $\gcd(n, r)$ and $r\in R$; and in this paper, we further modify the definition to $m$ and $m^3$ are divisors of $\gcd(n, r)$ and $n$, respectively, and $r\in R$. Type-2 isomorphism of circulant graphs is an isomorphism different from already known Adam's isomorphism or {\em Type-1 isomorphism}. Using the modified definition \ref{d16}, we present our study on Type-2 isomorphism of circulant graphs $C_n(R)$ w.r.t. $m$ = 2. We reprove that $(i)$ $C_{16}(1,2,7)$ and $C_{16}(2,3,5)$ are Type-2 isomorphic w.r.t. $m$ = 2; $(ii)$ For $n \geq 2$, $k \geq 3$, $1 \leq 2s-1 \leq 2n-1$, $n \neq 2s-1$, $R$ = $\{2, 2s-1, 4n-(2s-1)\}$ and $S$ = $\{2, 2n-(2s-1), 2n+2s-1\}$, $C_{8n}(R)$ and $C_{8n}(S)$ are Type-2 isomorphic w.r.t. $m$ = 2, $n,s\in\mathbb{N}$; and $(iii)$ For $n \geq 2$, $1 \leq 2s-1 < 2s'-1 \leq [\frac{n}{2}]$, $0 \leq t \leq [\frac{n}{2}]$, $R$ = $\{2,2s-1, 2s'-1\}$ and $n,s,s'\in \mathbb{N}$, if $\theta_{n,2,t}(C_n(R))$ and $C_n(R)$ are  isomorphic circulant graphs of Type-2 w.r.t. $m$ = 2 for some $t$, then $n \equiv 0~(mod ~ 8)$, $2s-1+2s'-1$ = $\frac{n}{2}$, $2s-1 \neq \frac{n}{8}$, $t$ = $\frac{n}{8}$ or $\frac{3n}{8}$, $1 \leq 2s-1 \leq \frac{n}{4}$ and $n \geq 16$ where $\theta_{n,m,t}$ is a transformation used to define Type-2 isomorphism of a circulant graph. At the end, we present a VB program POLY215.EXE which shows how Type-2 isomorphism w.r.t. $m$ = 2 of $C_{8n}(R)$ takes place for $R = \{2, 2s-1, 4n-(2s-1)\}$, $n \geq 2$ and $n,s\in {\mathbb N}$.
\end{abstract}

\maketitle

	
\section{Introduction}

\subsection{Historical Note} 

Investigation of symmetries as well as asymmetries of structures yields powerful results in mathematics. In 1846, Catalan (cf. \cite{da79}) introduced circulant matrices and thereafter properties of circulant graphs have been studied by many authors \cite{ad67}-\cite{frs}, \cite{hz14}-\cite{vw3}. Circulant graphs form a class of highly symmetric mathematical (graphical) structures and is one of the most studied graph classes having wide range of applications. Advances in science, engineering, technology and education largely dependent on computers \cite{hv}. Pattern recognition is the area where, based on some features, two different objects are matched. Graph isomorphism is a method to check whether two different graphs are similar or not and subgraph isomorphism is nothing but to identify whether an input graph is a part of full graph or not. Thus graph isomorphism is the area of pattern matching and widely used in various applications such as computer and information systems, image processing, protein structures, logic, cryptography, biometrics, compound identification in chemistry,  quantum computing, neuromorphic computing, social networks \cite{cae}, Ramsey theory, VLSI design, interconnection networks in parallel and distributed computing \cite{hz14}, the modeling of data connection networks and the theory of designs and error-correcting codes \cite{h03,st02}. 

Many authors investigated different properties of circulant graphs. Adam \cite{ad67} considered conditions for isomorphism of circulant graphs;  Boesch and Tindell \cite{bt} studied connectivity of circulant graphs; Sachs \cite{sa62} studied self-complementary circulant graphs and a conjecture stated by him on the existence of self-complementary circulant graphs was settled in \cite{amv, frs} (Also, see Theorem 3.7.5, pages 55, 56 in \cite{v96} containing solution to the conjecture.); Vilfred \cite{v13} developed a theory of Cartesian product and factorization of circulant graphs similar to that of natural numbers; a question on the isomorphism of $C_{16}(1,3,7)$ and $C_{16}(2,3,5)$ was raised by Elspas and Turner in \cite{eltu} and Vilfred \cite{v96,v20} gave its answer by defining Type-2 isomorphism; Li, Morris, Muzychuk, Palfy and Toida studied automorphism of circulant graphs \cite{li02}-\cite{mu97}, \cite{to77}. For further studies on circulant graphs, one can refer the book on circulant matrices by Davis \cite{da79} and the survey article by Kra and Simanca \cite{krsi}. 

Circulant graphs $C_n(R)$ and $C_n(S)$ are said to be \emph{Adam's isomorphic} if there exist some $a\in \mathbb{Z}_n^*$ such that $S = a R$ under arithmetic reflexive modulo $n$ \cite{ad67}. Hereafter, we call \emph{Adam's isomorphism} as {\em Type-1 isomorphism}. In 1996, Vilfred  \cite{v96} defined and studied Type-2 isomorphism, different from Type-1 isomorphism, of circulant graphs $C_n(R)$ w.r.t. $m$ $\ni$ $m$ = $\gcd(n, r) > 1$ in \cite{v17, v20}, $r\in R$ and $r,n\in\mathbb{N}$. He studied such graphs for $m$ = 2 in \cite{v20}. Vilfred and Wilson \cite{vw1} - \cite{vw3} studied Type-2 isomorphism of $C_n(R)$ w.r.t. $m$ = 3,5,7 and obtained families of such graphs, $n\in\mathbb{N}$. In \cite{v25}, the author modified the definition of Type-2 isomorphism w.r.t. $m$ by considering $m$ as a divisor of $\gcd(n, r)$ and $r\in R$. 

In this paper, the author further modifies the definition of Type-2 isomorphism of circulant graphs $C_n(R)$ w.r.t. $m$ by considering $m > 1$ is a divisor of $\gcd(n, r)$, $m^3$ is a divisor of $n$ and $r\in R$. See definition \ref{d16}. The theory so far developed on Type-2 isomorphic circulant graphs by Vilfred and Wilson is presented under ``A study on Type-2 isomorphic circulant graphs" in ten parts and this paper is its first part. This paper contains 6 sections. Section 1 is introduction and contains contents in the nine parts of ``A study on Type-2 isomorphic circulant graphs'', preliminaries and Table 1 showing most of the symbols used in this paper.  Section 2 contains our study on Adam's or Type-1 isomorphism and Type-1 group of circulant graphs. Section 3 presents our study on Type-2 isomorphism and Type-2 isomorphic circulant graphs. definition \ref{d16} is the modified definition of Type-2 isomorphism w.r.t. $m$ of circulant graph $C_{8n}(R)$.  We explain in details how Type-2 isomorphism is taking place. We prove that $(i)$ $C_{16}(1,2,7)$ and $C_{16}(2,3,5)$ are Type-2 isomorphic w.r.t. $m$ = 2 and present many problems on Type-1 and Type-2 isomorphic circulant graphs.  In Section 4, we prove that $(i)$ For $n \geq 2$, $k \geq 3$, $1 \leq 2s-1 \leq 2n-1$, $n \neq 2s-1$, $R$ = $\{2, 2s-1, 4n-(2s-1)\}$ and $S$ = $\{2, 2n-(2s-1), 2n+2s-1\}$, $C_{8n}(R)$ and $C_{8n}(S)$ are Type-2 isomorphic w.r.t. $m$ = 2, $n,s\in\mathbb{N}$; and $(ii)$ For $n \geq 2$, $1 \leq 2s-1 < 2s'-1 \leq [\frac{n}{2}]$, $0 \leq t \leq [\frac{n}{2}]$, $R$ = $\{2,2s-1, 2s'-1\}$ and $n,s,s'\in \mathbb{N}$, if $\theta_{n,2,t}(C_n(R))$ and $C_n(R)$ are  isomorphic circulant graphs of Type-2 w.r.t. $m$ = 2 for some $t$, then $n \equiv 0~(mod ~ 8)$, $2s-1+2s'-1$ = $\frac{n}{2}$, $2s-1 \neq \frac{n}{8}$, $t$ = $\frac{n}{8}$ or $\frac{3n}{8}$, $1 \leq 2s-1 \leq \frac{n}{4}$ and $n \geq 16$ where $\theta_{n,m,t}$ is a transformation used in definition \ref{d16} to define Type-2 isomorphism of circulant graphs. In Section 5, a Visual Basic (VB) program POLY215.EXE is presented to show visually how Type-2 isomorphism w.r.t. $m$ = 2 takes place on $C_{8n}(R)$ for $R$ = $\{2,2s-1,4n-(2s-1)\}$, $n\geq 2$, $1 \leq 2s-1 \leq 2n-1$ and $n,s\in\mathbb{N}$. Section 6 is the conclusion and contains suggestions for further research. 

Effort to understand the isomorphism that exists between $C_{16}(1,2,7)$ and $C_{16}(2,3,5)$, pointed out by Elspas and Turner \cite{eltu}, and to develop its general theory are the motivation to develope this theory. For all basic ideas in graph theory, we follow \cite{ha69}.

\subsection{Abstracts of the ten parts of ``A study on Type-2 isomorphic circulant graphs''}

 The ten parts of ``A study on Type-2 isomorphic circulant graphs'' are \\
\\
Part 1: Type-2 isomorphic circulant graphs $C_n(R)$ w.r.t. $m$ = 2
\\
{\em Abstract:} Definition of Type-2 isomorphism w.r.t. $m$ of  circulant graph $C_n(R)$ is modified to $m$ divides $\gcd(n, r)$, $m^3$ divides $n$ and $r\in R$. Study on Type-2 isomorphism of circulant graphs $C_n(R)$ w.r.t. $m$ = 2 is presented. It is reproved that (i) $C_{16}(1,2,7)$ and $C_{16}(2,3,5)$ are Type-2 isomorphic w.r.t. $m$ = 2; (ii) For $n \geq 2$, $k \geq 3$, $1 \leq 2s-1 \leq 2n-1$, $n \neq 2s-1$, $R$ = $\{2, 2s-1, 4n-(2s-1)\}$ and $S$ = $\{2, 2n-(2s-1), 2n+2s-1\}$, $C_{8n}(R)$ and $C_{8n}(S)$ are Type-2 isomorphic w.r.t. $m$ = 2, $n,s\in\mathbb{N}$; and (iii) For $n \geq 2$, $1 \leq 2s-1 < 2s'-1 \leq [\frac{n}{2}]$, $0 \leq t \leq [\frac{n}{2}]$, $R$ = $\{2,2s-1, 2s'-1\}$ and $n,s,s'\in \mathbb{N}$, if $\theta_{n,2,t}(C_n(R))$ and $C_n(R)$ are  isomorphic circulant graphs of Type-2 w.r.t. $m$ = 2 for some $t$, then $n \equiv 0~(mod ~ 8)$, $2s-1+2s'-1$ = $\frac{n}{2}$, $2s-1 \neq \frac{n}{8}$, $t$ = $\frac{n}{8}$ or $\frac{3n}{8}$, $1 \leq 2s-1 \leq \frac{n}{4}$ and $n \geq 16$ where $\theta_{n,m,t}$ is a transformation used to define Type-2 isomorphism of a circulant graph. A VB program POLY215.EXE which shows how Type-2 isomorphism w.r.t. $m$ = 2 of $C_{8n}(R)$ takes place is presented for $R$ = $\{2, 2s-1, 4n-(2s-1)\}$, $n \geq 2$ and $n,s\in {\mathbb N}$.\\
\\
Part 2: Type-2 isomorphic circulant graphs of orders 16, 24, 27
\\
{\em Abstract:} Obtain all 8 and 32 pairs of Type-2 isomorphic circulant graphs of orders 16 and 24, respectively and all 27 triples of Type-2 isomorphic circulant graphs of order 27. \\
\\
Part 3: 384 pairs of Type-2 isomorphic $C_{32}(R)$ 
\\
{\em Abstract:} Obtain all 384 pairs of Type-2 isomorphic circulant graphs of order 32.\\
\\
Part 4: 960 triples of Type-2 isomorphic circulant graphs $C_{54}(R)$
\\
{\em Abstract:}  Obtain all 960 triples of Type-2 isomorphic circulant graphs of order 54 and each triple is of Type-2 isomorphic w.r.t. $m$ = 3.\\
\\
Part 5: Type-2 isomorphic circulant graphs of orders 48, 81, 96
\\
{\em Abstract:} Find that for $n$ = 48,96, total number of pairs of isomorphic circulant graphs of Type-2 w.r.t. $m$ = 2 of the forms  $C_{n}(r_1,r_2,r_3)$ and $C_{n}(s_1,s_2,s_3)$ are 18 and 72, respectively, and total number of triples of isomorphic circulant graphs of Type-2 w.r.t. $m$ = 3 of the form $C_{81}(x_1,x_2,x_3)$, $C_{81}(y_1,y_2,y_3)$ and $C_{81}(z_1,z_2,z_3)$ are 27.\\
\\
Part 6: Abelian groups $(T2_{n,m}(C_n(R)), \circ)$ and $(V_{n,m}(C_n(R)), \circ)$
\\
{\em Abstract:}  Define $V_{n,m}(C_n(R))$ and Type-2 set $T2_{n,m}(C_n(R))$ of $C_n(R)$ and present their properties. Prove that $(V_{n,m}(C_n(R)), \circ)$ is an Abelian group and $(T2_{n,m}(C_n(R)), \circ)$ is a subgroup of $(V_{n,m}(C_n(R)), \circ)$ where $T2_{n,m}(C_n(R))$ =  $\{C_n(R)\}$ $\cup$ $\{C_n(S):$ $C_n(S)$ is Typ-2 isomorphic to $C_n(R)$ w.r.t. $m \}$ and $(T2_{n,m}(C_n(R)), \circ)$ is the Type-2 group of $C_n(R)$ w.r.t. $m$. Present many examples of Type-1 and Type-2 groups where $T1_{n}(C_n(R))$ = $\{C_n(xR): x\in\varphi_{n}\}$ is the Type-1 set of $C_n(R)$ and $(T1_{n}(C_n(R)), \circ')$ is its Type-1 group.\\
\\
Part 7: $C_{432}(R)$, $C_{6750}(S)$ - each has 2 types of Type-2 isomorphic circulant graphs
\\
{\em Abstract:} Obtain two families of circulant graphs, (i) family of $C_{432}(R)$, each has isomorphic circulant graphs of Type-2 w.r.t. $m$ = 2,3; and (ii) family of $C_{6750}(S)$, each has isomorphic circulant graphs of Type-2 w.r.t. $m$ = 3,5.\\
\\
Part 8: Isomorphism series, digraph and graph of $C_n(R)$
\\
{\em Abstract:} Define {\em isomorphic set}, {\em isomorphism series}, {\em isomorphism digraph} $\mathcal{D}$ or {\em isomorphism diagram} and {\em isomorphism graph} $\mathcal{G}$ of circulant graphs and obtain these  corresponding to $C_{16}(R)$, $C_{27}(S)$ and $C_{54}(1,3,17,19)$ and present the isomorphism digraph and graph of $C_{432}(T)$ which has isomorphic circulant graphs of Type-2 w.r.t. $m$ = 2,3 and of $C_{6750}(U)$ which has isomorphic circulant graphs of Type-2 w.r.t. $m$ = 3,5 where $T$ = $\{16, 27, 48, 54, 128, 160, 189\}$ and $U$ = $\{135, 243, 250, 750, 1107, 1593$, $2000, 2457, 2500, 2943\}$. Also define {\em diameter of isomorphic set} of $C_n(R)$ and {\em isomorphic distance} between any two circulant graphs. We find that the order of isomorphic sets of $C_{432}(T)$ and $C_{6750}(U)$ are 36 and 450 and their diameters are 3 and 3, respectively. \\
\\
Part 9: Computer programs to show Type-1 $\&$ -2 isomorphic circulant graphs
\\
{\em Abstract:} Present a $C^{++}$ program and also a VB program POLY415.EXE which is used to show how Type-1 and Type-2 isomorphisms of a circulant graph take place as well as for checking and finding Type-1 and Type-2 circulant graphs of a given order. \\
\\
Part 10: Type-2 isomorphic  $C_{np^3}(R)$ w.r.t. $m$ = $p$ and related groups
\\
{\em Abstract:} Families of Type-2 isomorphic circulant graphs $C_{np^3}(R)$ w.r.t. $m$ = $p$, and related Abelian groups are obtained where $p$ is a prime number and $n\in\mathbb{N}$. It is proved that for $i$ = 1 to $p$, circulant graphs $C_{np^3}(R^{np^3,x+yp}_i)$ are isomorphic of Type-2 w.r.t. $m$ = $p$ and they form Abelian group $(T2_{np^3,p}(C_{np^3}(R^{np^3,x+yp}_i)), \circ)$ where $T2_{np^3,p}(C_{np^3}(R^{np^3,x+yp}_i))$ = $\{\theta_{np^3,p,jn}(C_{np^3}(R^{np^3,x+yp}_i))$ = $C_{np^3}(R^{np^3,x+yp}_{i+j}) :$ $j$ = $0,1,...,p-1$ and $i+j$ in $C_{np^3}(R^{np^3,x+yp}_{i+j})$ is calculated under addition modulo $p \}$, $1 \leq x \leq p-1$, $0 \leq y \leq np - 1$, $1 \leq x+yp \leq np^2-1$, $y\in\mathbb{N}_0$, $p,np^3-p\in R^{np^3,x+yp}_i$ and $i,n,x\in\mathbb{N}$. A list of $T2_{np^3,p}(C_{np^3}(R^{np^3,x+yp}_i))$, each containing $p$ isomorphic circulant graphs $C_{np^3}(R^{np^3,x+yp}_i)$ of Type-2 w.r.t. $m$ = $p$, for $p$ = 3,5,7, $n$ = 1,2 and $y$ = 0 is given in the Annexure and more such families of Type-2 isomorphic circulant graphs are presented in \cite{v24}.

\subsection{Preliminaries}

We present here a few definitions and results that are required in the subsequent sections.  

If a graph $G$ is circulant, then its adjacency matrix $A(G)$ is circulant. It follows that if the first row of the adjacency matrix of a circulant graph is $[a_1,a_2,\dots, a_n]$, then $a_1$ = $0$ and $a_i$ = $a_{n-i+2}$, $2 \leq i \leq n$ \cite{da79}. Through-out this paper, for a set $R$ = $\{ r_1, r_2, \dots, r_k \}$, $C_n(R)$ denotes circulant graph $C_n(r_1,r_2,\dots,r_k)$ where $1 \leq r_1 < r_2 < . . . < r_k \leq [\frac{n}{2}]$. And we consider only connected circulant graphs of finite order, $V(C_n(R))$ = $\{v_0, v_1, v_2, . . . , v_{n-1}\}$ with $v_i$ adjacent to $v_{i+r}$ for each $r \in R$, subscript addition taken modulo $n$, $K_n$ = $C_n(1,2,...,n-1)$ and all cycles have length at least $3$, unless otherwise specified, $0 \leq i \leq n-1$. However when $\frac{n}{2} \in R$, edge $v_iv_{i+\frac{n}{2}}$ is taken as a single edge for considering the degree of the vertex $v_i$ or $v_{i+\frac{n}{2}}$ and as a double edge while counting the number of edges or cycles in $C_n(R)$, $0 \leq i \leq n-1$. 

We generally write $C_n$ for $C_n(1)$. We will often assume that the vertices of circulant graph $C_n(R)$, with-out further comment, are the corners of a regular $n$-polygon, labeled clockwise. In $C_n(R)$ when $r\in R$, it is understood that $r\in [1, [\frac{n}{2}]]$. Vertex $v_i$ in each figure is considered with label $i$, $v_i\in V(C_n(R))$ and $i\in\mathbb{Z}_n$. Circulant graphs $C_{16}(1,2,4,7)$ and $C_{16}(2,3,4,5)$ are shown in Figures 1 and 2. 
\begin{figure}[ht]
	\centerline{\includegraphics[width=6in]{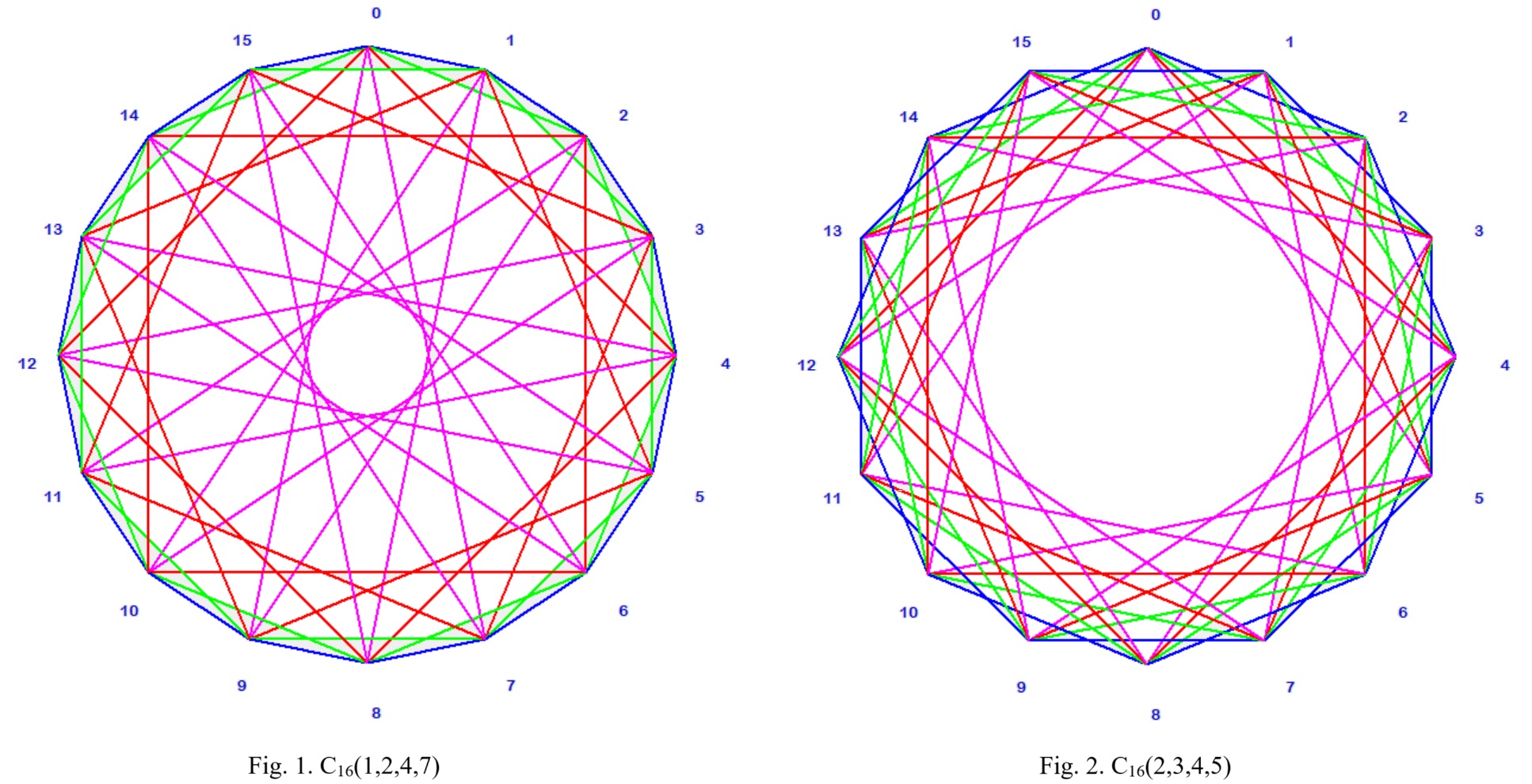}}
\end{figure}

Many symbols are used in this study and we present important symbols in the following Table 1 to make it easier to the readers. 

\begin{table}\label{t1}
	\caption{List of important symbols used in this paper}
	\begin{center}
		\scalebox{.9}{
	\begin{tabular}{||c||*{4}{c||}}\hline \hline 
				S. No. &  Symbol & Meaning 
				\\ \hline \hline 
				& & \\
				
1. & $C_n(1)$, $C_n(1,2,...,[n/2])$ & $C_n(1)$ = $C_n$, ~ ~ $C_n(1,2,...,[n/2])$ = $K_n$. \\
				
2. & $C_n(R)$ & Circulant graph of order $n$ with jump sizes \\
		    & & $r_1,r_2,...,r_k$, $R = \{r_1,r_2,...,r_k\} \subseteq \mathbb{Z}_{\frac{n}{2}}$ \\
			& & and $1 \leq r_1 < r_2 < ... < r_k \leq [n/2]$. \\
				
3. & $\mathbb{Z}_n$, $\mathbb{Z}^*_n$ & $\mathbb{Z}_n$ =  
           $\{0,1,2,...,n-1\}$, ~ ~ $\mathbb{Z}^*_n$ = $\mathbb{Z}_n \setminus \{0\}$.\\
				
4. & $\varphi_n$ & $\varphi_n$ = $\{ x \in \mathbb{Z}_n : \gcd(n, x) = 1 \}$.\\
				
5. & $\varphi_{n, x}:$ $S$ $\rightarrow$ $\mathbb{Z}_n$ &  $\ni$ $\varphi_{n,x}(s)$ = $xs$, $\forall$ $s \in S$, $S \neq \emptyset$, $S \subseteq \mathbb{Z}_n$ and $x \in \varphi_{n}$.\\ 

6. & $\varphi_{n,x}(C_n(R))$ = $Ad_{n,x}(C_n(R))$ &  = Adam's isomorphic circulant graph of $C_n(R)$ w.r.t. $x$ \\ 
			& \hfill = $T1_{n,x}(C_n(R))$  &  = $C_n(\varphi_{n,x}(R))$ =  $C_n(xR)$ where $x\in\varphi_n$ and $\varphi_{n,x}(R)$ in  \\
			& & \hfill $C_n(\varphi_{n,x}(R))$ is calculated  under reflexive modulo $n$. \\
				
7. & $Ad_n$ = $T1_n$, ~ $\varphi_{n,x} \circ \varphi_{n,y}$ & $Ad_n$ = $\{\varphi_{n,x}/~ x\in\varphi_n\}$, ~ $\varphi_{n,x} \circ \varphi_{n,y}$ = $\varphi_{n,xy}$, $x,y\in\varphi_n$. \\

8.  & $Ad_n(C_n(R))$ = $T1_n(C_n(R))$ & $= \{\varphi_{n,x}(C_n(R)) = C_n(xR)/ x\in\varphi_n\}$. \\

9. & $(\varphi_{n,x} \circ \varphi_{n,y})(C_n(R))$ & = $\varphi_{n,xy}(C_n(R))$ = $C_n((xy)R)$ \\	
          & & \hfill = $\varphi_{n,x}(C_n(R)) \circ \varphi_{n,y}(C_n(R))$, $x,y\in\varphi_n$. \\	

10. & $C_n(xR) \circ C_n(yR)$ & = $C_n((xy)R)$ = $\varphi_{n,xy}(C_n(R))$, $x,y\in\varphi_n$. \\	

11. & $(Ad_n(C_n(R)),~\circ)$ & = $(T1_n(C_n(R)),~\circ)$ = Adam's or Type-1 group of $C_n(R)$. \\
        
12. & $C_n(R) \cong_{T1_{n,x}} C_{n}(S)$ & $T1_{n,x}(C_n(R))$ (= $\varphi_{n,x}(C_n(R))$ = $C_{n}(xR)$) = $C_n(S)$, $x\in\varphi_n$. \\ 

13. & $C_n(R) \cong_{T1_{n}} C_{n}(S)$ & $C_n(R)$ and $C_n(S)$ are Adam's or Type-1 isomorphic. \\ 
         & & \\

14. & $\theta_{n,m,t}:$ $\mathbb{Z}_n$ $\rightarrow$ $\mathbb{Z}_n$ & $\ni$   $\theta_{n,m,t}(x) =  x+jtm$ where $x = qm+j$,  $0 \leq j \leq m-1$,  \\
		& & $m > 1$,  $m$ is a divisor of $\gcd(n,r)$, $m^3$ is a divisor of $n$,   \\
		& & $r\in\mathbb{Z}^*_n$, $m,x\in\mathbb{Z}_n$ and  $0 \leq q,t \leq \frac{n}{m}-1$. \\
				
15. & $\theta_{n,m,t} \circ \theta_{n,m,t'}$ = $\theta_{n,m,t+t'}$  &  where $0 \leq t,t' \leq \frac{n}{m}-1$, $m > 1$, $m$ is a divisor of  $\gcd(n,r)$,  $m^3$ is \\
		& &  a divisor of  $n$  and $t+t'$ is calculated under arithmetic $mod~\frac{n}{m}$. \\
				
16. & $\theta_{n,m,t}:$ $V(C_n(R))$ $\rightarrow$ $V(K_n)$  & $\ni$ $\theta_{n,m,t}(v_x) = u_{x+jtm}$, $\forall$ $v_x\in V(C_n(R))$, \\
		& & $\theta_{n,m,t}((v_x, v_y))$ = $ (\theta_{n,m,t}(v_x), \theta_{n,r,t}(v_y))$ and \\ 
		&  & $\theta_{n,m,t}(C_n(R)) = C_n(\theta_{n,m,t}(R))$ where $\theta_{n,m,t}(R)$ in \\
		& &   $C_n(\theta_{n,m,t}(R))$  is calculated under reflexive modulo $n$, \\
		& &  $x = qm+j$,  $m > 1$, $0 \leq j \leq m-1$, $0 \leq q,t \leq \frac{n}{m}-1$,  \\
		& &  $m$ is a divisor of  $\gcd(n,r)$, $m^3$ is a divisor of  $n$, $r\in R$,  \\ 
		& &  $V(C_n(R))$ = $\{v_0, v_1, v_2, . . . , v_{n-1}\}$, \\ 
		& & $V(K_n)$ =  $\{u_0, u_1, u_2, . . . , u_{n-1}\}$ and $(v_x, v_y)\in E(C_n(R))$. \\
       & & \\
				
17. & $V_{n, m}$  & = $\{\theta_{n,m,t}/ t = 0,1,2,...,\frac{n}{m}-1\}$, $m > 1$, $m$ is a divisor of $\gcd(n,r)$, \\

& &   $m^3$ is a divisor of $n$ and $r\in\mathbb{Z}^*_n$. \\
				
18. & $V_{n, m}(s)$  & $= \{\theta_{n,m,t}(s)/ t = 0,1,2,...,\frac{n}{m}-1\}$, $r\in\mathbb{Z}^*_n$, $s\in\mathbb{Z}_n$, $m > 1$,   \\
		& & \hfill $m$ is a divisor of $\gcd(n,r)$ and $m^3$ is a divisor of $n$.  \\
				
19. & $V_{n, m}(C_n(R))$ & $= \{\theta_{n,m,t}(C_n(R))/ t = 0,1,2,...,\frac{n}{m}-1\}$, $m > 1$, \\
		& & \hfill $m$ is a divisor of $\gcd(n,r)$, $m^3$ is a divisor of $n$ and $r\in R$. \\
     & & \\ 	
				
20. &  $C_n(R)$ and $C_n(S)$ are Type-2 &  if $\theta_{n,m,t}(C_n(R)) = C_n(S)$ for some $t$, $1 \leq t \leq \frac{n}{m}-1$, $m > 1$, \\
		& isomorphic w.r.t. $m$ &  $m$ is a divisor of $\gcd(n,r)$, $m^3$ is a divisor of $n$, $r\in R,S$ and \\	
		&  &   $C_n(S) \neq C_n(xR)$, $\forall$ $x\in\varphi_n$. \\
				
21. & $T2_{n, m}(C_n(R))$ = & $\{C_n(R)\} \cup \{C_n(S): C_n(S)$ and $C_n(R)$ are Type-2 isomorphic w.r.t. $m$, \\
	& &    $m > 1$, $m$ is a divisor of $\gcd(n,r)$, $m^3$ is a divisor of $n$ and $r\in R,S\}$.  \\ 
				
22. & $(T2_{n, m}(C_n(R)),~\circ)$ & = Type-2 group of $C_n(R)$ w.r.t. $m$, $m > 1$, $m$ is a divisor of $\gcd(n,r)$, \\
		& & \hfill $m^3$ is a divisor of $n$ and $r\in R$.\\
     & & \\ 					

23. & $T2_{n, m,t}(C_{n}(R))$ is same as & $\theta_{n,m,t}(C_n(R))$ = $C_n(S)$ for some $S$ and $C_n(S)\in T2_{n,m}(C_n(R))$,  $m > 1$, \\ 
& &  $m$ is a  divisor of  $\gcd(n,r)$, $m^3$ is a divisor of $n$, $r\in R,S$, $1 \leq t \leq \frac{n}{m}-1$.  \\
 
24. & $C_n(R) \cong_{T2_{n,m,t}} C_{n}(S)$ &  $\theta_{n,m,t}(C_n(R))$ = $C_n(S)$ and $C_n(S)\in T2_{n,m}(C_n(R))$, $1 \leq t \leq \frac{n}{m}-1$, \\
   & &  $m > 1$, $m$ is a  divisor of  $\gcd(n,r)$, $m^3$ is a divisor of $n$ and $r\in R,S$.  \\ 

25. & $C_n(R)$ $\cong_{T2_{n,m}}$ $C_n(S)$ &   $C_n(S) \in T2_{n,m}(C_{n}(R))$ if and only if $T2_{n,m}(C_{n}(R)) $ = $T2_{n,m}(C_{n}(S))$. \\
     & & \\ 	 	\hline \hline 
	\end{tabular}}
	\end{center}
\end{table}

\begin{definition} {\rm\cite{v96}}  \label{d1}  \quad Let $n$ and $s$ be positive integers with $n \geq 4$ and $s < \frac{n}{2}$. Then, clearly, $C_n(s)$ consists of a collection of cycles $(v_0 v_s v_{2s} \dots v_0)$, $(v_1 v_{1+s} v_{1+2s} \dots v_1)$, $\dots$, $(v_{s-1} v_{2s-1} v_{3s-1} \dots v_{s-1})$. If $d$ = $\gcd(n, s)$, then there are $d$ such disjoint cycles, each of length $n/d$. We say that each of these cycles is of {\em period} $s$, {\em length} $n/d$ and {\em rotation} $s/d$.
\end{definition}

If $s$ = $n/2$, then clearly $C_n(s)$ is just a 1-factor. Since $C_n(R)$ is just the union of the cycles $C_n(s)$ for $s \in R$, we have a decomposition of $C_n(R)$.

\begin{theorem}{\rm\cite{v96}  \quad \label{t2} Let $r\in R$.  In $C_n(R)$, the length of a cycle of period $r$ is $n/\gcd(n,r)$ and the number of disjoint periodic cycles of period $r$ is $\gcd(n,r)$, $r \in R$. \hfill $\Box$}
\end{theorem}

\begin{cor}{\rm \cite{v96} \quad \label{c3} In $C_n(R)$, length of a cycle of period $r$ is $n$ if and only if $\gcd(n,r)$ = $1,$ $r \in R$. \hfill $\Box$}
\end{cor}

\begin{rem}{\rm\cite{v96}\quad \label{r4} Let $|R|$ = $k$. Then, the circulant graph $C_n(R)$ for a set $R$ = $\{ r_1, r_2, \dots, r_k \}$ is $(2k-1)$-regular if $\frac{n}{2} \in R$ and $2k$-regular otherwise. }
\end{rem} 

The following Lemmas are useful to obtain one-to-one mappings.

\begin{lemma}{\rm\cite{v13}\quad \label{l5} Let $A$ and $B$ be two non-empty sets and $f:$ $A \rightarrow B$ be a mapping. Then, $f$ is one-to-one if and only if $f/_{A'}$ is one-to-one for every non-empty subset $A'$ of $A$.\hfill $\Box$}
\end{lemma}

\begin{lemma}{\rm\cite{v13}\quad \label{l6} Let $A$ and $B$ be non-empty sets and $A_1,A_2,\dots,A_k$ be a partition of $A$ (each $A_i$ being non-empty, $i$ = $1,2,\dots,k$). Let $f:$ $A$ $\rightarrow$ $B$ be a mapping. Then $f$ is one-to-one if and only if $f/_{A_i}$ is one-to-one for every $i,$ $i$ = $1,2,\dots,k.$ \hfill $\Box$}
\end{lemma}

\begin{theorem}{\rm \cite{v17} \quad \label{t7} If $C_n(R)$ and $C_n(S)$ are isomorphic, then there exists a bijection $f$ from $R \to S$ such that $\gcd(n, r)$ = $\gcd(n, f(r))$ for all $r\in R$.  }
\end{theorem}
\begin{proof} The proof is by induction on the order of $R$. 
\end{proof}

\section{Adam's or Type-1 Isomorphism and Type-1 group of circulant graphs}

In this section, we present our study on Adam's isomorphism or Type-1 isomorphism of circulant graph $C_n(R)$ and define Type-1 group of $C_n(R)$. 

\begin{lemma}{\rm \cite{v17} \quad \label{l8} Let $S \subseteq \mathbb{Z}_n$, $S \neq \emptyset$ and $x \in \mathbb{Z}_n.$ Define a mapping $\varphi_{n, x}:$ $S$ $\rightarrow$ $\mathbb{Z}_n$ $\ni$ $\varphi_{n, x}(s)$ = $xs$, $\forall$ $s \in S,$ under multiplication modulo $n$. Then $\varphi_{n, x}$ is bijective if and only if $S = \mathbb{Z}_n$ and $\gcd(n, x) = 1$. }
\end{lemma}
\begin{proof}\quad Here, we give a proof using a property of periodic cycles in a circulant graph $C_n(R)$ with jump size $x$. Let $S = \mathbb{Z}_n$. Then the numbers $0,$ $x,$ $2x,$ $3x,$ $\dots,$ $(n-1)x$, under multiplication modulo $n$, are all distinct if and only if $\gcd(n, x)$ = 1 since the cycle of period $x$ in $C_n(R),$ $x \in R$, is of length $n$ if and only if $\gcd(n, x)$ = 1, using Corollary \ref{c3}.
	
Conversely, if $S \neq \mathbb{Z}_n$, then $S$ is a proper subset of $\mathbb{Z}_n$ and so $\varphi_{n,x}$ is not a bijective mapping. Hence the result follows. 
\end{proof} 

Hereafter, unless otherwise it is mentioned in other way, we consider $\varphi_{n,x}$ with $\gcd(n,x)$ = 1 only. Let $\varphi_n$ = $\{ x \in \mathbb{Z}_n : \gcd(n, x) = 1 \}$. Clearly,  $( \varphi_n,~\circ)$ is an abelian group under the binary operation $\lq\circ\rq$,  multiplication modulo $n$.

\begin{definition}{\rm\cite{ad67}} \quad \label{d9} For $R =$ $\{r_1$, $r_2$, $\dots$, $r_k\}$ and $S$ = $\{s_1$, $s_2$, $\dots$, $s_k\}$, circulant graphs $C_n(R)$ and $C_n(S)$ are {\it Adam's isomorphic} if there exists a positive integer $x$ $\ni$ $\gcd(n, x)$ = 1 and $S$ = $\{xr_1$, $xr_2$, $\dots$, $xr_k\}_n^*$ where $<r_i>_n^*$, the {\it reflexive modular reduction} of a sequence $< r_i >$, is the sequence obtained by reducing each $r_i$ under modulo $n$ to yield $r_i'$ and then replacing all resulting terms $r_i'$ which are larger than $\frac{n}{2}$ by $n-r_i'.$  
\end{definition}

The following definition based on Lemma \ref{l8} is same as the above.

\begin{definition} \label{d10} Let $n\in\mathbb{N}$. For a subset $S$ of $\mathbb{Z}_n$ one can define the circulant graph $C_n(S)$ as the graph with vertex set $\mathbb{Z}_n$ and edge set $\{\{x, y\} | x,y\in \mathbb{Z}_n,$ $x-y\in S\}$. For an integer $k$ coprime to $n$, there is a very natural isomorphism between circulants $C_n(S)$ and $C_n(kS)$  namely the map $\varphi_k : x \mapsto kx$ (observe that $\varphi_k$ is an automorphism of the group $\mathbb{Z}_n$) where $kS$ = $\{ks : s\in S\}$. In this case, we say that $C_n(S)$ and $C_n(kS)$ are \emph{Adam's isomorphic} or \emph{Type-1 isomorphic} and we call the isomorphism as \emph{Type-1 isomorphism}.  
\end{definition}

Circulant graphs $C_{16}(1,2,4,7)$ and $C_{16}(3,4,5,6)$ are Adam's isomorphic since $C_{16}(3(1,2,4,7))$ = $C_{16}(3(1,2,4,7, 9,12,14,15))$ = $C_{16}(3,6,12,5, 11,4,10,13))$ = $C_{16}(3,4,5,6, 10,11,12,13))$ = $C_{16}(3,4,5,6)$. Circulant graphs $C_{16}(3,4,5,6)$ and $C_{16}(3(1,2,4,7))$ = $C_{16}(3,4,5,6)$ are show in Figures 3 and 4.

\begin{figure}[ht]
	\centerline{\includegraphics[width=6in]{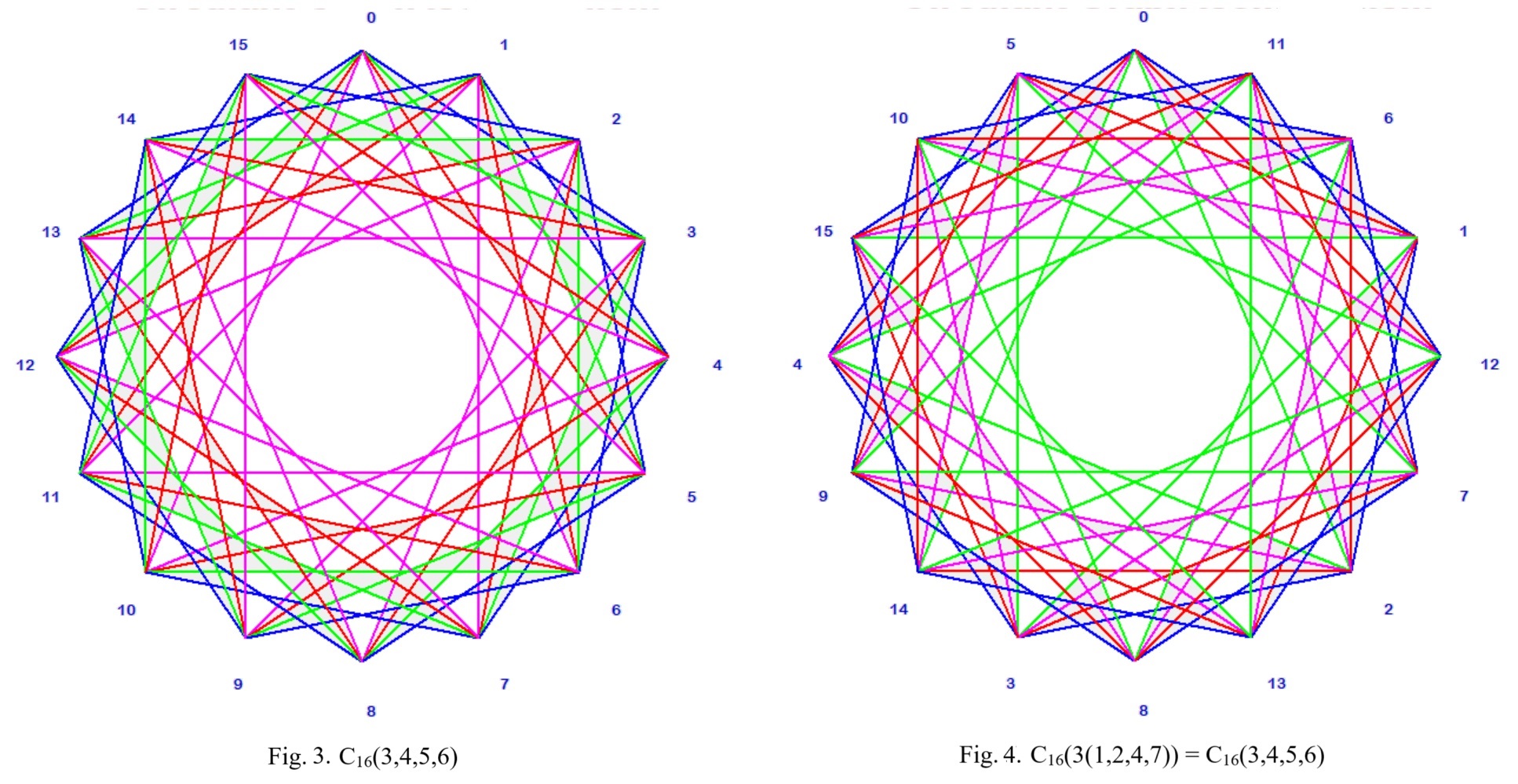}}
\end{figure}

\noindent
{\bf How Adam's Isomorphism acts on circulant graphs \cite{v96, v17}}

Here, we explain how Adam's isomorphism or Type-1 isomorphism acts on circulant graphs $C_n(R)$.

Given a circulant graph $C_n(R)$ for a set $R$ = $\{r_1, r_2, . . . , r_k\}$, corresponding to each $x\in\varphi_n$, we get its {\it Adam's isomorphic} circulant graph $C_n(xR)$. Under this Adam's isomorphism, each $v_i$ of $C_n(R)$ is taken to the position of $v_{ix}$, using subscript arithmetic modulo $n,$ $i$ = $0,1,2,\dots,n-1$. Thus, the action of Adam's isomorphism on $C_n(R)$ is that $v_0 \rightarrow v_0,$ $v_1 \rightarrow v_x$, $v_2 \rightarrow v_{2x}$, $\dots$, $v_{n-1} \rightarrow v_{(n-1)x}$ and in particular, each periodic cycle $(v_0 v_r v_{2r} \dots v_{n-r})$ of $C_n(R)$, $r \in R$, is taken to  periodic cycle ($v_0$ $v_{xr}$ $v_{2xr}$ $\dots$ $v_{(n-r)x}$) in the transformed graph $C_n(xR)$, using subscript arithmetic modulo $n$. That is, under Adam's isomorphism, each vertex (of the circulant graph $C_n(R))$, taken periodically, moves uniformly and still the length and the number of independent cycles of a particular jump size are unchanged in the transformed graph $C_n(xR)$ since $\gcd(n, xr)$ = $\gcd(n, r)$ when $\gcd(n, x)$ = 1. If $R$ = $R_1 \cup R_2$, then $C_n(xR)$ = $C_n(xR_1 \cup xR_2)$. Thus, under Adam's isomorphism all circulant subgraphs $C_n(R_1)$ and $C_n(R_2)$ of $C_n(R_1\cup R_2)$ are transformed to $C_n(xR_1)$ and $C_n(xR_2)$ which are Adam's isomorphic to $C_n(R_1)$ and $C_n(R_2)$, respectively. This is true for any number of circulant subgraphs of $C_n(R)$. That is if $q\in \mathbb{N}$ and $R$ = $R_1$ $\cup$ $R_2$ $\cup$ $\dots$ $\cup$ $R_q$, then $C_n(xR)$ = $C_n(xR_1$ $\cup$ $xR_2$ $\cup$ $\dots$ $\cup$ $xR_q)$ and under Adam's isomorphism circulant subgraphs $C_n(R_1)$, $C_n(R_2)$, $\dots$, $C_n(R_q)$ of $C_n(R_1$ $\cup$ $R_2$ $\cup$ $\dots$ $\cup$ $R_q)$ are transformed to $C_n(xR_1)$, $C_n(xR_2)$, $\dots$, $C_n(xR_q)$ which are Adam’s isomorphic to $C_n(R_1)$, $C_n(R_2)$, $\dots$, $C_n(R_q),$ respectively. 

\begin{definition} \label{d11} Let $Ad_n = \{\varphi_{n,x}: x\in \varphi_n\}$, $Ad_n(S) = \{\varphi_{n,x}(S): x\in \varphi_n\}$ = $\{xS: x\in \varphi_n\}$, $Ad_{n,x}(C_n(R))$ = $T1_{n,x}(C_n(R))$ = $\varphi_{n,x}(C_n(R))$ = $C_n(\varphi_{n,x}(R))$ = $C_n(xR)$, $x\in \varphi_n$ and $Ad_n(C_n(R)) = T1_n(C_n(R)) = \{\varphi_{n,y}(C_n(R)) = C_n(yR): y\in \varphi_n\}$ for sets $R,S \subseteq \mathbb{Z}_n$ where $\varphi_{n,x}(R)$ in $C_n(\varphi_{n,x}(R))$ is calculated under the reflexive modulo $n$. Define $'\circ'$ in $Ad_n(C_n(R))$ such that $\varphi_{n,x} \circ \varphi_{n,y}$ = $\varphi_{n,xy}$, $C_n(xR) \circ C_n(yR)$ = $C_n((xy)R)$ and $\varphi_{n,x}(C_n(R)) \circ \varphi_{n,y}(C_n(R))$ = $\varphi_{n,xy}(C_n(R))$, $\forall$ $x,y\in\varphi_n$.

Here, $(\varphi_{n,x} \circ \varphi_{n,y})(C_n(R))$ = $\varphi_{n,xy}(C_n(R))$ = $C_n((xy)R)$ = $C_n(xR) \circ C_n(yR)$  = $\varphi_{n,x}(C_n(R)) \circ \varphi_{n,y}(C_n(R))$, $\forall$ $x,y \in \varphi_n$, under arithmetic modulo $n$. 
\end{definition}
  	
 Clearly, $Ad_n(C_n(R))$ is the set of all circulant graphs which are Adam's isomorphic to $C_n(R)$ and we call it as the {\em Adam's set} or {\em Type-1 set} of $C_n(R)$. Also, $(Ad_n(C_n(R)), \circ )$ = $(T1_n(C_n(R)), \circ )$ is an Abelian group and we call it as the {\em Adam's group} or {\em Type-1 group} of $C_n(R)$ under $'\circ'$. 	

{\rm Moreover, $C_n(S)\in Ad_n(C_n(R))$ implies, $\exists$ $x\in\varphi_n$ $\ni$ $C_n(S)$ = $\varphi_{n,x}(C_n(R))$ = $C_n(xR)$. Corresponding to $x\in \varphi_n$, $\exists$ $x^*\in \varphi_n$ $\ni$ $xx^*$ = $1\in\varphi_n$ and $\varphi_{n,x^*}(C_n(S))$ = $\varphi_{n,x^*}(C_n(xR))$ = $C_n(x^*(xR))$ = $C_n((x^*x)R)$ = $C_n(R)$. Thus, $C_n(R)$ = $\varphi_{n,x^*}(C_n(S))$ which implies, $C_n(R)\in Ad_n(C_n(S))$, $x^*\in \varphi_n$. This also implies, $Ad_n(C_n(R))$ = $Ad_n(C_n(S))$ whenever $C_n(S)\in Ad_n(C_n(R))$ or $C_n(R)\in Ad_n(C_n(S))$. Thus, we get the following result corresponding to $Ad_n(C_n(R))$. }
 
\begin{theorem} \quad \label{t12} {\rm Let $Ad_n(C_n(R))$ = $\{\varphi_{n,x}(C_n(R)) = C_n(xR): x\in\varphi_n \}$. Then, $C_n(S)\in Ad_n(C_n(R))$ if and only if $Ad_n(C_n(R))$ = $Ad_n(C_n(S))$ if and only if $(Ad_n(C_n(R)), \circ )$ = $(Ad_n(C_n(S)), \circ )$ if and only if $C_n(R)\in Ad_n(C_n(S))$. \hfill $\Box$   }
\end{theorem}

Circulant graphs $C_{54}(1,17,18,19)$, $C_{54}(5,13,18,23)$ and $C_{54}(7,11,18,25)$ are Adam's isomorphic since $C_{54}(5(1,17,18,19))$ = $C_{54}(5,13,18,23)$ and  $C_{54}(7(1,17,18,19))$ = $C_{54}(7,11,18,25)$. $Ad_{54}(C_{54}(1$, $17,18,19))$ = $\{\varphi_{54,x}(C_{54}(1,17,18,19)): x\in \varphi_{54}\}$ = $\{\varphi_{54,x}(C_{54}(1,17,18,19)): x = 1,5,7,11,13,17,19$, $23,25,29,31,35,37,41,43,47,49,53\}$ = $\{C_{54}(1,17, 18,19),$ $C_{54}(5,13,18,23),$ $C_{54}(7,11,18,25)\}$ = $Ad_{54}($ $C_{54}(5,13,18,23))$ = $Ad_{54}(C_{54}(7,11,18,25))$. 

\medskip
\noindent
{\bf Notations.}  We introduce the following notations under Adam's isomorphism of circulant graphs.

$C_n(R) \cong_{Ad_{n,x}} C_{n}(S)$ or $C_n(R) \cong_{T1_{n,x}} C_{n}(S)$ when $Ad_{n,x}(C_n(R))$ (= $T1_{n,x}(C_n(R))$ = $C_{n}(xR)$) = $C_n(S)$, $x\in\varphi_n$ and $C_n(R) \cong_{Ad_{n}} C_{n}(S)$ or $C_n(R) \cong_{T1_{n}} C_{n}(S)$ when $C_n(R)$ and $C_n(S)$ are Adam's isomorphic or Type-1 isomorphic. 

When $T1_{n,x}(C_n(R))$ = $C_{n}(xR)$ = $C_n(S)$, $x\in\varphi_n$ and $k\in\mathbb{N}$, $k T1_{n,x}(C_n(R))$ = $k C_{n}(xR)$ = $k C_n(S)$ (= union of $k$ copies of $C_n(S)$) = $C_{kn}(kS)$ = $T1_{kn,x}(C_{kn}(kR))$. 

Thus, for any circulant graph $C_n(R)$, $k\in\mathbb{N}$ and $x\in\varphi_n$, we always consider $k T1_{n,x}(C_n(R))$ = $T1_{kn,x}(C_{kn}(kR))$.   

\section{Type-2 isomorphism and Type-2 isomorphic circulant graphs}

In 1970, Elspas and Turner \cite{eltu} rised a question on the type of isomorphism that exists on the two circulant graphs $C_{16}(1, 2, 7)$ and $C_{16}(2, 3, 5)$, see Figures 5 and 6, which are  isomorphic but not of Adam’s. In 1996, Vilfred \cite{v96} defined a new type of circulant graph isomorphism, different from Adam’s isomorphism, and studied under the heading ‘generalized circulant graph isomorphism’ and thereafter, we call the new type of isomorphism of circulant graphs as {\em Type-2 isomorphism} \cite{v17}-\cite{v24} and Adam’s isomorphism as {\em Type-1 isomorphism}. 
\begin{figure}[ht]
	\centerline{\includegraphics[width=5in]{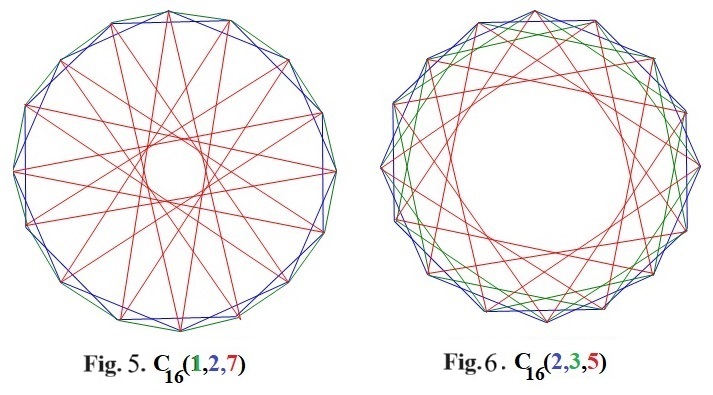}}
\end{figure}

While studying Type-2 isomorphism of $C_n(R)$ w.r.t. $m$ in \cite{v96}-\cite{v20}, the author considered $m$ = $\gcd(n, r) > 1$ and $r\in R$ and  proved that $C_{16}(1,3,7)$ and $C_{16}(2,3,5)$ are Type-2 isomorphic w.r.t. $m$ = 2. Vilfred and Wilson \cite{vw1} - \cite{vw3} obtained families of isomorphic circulant graphs of Type-2 w.r.t. $m$ = 3,5,7. In \cite{v25}, Vilfred modified the definition of Type-2 isomorphism of $C_n(R)$ w.r.t. $m$ by considering $m > 1$ is a divisor of $\gcd(n, r)$ and $r\in R$. In this paper, the definition is further modified to $m > 1$ and $m^3$ are divisors of $\gcd(n, r)$ and $n$, respectively, and $r\in R$. This modification simplifies our calculations while finding isomorphic circulant graphs of Type-2. In this section, we present our study on Type-2 isomorphism and Type-2 isomorphic circulant graphs, explain in details on how Type-2 isomorphism takes place and also present proof of $C_{16}(1, 2, 7)$ and $C_{16}(2, 3, 5)$ are Type-2 isomorphic w.r.t. $m$ = 2 under the modified definition. This study is an extension and improved results obtained in \cite{v96, v24, v20}. We start the section with a slightly modified lemma but with the same proof as given in \cite{v20} that is essential to define Type-2 isomorphism.

\begin{lemma}{\rm \quad \label{l13} Let $m >1$ be a divisor of $n$. Then, for each value of $t$, the mapping $\theta_{n,m,t}: \mathbb{Z}_n \rightarrow \mathbb{Z}_n$ defined by $\theta_{n,m,t}(x)$ = $x+jtm$, $x \in \mathbb{Z}_n$, is bijective under arithmetic modulo $n$ where $x = qm+j,$ $0 \leq j \leq m-1,$ $0 \leq q,t \leq \frac{n}{m}-1$ and $j,m,q,t \in \mathbb{Z}_n$. In particular, the result is true when $m > 1$ is a divisor of $\gcd(n, r)$ and $r \in \mathbb{Z}_n$. } 
\end{lemma}
\begin{proof}\quad From the definition of $\theta_{n,m,t}$, we get the following properties. 
	\begin{enumerate}
		\item[{\rm(i)}] $\theta_{n,m,t}(km) = km$ for every $k \in \mathbb{Z}_n$, $km \in \mathbb{Z}_n$ and $0 \leq t \leq \frac{n}{m}-1$.
		\item[{\rm(ii)}] For $0 \leq i,j \leq m-1$ and $0 \leq t \leq \frac{n}{m}-1$, $\theta_{n,m,t}(i)$ = $\theta_{n,m,t}(j)$ if and only if $i$ = $j$ if and only if $\theta_{n,m,t}(qm+i)$ = $\theta_{n,m,t}(qm+j),$ $\gcd(n,r)$ = $m > 1$, and $0 \leq qm \leq n-1$.
		\item[{\rm(iii)}]	For $0 \leq i \leq m-1$, $0 \leq km,qm \leq n-1$ and $0 \leq q,t \leq \frac{n}{m}-1$,  $\theta_{n,m,t}(km+i)$ = $\theta_{n,m,t}(qm+i)$ if and only if $k$ = $q.$ 
		
		From the above three properties, we get,
		
		\item[{\rm(iv)}]	For $0 \leq i,j \leq m-1$, $0 \leq km,qm \leq n-1$ and $0 \leq t \leq \frac{n}{m}-1$, $\theta_{n,m,t}(i+km)$ = $\theta_{n,m,t}(j+qm)$ if and only if $i$ = $j$ and $k$ = $q$.  
	\end{enumerate}
	This implies that for each value of $t$, the mapping $\theta_{n,m,t}$ is bijective when $m > 1$ is a divisor of $n$ and $0 \leq t \leq \frac{n}{m}-1$.  
	
	In particular, when $m > 1$ is a divisor of $\gcd(n, r)$ and $r \in \mathbb{Z}_n$ implies, $m > 1$ is a divisor of $n$ and hence the result is also true in this case. Hence, we get the result.   
\end{proof}

\begin{rem}\quad \label{r14} Related to Lemma \ref{l13}, when $n > r > m > 1$ and $m$ is a divisor of $\gcd(n, r)$, the mapping $\theta_{n,r,t}: \mathbb{Z}_n \rightarrow \mathbb{Z}_n$ defined by $\theta_{n,r,t}(x)$ = $x+jrt$ need not be bijective for every $t$ where  $x$ = $qr+j$, $0 \leq j \leq r-1$, $0 \leq q,t \leq \frac{n}{r}$, $0 \leq qrt \leq n-1$ and $j,m,q,t\in\mathbb{Z}_n$. Refer \cite{v24} for further details.
\end{rem}

\begin{theorem} \quad  \label{t15} {\rm Let $V(C_n(R))$ = $\{v_0,v_1,v_2,...,v_{n-1}\}$, $V(K_n)$ = $\{u_0,u_1,u_2,...,u_{n-1}\}$ and $r\in R$ such that $m > 1$ be a divisor of $\gcd(n, r)$ and $|R| \geq 3$. Then the mapping $\theta_{n,m,t} :$ $V(C_n(R)) \rightarrow V(K_n)$ defined by $\theta_{n,m,t}(v_x)$ = $u_{x+jtm}$, $\theta_{n,m,t}((v_x, v_{x+s}))$ = $(\theta_{n,m,t}(v_x), \theta_{n,m,t}(v_{x+s}))$ under subscript arithmetic modulo $n$ and $\theta_{n,m,t}(C_n(R))$ = $C_n(\theta_{n,m,t}(R))$ for every $x \in \mathbb{Z}_n$ and $s\in R$ for a set $R$ = $\{r_1, r_2, . . . , r_k\}$ is one-to-one, preserves adjacency and $\theta_{n,m,t}(C_n(R))$ $\cong$ $C_n(R)$ for $t$ = $0, 1, 2, . . . , \frac{n}{m}-1$ where $x$ = $qm + j$, $0 \leq j \leq m - 1$, $0 \leq q,t \leq \frac{n}{m}-1$, $j,q,t\in\mathbb{N}_0$ and $\theta_{n,m,t}(R)$ in $C_n(\theta_{n,m,t}(R))$ is calculated under reflexive modulo $n$. } 
\end{theorem}

\begin{proof} \quad Using Lemma \ref{l13}, the mapping $\theta_{n,m,t}$ is one-to-one. By the given condition $\theta_{n,m,t}((v_x, v_{x+s}))$ = $(\theta_{n,m,t}(v_x), \theta_{n,m,t}(v_{x+s}))$ for every $x \in \mathbb{Z}_n$ and $s\in R$, the mapping preserves adjacency, under subscript arithmetic modulo $n$ and hence $\theta_{n,m,t}(C_n(R))$ $\cong$ $C_n(R)$ for $t$ = $0, 1, 2, . . . , \frac{n}{m}-1$. Hence the result.
\end{proof}

Based on Lemma \ref{l13} and Theorem \ref{t15}, we modify the definition of Type-2 isomorphism of circulant graphs as follows.

\begin{definition} \label{d16} Let $V(K_n) = \{u_0,u_1,u_2,...,u_{n-1}\}$, $V(C_n(R)) = \{v_0,v_1,v_2,...,$ $v_{n-1}\}$, $r\in R$, $|R| \geq 3$, $m > 1$ and $m$ and $m^3$ be divisors of $\gcd(n, r)$ and $n$, respectively.  Define one-to-one mapping $\theta_{n,m,t} :$ $V(C_n(R)) \rightarrow V(K_n)$ such that $\theta_{n,m,t}(v_x)$ = $u_{x+jtm}$,  $\theta_{n,m,t}((v_x, v_{x+s}))$ = $(\theta_{n,m,t}(v_x), \theta_{n,m,t}(v_{x+s}))$ under subscript arithmetic modulo $n$ and $\theta_{n,m,t}(C_n(R))$ = $C_n(\theta_{n,m,t}(R))$ for every $x$ = $qm+j \in \mathbb{Z}_n$, $s\in R$, $0 \leq j \leq m-1$, $0 \leq q,t \leq \frac{n}{m} -1$ and $\theta_{n,m,t}(R)$ in $C_n(\theta_{n,m,t}(R))$ is calculated under the reflexive modulo $n$. And for a particular value of $t,$ if  $\theta_{n,m,t}(C_n(R))$ = $C_n(S)$ for some $S$  and  $S \neq yR$ for all $y\in \varphi_n$ under reflexive modulo $n$, then $C_n(R)$ and $C_n(S)$ are called {\em isomorphic circulant graphs of Type-2 w.r.t. $m$} and the isomorphism as {\em Type-2 isomorphism w.r.t. $m$.} 
	
When $C_n(R)$ and $C_n(S)$ are Type-2 isomorphic w.r.t. $m$, then we also say that $C_{kn}(kR)$  and $C_{kn}(kS)$ are Type-2 isomorphic w.r.t. $m$, $k\in\mathbb{N}$. Here, $k.C_n(T)$ = $C_{kn}(kT)$, $k\in\mathbb{N}$.
\end{definition}

\begin{definition}\quad \label{d17} Let $V_{n,m}(C_n(R))$ = $\{\theta_{n,m,t}(C_n(R)): t = 0,1,\dots,\frac{n}{m}-1 \}$ and  $T2_{n,m}(C_n(R))$ = $\{C_n(R)\}$ $\cup$ $\{C_n(S):$ $C_n(S)$ is Type-2 isomorphic to $C_n(R)$ w.r.t. $m$, $m > 1$ and  $m^3$ be divisors of $\gcd(n, r)$ and $n$, respectively, and $r\in R\}$. We call $T2_{n,m}(C_n(R))$ as {\em the Type-2 set of $C_n(R)$ w.r.t. $m$}.

That is, {\em the Type-2 set of $C_n(R)$ w.r.t. $m$}, denoted by $T2_{n,m}(C_n(R))$, is $\{C_n(R)\}$ $\cup$ $\{\theta_{n,m,t}(C_n(R)):$ $\theta_{n,m,t}(C_n(R))$ = $C_n(S)$ and $C_n(S)$ is Type-2 isomorphic to $C_n(R)$ w.r.t. $m$, $m > 1$ and $m^3$ are divisors of $\gcd(n, r)$ and $n$, respectively, $r\in R$ and $1 \leq t \leq \frac{n}{m}-1\}$ = $\{\theta_{n,m,0}(C_n(R))\}$ $\cup$ $\{\theta_{n,m,t}(C_n(R)):$ $\theta_{n,m,t}(C_n(R))$ = $C_n(S)$ and $C_n(S)$ is Type-2 isomorphic to $C_n(R)$ w.r.t. $m$, $m > 1$ and $m$ and $m^3$ are divisors of $\gcd(n, r)$ and $n$, respectively, $r\in R$ and $1 \leq t \leq \frac{n}{m}-1\}$.
\end{definition}

Clearly, $T2_{n,m}(C_n(R)) \subseteq V_{n,m}(C_n(R))$ and $T1_n(C_n(R))$ $\cap$ $T2_{n,m}(C_n(R))$ = $\{C_n(R)\}$ where $T1_n(C_n(R))$ = $Ad_n(C_n(R))$ and $V_{n,m}(C_n(R))$ = $\{\theta_{n,m,t}(C_n(R)):$ $t$ = $0,1,\dots,\frac{n}{m}-1\}$. Here, subsets $R$ and $S$ of $\mathbb{Z}_n$ are called {\em Type-2 isomorphic subsets of $\mathbb{Z}_n$ w.r.t. $m$.} That is subsets $R$ and $S$ of $\mathbb{Z}_n$ are called {\em Type-2 isomorphic subsets of $\mathbb{Z}_n$} w.r.t. $m$ if and only if $C_n(R)$ and $C_n(S)$ are Type-2 isomorphic w.r.t. $m$.

Subsets $S$ and $R$ of $\mathbb{Z}_n$ are called {\em (Type-1} or {\em Adam's) isomorphic} subsets of $\mathbb{Z}_n$ if and only if circulant graphs $C_n(S)$ and $C_n(R)$ are (Type-1 or Adam's) isomorphic.   

A circulant graph $C_n(R)$ is said to have {\it Cayley Isomorphism} (CI) property if whenever $C_n(S)$ is isomorphic to $C_n(R)$, they are Adam's isomorphic. CI-problem determines which graphs (or which groups) have the $CI$-property. Muzychuk \cite{mu04} completed the complete classification of cyclic $CI$-groups but investigation of circulant graphs without $CI$-property is not much done. 

Clearly, when either $n$ is a prime number or $\gcd(n, x)$ = 1, $\forall$ $x\in R$, then circulant graph $C_n(R)$ has CI-property. 

\medskip
\noindent
{\bf How Type-2 isomorphism acts on circulant graphs \cite{v17}}

It is interesting to know how Type-2 isomorphism is defined, how it acts on circulant graphs and under what conditions Type-2 isomorphic circulant graph(s) are obtained. These are presented in the following paragraphs. 

For convenience, let $\gcd(n, r_i)$ = $m_i > 1$ for at least one $i$, instead of considering $\gcd(n, r_i)$ be a multiple of $m_i > 1$, $r_i\in R$ and $1 \leq i \leq k$ so that $C_n(R)$ for a set $R$ = $\{ r_1, r_2, \dots, r_k \}$ can be decomposed into $m_i$ number of disjoint isomorphic connected circulant subgraphs, say, $\Gamma_0, \Gamma_1, \dots, \Gamma_{m_i-1}$, each having one periodic cycle of period $r_i$ and of length $n_i$ = $\frac{n}{m_i}$, using Theorem \ref{t2} where $v_j \in V(\Gamma_j)$ (and thereby $v_{j+qm_i} \in V(\Gamma_j)$, using subscript modulo $n$, $0 \leq q \leq \frac{n}{m_i}-1)$, $0 \leq j \leq m_i-1$. For some value of $t$, the transformed graph $\theta_{n,m_i,t}(C_n(R))$ may be in the form of $C_n(S)$ for some $S$, $0 \leq t \leq \frac{n}{m_i}-1$. Suppose $\theta_{n,m_i,t}(C_n(R))$ = $C_n(S)$ for some $S$ = $\{s_1, s_2, \dots, s_k\}$, then by the necessary condition for isomorphism between $C_n(R)$ and $C_n(S)$, for each $j$ there exists $q$ such that $\gcd(n, r_j)$ = $\gcd(n, s_q)$, $1 \leq r_j,s_q \leq n/2$ and $1 \leq j,q \leq k$, using Theorem \ref{t7}. This implies, there exists $b_q$ such that $s_q$ = $b_qm_j$ and $\gcd(n/m_j, b_q)$ = 1 where $m_j$ = $\gcd(n, r_j)$, $1 \leq j,q \leq k.$ Under the transformation of $\theta_{n,m_i,t}$, each $\Gamma_j$ simply rotates $tm_i$ positions (points) w.r.t. $\Gamma_{j-1}$ in the regular $n$-gon in the clockwise direction and thereby the vertices of each $V_j$ = $V(\Gamma_j)$ are taking the positions of the vertices of $V_j$ only, $0 \leq t \leq \frac{n}{m_i}-1$ and $0 \leq j \leq m_i-1$. By this transformation $\theta_{n,m_i,t}$ acting on $C_n(R)$, we may get a circulant graph of the form $C_n(S)$ for some $t$ and $S$ where $0 \leq t \leq \frac{n}{m_i}-1$ and $S \subseteq [1,n/2]$.  

While considering Type-2 isomorphism of $C_n(R),$ under the transformation $\theta_{n,r_i,t},$ sets of elements $V(\Gamma_j)$ move uniformly instead of individual elements (points) under Adam's isomorphism. Also, the sets of elements $\Gamma_j$ are isomorphic circulant subgraphs of the given circulant graph and the {\it jump sizes of these isomorphic circulant subgraphs do not change and all others move and may change} but still have to satisfy the necessary condition for circulant graph isomorphism (Theorem \ref{t7}) when the transformed graph is a circulant graph in its representation. The movements are so arranged that the resultant graph may be circulant in its representation.

\begin{theorem}{\rm \cite{vc13}\quad \label{t18} Let  $\Gamma_{m_i}$ = $\Gamma_0,$ $\Gamma_1,$ $\dots,$ $\Gamma_{m_i-1}$ be as defined above. Then, in $C_n(R),$ $v_x \in V(\Gamma_j)$ if and only if $v_{n-1-x} \in V(\Gamma_{m_i-1-j}),$ $0 \leq x \leq n-1$ and $0 \leq j \leq m_i-1.$}
\end{theorem}
\begin{proof}\quad For $0 \leq x \leq n-1$,  $v_x \in V(\Gamma_j)$  if and only if  $x$ = $qm_i+j$,  $0 \leq j \leq m_i-1$ and  $0 \leq q \leq n_i-1$ if and only if $n-1-x$ = $n-1-(qm_i+j)$ = $m_i-1-j+(n_i-1-q)m_i$ where $n$ = $n_im_i$, $0 \leq n_i-1-q \leq n_i-1$ and $0 \leq m_i-1-j \leq m_i-1$ if and only if $v_{n-1-x} \in V(\Gamma_{m_i-1-j})$, $n-x \equiv m_i-j~(mod~m_i).$  
\end{proof}

\begin{definition}{\rm \cite{vc13}}\quad \label{d19} The {\it symmetric equidistance condition} with respect to $v_i$ in $C_n(R)$ for a set $R = \{r_1,r_2,\dots,r_k\}$ is that $v_{i+j}$ is adjacent to $v_i$ if and only if $v_{n-j+i}$ is adjacent to $v_i$, using subscript arithmetic modulo $n,$ $0 \leq i,j \leq n-1$. 
\end{definition} 

The following theorem is a modified one of Theorem 5.2 in \cite{v24} that help to identify the graph $\theta_{n,m,t}(C_n(R))$ is again equal to $C_n(S)$ for some $S$ or not for a given circulant graph $C_n(R)$ and $t$ where $m > 1$, $m$ divides $\gcd(n,r)$, $m^3$ divides $n$, $r\in R$ and $0 \leq t \leq \frac{n}{\gcd(n,r)}$. 

\begin{theorem}{\rm \quad \label{t20} Let $R$ = $\{r_1,r_2,\dots,r_k\}  \subseteq [1, \frac{n}{2}]$, $m_i > 1$ be a divisor of $\gcd(n,r_i)$ for at least one $i$,  $1 \leq i \leq k$. Then, $\theta_{n,m_i,t}(C_n(R))$ = $C_n(S)$ if and only if $\theta_{n,m_i,t}(R \cup (n-R))$ = $S \cup (n-S)$ for some $t$ and $S$ $\ni$ $0 \leq t \leq \frac{n}{m_i}-1$ and $S \subseteq [1, \frac{n}{2}]$ if and only if $\theta_{n,m_i,t}(C_n(R))$ satisfies the symmetric equidistant condition w.r.t. $v_0$.}
\end{theorem}
\begin{proof}\quad Obviously, for a given $C_n(R)$, $\theta_{n,m_i,t}(C_n(R))$ = $C_n(S)$ if and only if $\theta_{n,m_i,t}(R \cup (n-R))$ = $S \cup (n-S)$ if and only if $\theta_{n,m_i,t}(C_n(R))$ satisfies the symmetric equidistant condition w.r.t. $v_0$ for some $t$ and $S$ $\ni$ $0 \leq t \leq \frac{n}{m_i}-1$ and $S \subseteq [1, \frac{n}{2}]$.  

Conversely, let $\theta_{n,m_i,t}(C_n(R))$ satisfy the symmetric equidistant condition w.r.t. $v_0$. To prove that $\theta_{n,m_i,t}(C_n(R))$ = $C_n(S)$ for some $S \subseteq [1, \frac{n}{2}]$. For convenience, let $m_i$ = $\gcd(n, r_i) > 1$, $1 \leq i \leq k$. In $C_n(R)$, $v_q v_{q+r} \in E(\Gamma_j)$ if and only if $v_q v_{q+r} v_{q+2r} \dots v_q \in \Gamma_j$, using subscript arithmetic modulo $n$ where $\Gamma_j$ is the periodic cycle of period $r$ and length $\frac{n}{\gcd(n, r)}$ in $C_n(R)$ and contains $v_j$, $0 \leq j \leq$ $m_i-1$, $r \in R$ and $0 \leq q \leq n-1$. By the transformation $\theta_{n,m_i,t}$, for some $t$, each $\Gamma_j$ simply rotates $tm_i$ positions (points) w.r.t. $\Gamma_{j-1}$ in the regular $n$-gon in the clockwise direction and thereby the vertices of each $V_j$ = $V(\Gamma_j)$ are taking the positions of the vertices of $V_j$ only, $0 \leq t \leq \frac{n}{m_i}-1$ and $0 \leq j \leq m_i-1$. Thus by this transformation the consecutive edges, except the last one, of any periodic cycle of $C_n(R)$ connecting different $\Gamma_j$'s (not belonging to $\Gamma_j$), starting from any vertex of $\Gamma_0,$ change uniformly, $j$ = $0,1, \dots, m_i-1$. Since the transformed graph satisfies the symmetric equidistant condition w.r.t. $v_0$ (and hence w.r.t. each vertices of $\Gamma_0$), the resultant changes on the last edges of each periodic cycle (each starting from $\Gamma_0$) of $C_n(R)$ also follows the same uniform (or periodic) movements and thereby all the transformed cycles corresponding to the periodic cycles of $C_n(R)$ are also periodic cycles in the transformed graph $\theta_{n,m_i,t}(C_n(R))$ and thereby the resultant graph $\theta_{n,m_i,t}(C_n(R))$ = $C_n(S)$ for some $S \subseteq [1, \frac{n}{2}]$. Hence the result follows. 
\end{proof}

The transformation $\theta_{16,2,t}$ acting on $C_{16}(1,2,7)$ is illustrated by problem \ref{p1}. In $C_{16}(1,2,7)$, $r_2$ = 2 is the only jump size such that $\gcd(n, r_2)$ = 2 = $m_2 > 1$. This implies that there exist two disjoint periodic cycles of period 2, namely, $\Gamma_0$ = $(v_0 v_2 v_4 \dots v_0)$ and $\Gamma_1$ = $(v_1 v_3 v_5 \dots v_1)$ in $C_{16}(1,2,7)$. As $t$ changes, the second cycle $\Gamma_1$ of period $2$ also changes. $G_0$, $G_1$, $G_2$, $G_3$ are the transformed graphs isomorphic to $C_{16}(1,2,7)$ corresponding to $t$ = 0, 1, 2, 3, respectively.  See Figures 7,8,9,10. $G_2$ is a circulant graph in its representation and is isomorphic (non-Adam's) to $C_{16}(1,2,7)$. $G_1$ and $G_3$ are not circulant in their representation in the sense that it is not of the form $C_n(S)$ for any $S \subseteq [1, \frac{n}{2}]$. 

\begin{rem}\quad \label{r21} Following steps are used to establish isomorphism of Type-2 w.r.t. $m$ between circulant graphs $C_n(R)$ and $C_n(S)$. (i) $R \neq S$ and $|R|$ = $|S| \geq 3$; (ii) $\exists$ $r\in R,S$, $m > 1$ $\ni$ $m$ and $m^3$ are divisors of $\gcd(n, r)$ and $n$, respectively, and for some $t$ $\ni$ $1 \leq t \leq \frac{n}{m} -1$, $\theta_{n,m,t}(C_n(R))$ = $C_n(S)$ and (iii) $S \neq xR$ for all $x\in\varphi_n$ under arithmetic reflexive modulo $n$. 

Thus, if $C_n(R)$ and $C_n(S)$ are Type-2 isomorphic circulant graphs w.r.t. $m$, then there exist $r\in R,S$, $m > 1$ and some $t$ $\ni$ $m$ is a divisor of $\gcd(n, r)$, $m^3$ is a divisor of $n$, $1 \leq t \leq \frac{n}{m} -1$, $\theta_{n,m,t}(C_n(R))$ = $C_n(S)$ and $S$ $\neq$ $xR$ for all $x\in\varphi_n$ under arithmetic reflexive modulo $n$.
\end{rem} 

\begin{rem} \label{r22} \quad The calculation on jump sizes $r_i$s which are integer multiples of $m$ need not be done under the transformation $\theta_{n,m,t}$, while searching for possible value(s) of $t$ for which the transformed graph $\theta_{n,m,t}(C_n(R))$ is circulant of the form $C_n(S)$ for some $S \subseteq [1, \frac{n}{2}]$, as there is no change in these $r_i$s where $r\in R$ and $m > 1$ and $m^3$ are divisors of $\gcd(n, r)$ and $n$, respectively. 

Thus, if $\theta_{n,m,t}(C_n(R))$ = $C_n(S)$ for some $S$ and thereby $C_n(R)$ $\cong$ $C_n(S)$, then $\theta_{n,m,t}(C_n(R \cup mT))$ = $C_n(S \cup mT)$ for any $T$ and thereby $C_n(R \cup mT)$ $\cong$ $C_n(S \cup mT)$.

Also, for a given $C_n(R)$, w.r.t. different values of $m$, we may get different Type-2 isomorphic circulant graphs.
\end{rem}

\begin{rem} \label{r23} \quad {\rm For given $C_n(R)$ and $C_n(S)$ when either $\theta_{n,m,t}(C_n(R))$ = $C_n(S)$ for some $t$ or $C_n(xR)$ = $C_n(S)$ for some $x$, then $C_n(R)$ and $C_n(S)$ are isomorphic and are either Type-1 or Type-2 w.r.t. $m$, $0 \leq t \leq \frac{n}{m} -1$ and $x\in\varphi_n$. }
\end{rem}

\begin{note} \label{n1} {\rm We consider $t$ = $0,1,\dots,\frac{n}{\gcd(n, r)}-1$ in the transformation $\theta_{n,m,t}$ on $C_n(R)$ since the length of a periodic cycle of period $r$ in $C_n(R)$ is $\frac{n}{\gcd(n, r)}$, using Theorem \ref{t2}, where $m > 1$, $m$ and $m^3$ are divisors of $\gcd(n,r)$ and $n$, respectively, and $r\in R$.}
\end{note}
         
\begin{rem} \label{24} From the above examples, it is noted that  there are Type-1 isomorphic circulant graphs $C_n(R)$ and $C_n(S)$ such that $\theta_{n,m,t}(C_n(R))$ = $C_n(S)$ for some $t\in\mathbb{N}$ with $R \neq S$, $r\in R,S$, $m > 1$, $m$ and $m^3$ are divisors of $\gcd(n,r)$ and $n$, respectively, $r\in R,S$ and $1 \leq t \leq \frac{n}{m} - 1$. And so a separate study is needed to find conditions under which such type of isomorphic circulant graphs are found. 
\end{rem}

When $T2_{n,m,t}(C_n(R))$ = $\theta_{n,m,t}(C_{n}(R))$ = $C_n(S)$ for some $S$ and $C_n(S)\in T2_{n,m}(C_n(R))$, $m$ and $m^3$ divide $\gcd(n,r)$ and $n$, respectively, $r\in R,S$ and $0 \leq t \leq \frac{n}{m}-1$, 

$k T2_{n,m,t}(C_n(R))$ = $k. \theta_{n,m,t}(C_{n}(R))$ = $k. C_n(S)$ 

\hfill  (= union of $k$ copies of $C_n(S)$) = $C_{kn}(kS)$ = $T2_{kn,km,t}(C_{kn}(kR))$. 

Thus, for any circulant graph $C_n(R)$, $k,m\in\mathbb{N}$, $m$ and $m^3$ divide $\gcd(n,r)$ and $n$, respectively, $r\in R,S$, $0 \leq t \leq \frac{n}{m}-1$ and $t\in\mathbb{N}_0$, we always consider $k T2_{n,m,t}(C_n(R))$ = $T2_{kn,km,t}(C_{kn}(kR))$.   

Using definition \ref{d16}, we show in the following problem that circulant graphs $C_{16}(1,2,7)$ and $C_{16}(2,3,5)$ are Type-2 isomorphic w.r.t. $m$ = 2. We also consider more problems to understand definition \ref{d16} and remarks \ref{r21} and \ref{r22}.

\begin{prm}\quad \label{p1} {\rm Show that circulant graphs $C_{16}(1,2,7)$ and $C_{16}(2,3,5)$ are Type-2 isomorphic w.r.t. $m$ = 2. Also, find $\theta_{16,2,t}(C_{16}(1,2,7))$ for all possible values of $t$. }
\end{prm}
\noindent
{\bf Solution.}\quad Let $R$ = $\{1,2,7\}$, $S$ = $R \cup (16-R)$, $T$ = $\{2,3,5\}$, $n$ = 16, $r$ = 2 and $m$ = 2. This implies, $r$ = $2\in R,T$, $\gcd(n, r)$ = $\gcd(16, 2)$ = 2 = $m$ and $m^3$ = 8 divides $n$ = 16. We calculate $\theta_{16,2,t}(s)$ = $s+2jt$ and $\theta_{16,2,t}(C_{16}(1,2,7))$ = $\theta_{16,2,t}(C_{16}(1,2,7, 9,14,15))$ = $C_{16}(\theta_{16,2,t}(1,2,7, 9,14,15))$  for $t$ = 1 to $\frac{16}{\gcd(16, 2)}-1$ = 7, $s$ = $2q+j$, $0 \leq j \leq 1$, $s\in S$ and $q\in\mathbb{N}_0$ and present it in Table 2. From Table 2, we get,  
\\
$\theta_{16,2,0}(S)$ = $\theta_{16,2,0}(\{1,2,7, 9,14,15\})$ = $\{1,2,7, 9,14,15\}$. 

$\Rightarrow$ $\theta_{16,2,0}(C_{16}(R))$ = $C_{16}(R)$;
\\
$\theta_{16,2,1}(S)$ = $\theta_{16,2,1}(\{1,2,7, 9,14,15\})$ = $\{1,2,3, 9,11,14\}$. 

$\Rightarrow$   $\theta_{16,2,1}(C_{16}(R))$ = $\theta_{16,2,1}(C_{16}(S))$ $\neq$ $C_{16}(X)$ for any $X$ using Theorem \ref{t20}; 
\\
$\theta_{16,2,2}(S)$ = $\theta_{16,2,2}(\{1,2,7, 9,14,15\})$ = $\{2,3,5, 11,13,14\}$. 

$\Rightarrow$  $\theta_{16,2,2}(C_{16}(R))$ = $\theta_{16,2,2}(C_{16}(S))$ = $C_{16}(2,3,5)$; 
\\
$\theta_{16,2,3}(S)$ = $\theta_{16,2,3}(\{1,2,7, 9,14,15\})$ = $\{2,5,7, 13,14,15\}$. 

$\Rightarrow$  $\theta_{16,2,3}(C_{16}(R))$ = $\theta_{16,2,3}(C_{16}(S))$ $\neq$ $C_{16}(X)$ for any $X$ using Theorem \ref{t20};  
\\
$\theta_{16,2,4}(S)$ = $\theta_{16,2,4}(\{1,2,7, 9,14,15\})$ = $\{1,2,7, 9,14,15\}$. 

$\Rightarrow$  $\theta_{16,2,4}(C_{16}(R))$ = $\theta_{16,2,4}(C_{16}(S))$ = $C_{16}(R)$; 
\\
$\theta_{16,2,5}(S)$ = $\theta_{16,2,5}(\{1,2,7, 9,14,15\})$ = $\{1,2,3, 9,11,14\}$. 

$\Rightarrow$   $\theta_{16,2,5}(C_{16}(R))$ = $\theta_{16,2,5}(C_{16}(S))$ $\neq$ $C_{16}(X)$ for any $X$ using Theorem \ref{t20}; 
\\
$\theta_{16,2,6}(S)$ = $\theta_{16,2,6}(\{1,2,7, 9,14,15\})$ = $\{2,3,5, 11,13,14\}$. 

$\Rightarrow$  $\theta_{16,2,6}(C_{16}(R))$ = $\theta_{16,2,6}(C_{16}(S))$ = $C_{16}(2,3,5)$;
\\
$\theta_{16,2,7}(S)$ = $\theta_{16,2,7}(\{1,2,7, 9,14,15\})$ = $\{2,5,7, 13,14,15\}$. 

$\Rightarrow$ $\theta_{16,2,7}(C_{16}(R))$ = $\theta_{16,2,7}(C_{16}(S))$ $\neq$ $C_{16}(X)$ for any $X$ using Theorem \ref{t20}.

It is clear from Table 2 that there are 4 distinct isomorphic graphs $\theta_{16,2,t}(C_{16}(1,2,7))$, without vertex label, for $t$ = 0, 1, 2,  3 and these graphs are $G_0$ = $\theta_{16,2,0}(C_{16}(R))$ = $C_{16}(R)$, $G_1$ = $\theta_{16,2,1}(C_{16}(R))$; $G_2$ = $\theta_{16,2,2}(C_{16}(R))$ = $C_{16}(T)$ and $G_3$ = $\theta_{16,2,3}(C_{16}(R))$ which are shown in Figures 7 to 10.

Also, $\theta_{16,2,2}(C_{16}(1,2,7))$ = $C_{16}(2,3,5)$ which implies, $C_{16}(1,2,7)$ $\cong$ $C_{16}(2,3,5)$. 
\\
$Ad_{16}(C_{16}(1,2,7))$ = $\{\varphi_{16,x}(C_{16}(1,2,7)): x = 1,3,5,7,9,11,13,15\}$ 

\hspace{2.25cm} = $\{C_{16}(x(1,2,7)): x = 1,3,5,7,9,11,13,15\}$ = $\{C_{16}(1,2,7), C_{16}(3,5,6)\}$ 

\hspace{2.25cm} = $\{C_{16}(x(1,2,7)): x = 1,3\}$. This implies, $C_{16}(2,3,5)$ $\notin Ad_{16}(C_{16}(1,2,7))$. 

Hence, $C_{16}(1,2,7)$ and $C_{16}(2,3,5)$ are Type-2 isomorphic w.r.t. $m$ = 2.  \hfill $\Box$
\begin{table}
	\caption{ Calculation of $\theta_{16,2,t}(\{1,2,7, 9,14,15\})$ for $t$ = 0 to 7.}
\begin{center}
\scalebox{.9}{
      \begin{tabular}{||c||*{6}{c|}|c||c||}\hline \hline 
	Jump size $x$ &  1 & 2 & 7 & 9 & 14 & 15 & Pairwise Equidistant  \\ \cline{1-7} 
\backslashbox{ {\hspace{.5cm} $t$}}{$\theta_{16,2,t}(x)$} & $1+2t$ & ~~2 ~~&
		$7+2t$ & $9+2t$ & ~~14~~ & $15+2t$ & from $v_0$ or not 	\\\hline \hline
			0 & 1 & 2 & 7 & 9 & 14 & 15 & Yes (Identity) \\
			1 & 3 & 2 & 9 & 11 & 14 & 1 & Not\\
			2 & 5 & 2 & 11 & 13 & 14 & 3 & Yes (Type-2) \\
			3 & 7 & 2 & 13 & 15 & 14 & 5 & Not\\		\hline\hline 
			4 & 9 & 2 & 15 & 1 & 14 & 7 & Yes (Identity) \\
			5 & 11 & 2 & 1 & 3 & 14 & 9 & Not \\
			6 & 13 & 2 & 3 & 5 & 14 & 11 & Yes  (Type-2)\\
			7 & 15 & 2 & 5 & 7 & 14 & 13 & Not \\		\hline \hline
	\end{tabular}}
\end{center}
\end{table} 
\begin{figure}[ht]
	\centerline{\includegraphics[width=6in]{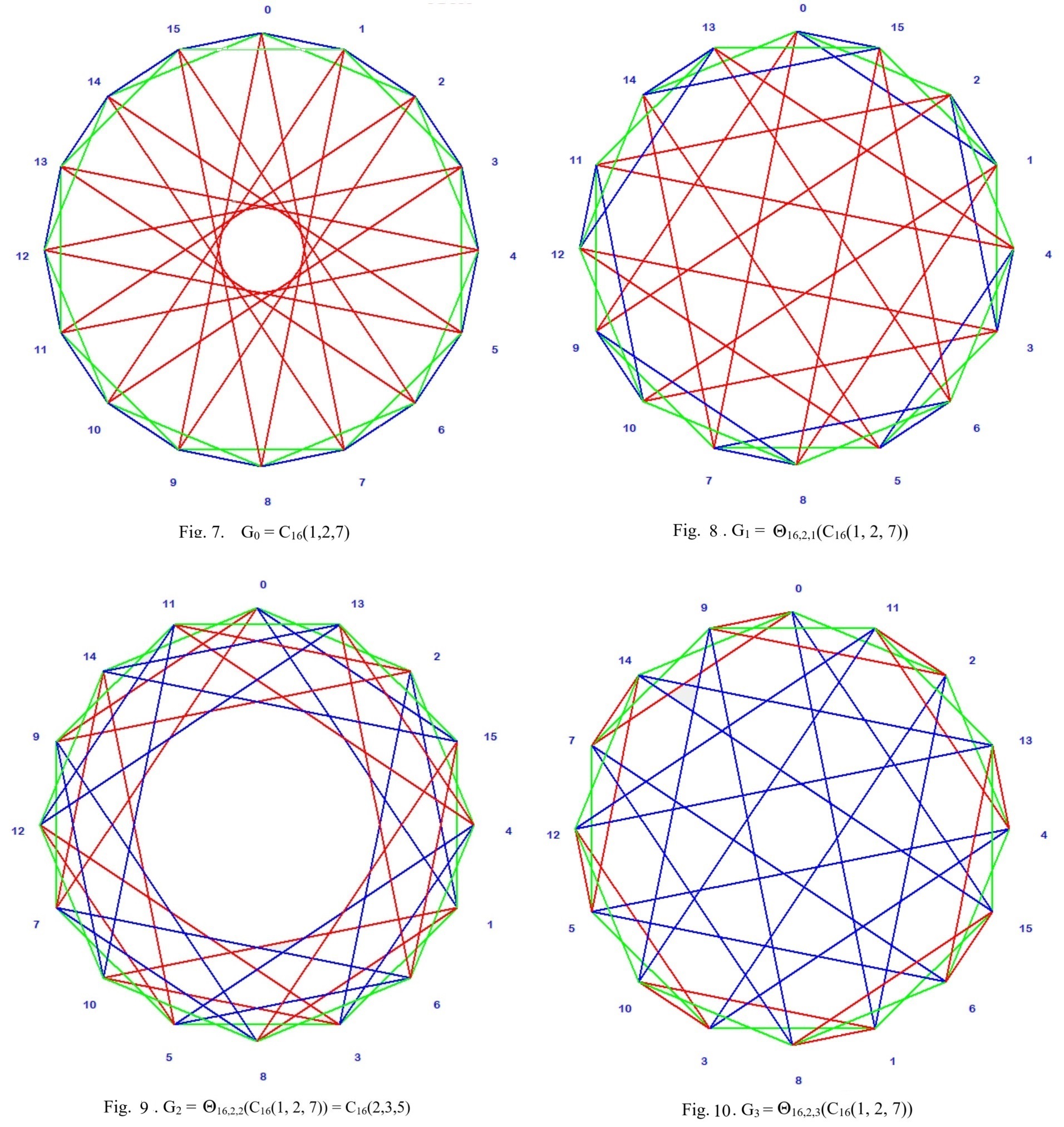}}
\end{figure}

\begin{prm}\quad \label{p2} {\rm Show that the following statements are true.

(i) $C_{16}(1,2,4,7,8)$ and $C_{16}(2,3,4,5,8)$ are Type-2 isomorphic w.r.t. $m$ = 2; and 

(ii) $C_{16}(1,2,4,6,7)$ and $C_{16}(2,3,4,5,6)$ are Type-1 isomorphic.  }
\end{prm}
\noindent
{\bf Solution.}\quad (i)  Let $R$ = $\{1,2,4,7,8\}$, $S$ = $R \cup (16-R)$, $T$ = $\{2,3,4,5,8\}$, $n$ = 16, $r$ = 2 and $m$ = 2. This implies, $r$ = $2\in R,T$, $\gcd(n, r)$ = $\gcd(16, 2)$ = 2 and $m^3$ = 8 divides $n$ = 16. We have $\theta_{16,2,t}(s)$ = $s+2jt$, $s\in R$, $0 \leq j \leq 1$ and $0 \leq t \leq 7$. Consider,  
\\
 $\theta_{16,2,2}(C_{16}(1,2,4,7, 8))$ = $\theta_{16,2,2}(C_{16}(1,2,4,7, 8, 9,12,14,15))$ 

\hspace{3.1cm} = $C_{16}(\theta_{16,2,2}(1,2,4,7, 8, 9,12,14,15))$  = $C_{16}(5,2,4,11, 8, 13,12,14,3)$ 

\hspace{3.1cm} = $C_{16}(2,3,4,5, 8, 11,12,13,14)$ = $C_{16}(2,3,4,5, 8)$. 

$\Rightarrow$ $C_{16}(1,2,4,7, 8)$ $\cong$ $C_{16}(2,3,4,5, 8)$.
\\
$Ad_{16}(C_{16}(1,2,4,7, 8))$ = $\{\varphi_{16,x}(C_{16}(1,2,4,7, 8): x = 1,3,5,7,9,11,13,15\}$ 

\hspace{1.7cm} = $\{C_{16}(x(1,2,4,7, 8)): x = 1,3,5,7,9,11,13,15\}$ = $\{C_{16}(1,2,4,7, 8), C_{16}(3,4,5,6, 8)\}$ 

\hfill = $\{C_{16}(x(1,2,4,7, 8)): x = 1,3\}$. This implies, $C_{16}(2,3,4,5, 8)$ $\notin Ad_{16}(C_{16}(1,2,4,7, 8))$. 

Hence, $C_{16}(1,2,4,7, 8)$ and $C_{16}(2,3,4,5, 8)$ are Type-2 isomorphic w.r.t. $m$ = 2.  
\\
(ii) $C_{16}(3(1,2,4,6,7))$ = $C_{16}(3(1,2,4,6,7, 9,10,12,14,15))$ = $C_{16}(3,6,12,2,5, 11,14,4,10,13)$

\hspace{2.9cm} = $C_{16}(2,3,4,5,6, 10,11,12,13,14)$ = $C_{16}(2,3,4,5,6)$ and 
\\
$Ad_{16}(C_{16}(1,2,4,6,7))$ = $\{C_{16}(1,2,4,6,7), C_{16}(2,3,4,5,6)\}$. 

$\Rightarrow$ $C_{16}(2,3,4,5,6)$ = $C_{16}(3(1,2,4,6,7))\in Ad_{16}(C_{16}(1,2,4,6,7))$ = $T1_{16}(C_{16}(1,2,4,6,7))$. 

This implies,  $C_{16}(1,2,4,6,7)$ and $C_{16}(2,3,4,5,6)$ are Type-1 isomorphic.  \hfill $\Box$

\begin{prm}\quad \label{p3} {\rm Show that the following statements are true. 

(i) $C_{24}(1,2,8,11)$ and $C_{24}(2,5,7,8)$ are Type-2 isomorphic w.r.t. $m$ = 2; and 

(ii) $C_{24}(1,2,10,11)$ and $C_{24}(2,5,7,10)$ are Type-1 isomorphic.   }
\end{prm}
\noindent
{\bf Solution.}\quad (i)  Let $R$ = $\{1,2,8,11\}$, $S$ = $R \cup (24-R)$, $T$ = $\{2,5,7,8\}$, $n$ = 24, $r$ = 2 and $m$ = 2. This implies, $r$ = $2\in R,T$, $\gcd(n, r)$ = $\gcd(24, 2)$ = 2 and $m^3$ = 8 divides $n$ = 24. We have $\theta_{24,2,t}(s)$ = $s+2jt$, $s\in R$, $0 \leq j \leq 1$ and $0 \leq t \leq 11$. Consider,  
\\
 $\theta_{24,2,3}(C_{24}(1,2,8,11))$ = $\theta_{24,2,3}(C_{24}(1,2,8,11,  13,16,22,23))$ 

\hspace{3cm} = $C_{24}(\theta_{24,2,3}(1,2,8,11,  13,16,22,23))$  = $C_{24}(7,2,8,17,  19,16,22,5)$ 

\hspace{3cm} = $C_{24}(2,5,7,8,  16,17,19,22)$ = $C_{24}(2,5,7,8)$. 

$\Rightarrow$ $C_{24}(1,2,8,11)$ $\cong$ $C_{24}(2,5,7,8)$.
\\
$Ad_{24}(C_{24}(1,2,8,11))$ = $\{\varphi_{24,x}(C_{24}(1,2,8,11): x = 1,5,7,11\}$ 

\hspace{2.4cm} = $\{C_{24}(x(1,2,8,11)): x = 1,5,7,11\}$ = $\{C_{24}(1,2,8,11), C_{24}(5,7,8,10)\}$ 

\hfill = $\{C_{24}(x(1,2,8,11)): x = 1,5\}$. This implies, $C_{24}(2,5,7,8) \notin Ad_{24}(C_{24}(1,2,8,11))$. 

Hence, $C_{24}(1,2,8,11)$ and $C_{24}(2,5,7,8)$ are Type-2 isomorphic w.r.t. $m$ = 2.  

Type-2 isomorphic $C_{24}(1,2,8,11)$ and $C_{24}(2,5,7,8)$ w.r.t. $m$ = 2 are given in Figures 11 and 12.
\\
(ii) $C_{24}(5(1,2,10,11))$ = $C_{24}(5(1,2,10,11, 13,14,22,23))$ = $C_{24}(5,10,2,7, 17,22,14,19)$

\hspace{2.9cm} = $C_{24}(2,5,7,10, 14,17,19,22)$ = $C_{24}(2,5,7,10)$ and
\\
$Ad_{24}(C_{24}(1,2,10,11))$ = $\{C_{24}(1,2,10,11), C_{24}(2,5,7,10)\}$. 

$\Rightarrow$ $C_{24}(2,5,7,10)$ = $C_{24}(5(1,2,10,11))\in Ad_{24}(C_{24}(1,2,10,11))$ = $T1_{24}(C_{24}(1,2,10,11))$. 

This implies,  $C_{24}(1,2,10,11)$ and $C_{24}(2,5,7,10)$ are Type-1 isomorphic.

Type-1 isomorphic $C_{24}(1,2,10,11)$ and $C_{24}(2,5,7,10)$ are given in Figures 13 and 14.  \hfill $\Box$
\begin{figure}[ht]
	\centerline{\includegraphics[width=6in]{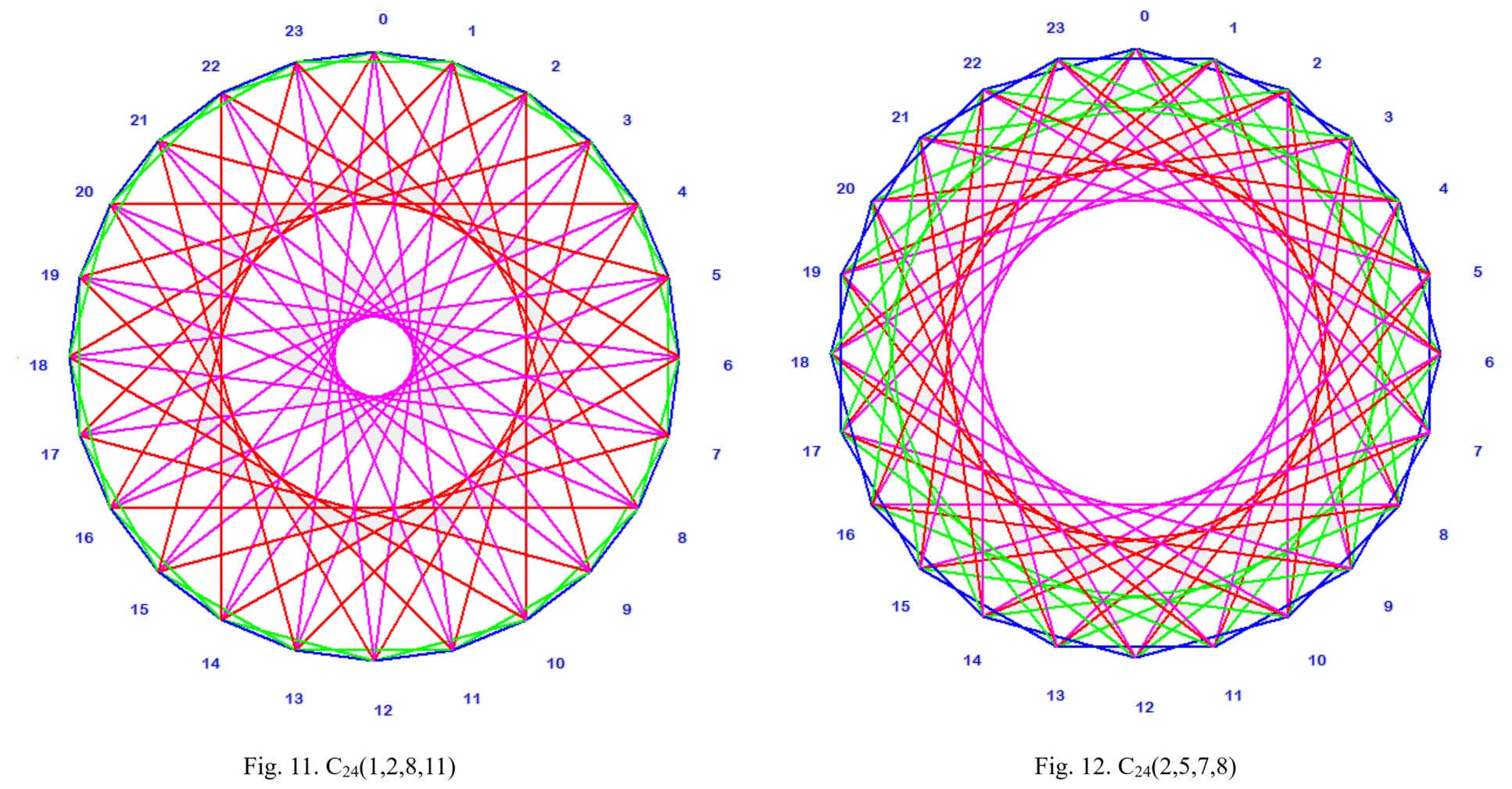}}
\end{figure}
\begin{figure}[ht]
	\centerline{\includegraphics[width=6in]{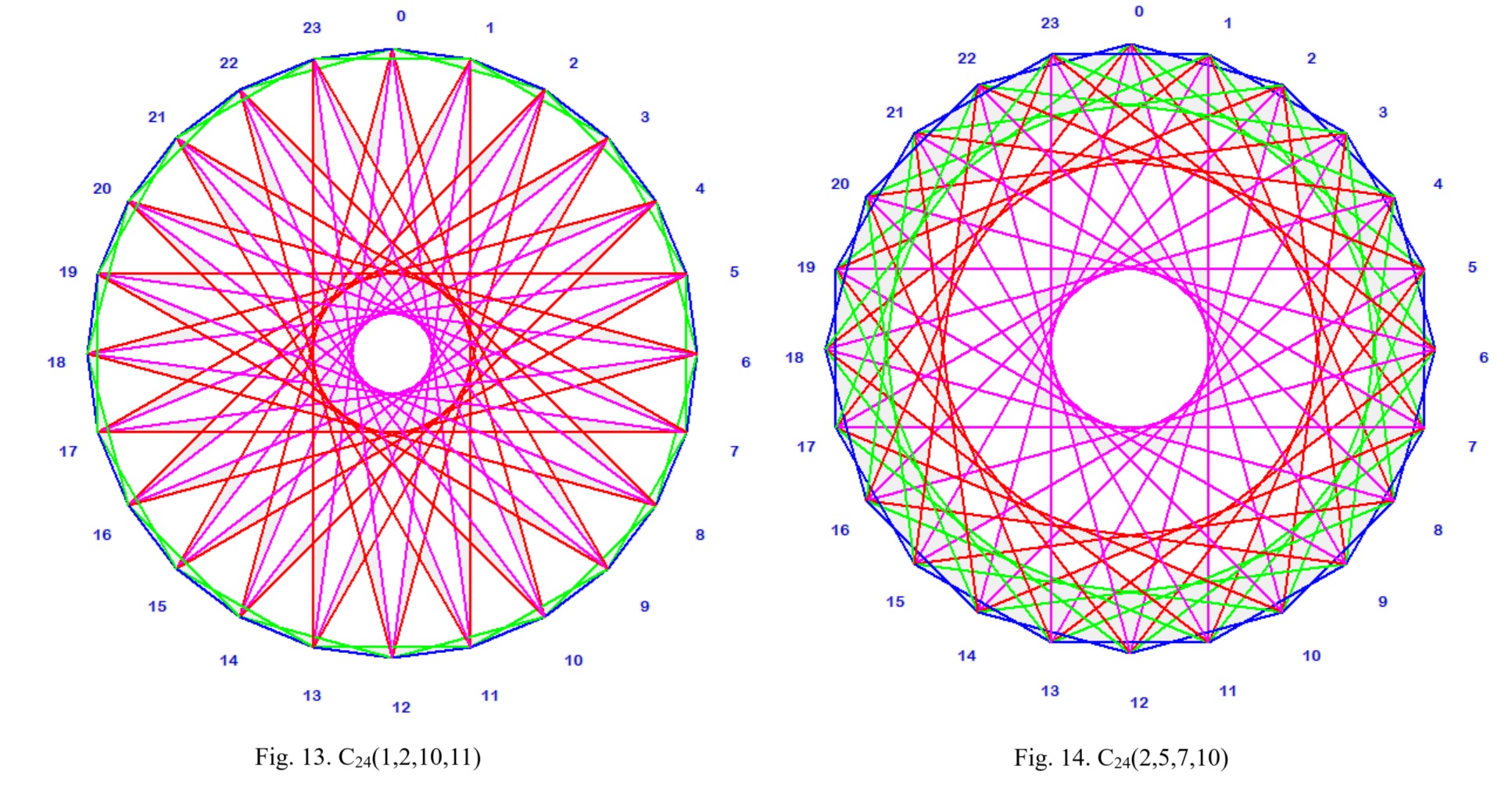}}
\end{figure}

\begin{prm}\quad \label{p4} {\rm Show that the following statements are true. 

(i) $C_{27}(1,3,8,10)$, $C_{27}(3,4,5,13)$ and $C_{27}(2,3,7,11)$ are isomorphic of Type-2 w.r.t. $m$ = 3; and 

(ii) $C_{27}(1,8,10)$, $C_{27}(4,5,13)$ and $C_{27}(2,7,11)$ are Type-1 isomorphic.   }
\end{prm}
\noindent
{\bf Solution.}\quad (i)  Let $R$ = $\{1,3,8,10\}$, $R'$ = $R \cup (27-R)$, $S$ = $\{3,4,5,13\}$, $T$ = $\{2,3,7,11\}$, $n$ = 27, $r$ = 3 and $m$ = 3. This implies, $r$ = $3\in R,S,T$, $\gcd(n, r)$ = $\gcd(27, 3)$ = 3 and $m^3$ = 27 divides $n$ = 27. We have $\theta_{27,3,t}(s)$ = $s+3jt$, $s\in R'$, $0 \leq j \leq 2$ and $0 \leq t \leq 8$. Consider,  
\\
 $\theta_{27,3,1}(C_{27}(1,3,8,10))$ = $\theta_{27,3,1}(C_{27}(1,3,8,10,  17,19,24,26))$ 

\hspace{3cm} = $C_{27}(\theta_{27,3,1}(1,3,8,10,  17,19,24,26))$  = $C_{27}(4,3,14,13,  23,22,24,5)$ 

\hspace{3cm} = $C_{27}(3,4,5,13,  14,22,23,24)$ = $C_{27}(3,4,5,13)$ and 
\\
 $\theta_{27,3,2}(C_{27}(1,3,8,10))$ = $\theta_{27,3,2}(C_{27}(1,3,8,10,  17,19,24,26))$ 

\hspace{3cm} = $C_{27}(\theta_{27,3,2}(1,3,8,10,  17,19,24,26))$  = $C_{27}(7,3,20,16,  2,25,24,11)$ 

\hspace{3cm} = $C_{27}(2,3,7,11,  16,20,24,25)$ = $C_{27}(2,3,7,11)$.

$\Rightarrow$ $C_{27}(1,3,8,10)$ $\cong$ $C_{27}(3,4,5,13)$ and $C_{27}(1,3,8,10)$ $\cong$ $C_{27}(2,3,7,11)$.
\\
$Ad_{27}(C_{27}(1,3,8,10))$ = $\{\varphi_{27, x}(C_{27}(1,3,8,10): x = 1,2,4,5,7,8,10,11,13\}$ 

\hspace{2.7cm} = $\{C_{27}(x(1,3,8,10)): x = 1,2,4,5,7,8,10,11,13\}$ 

\hfill  = $\{C_{27}(1,3,8,10), C_{27}(2,6,7,11), C_{27}(4,5,12,13)\}$ = $\{C_{27}(x(1,3,8,10)): x = 1,2,4\}$.

 This implies, $C_{27}(3,4,5,13),  C_{27}(2,3,7,11) \notin Ad_{27}(C_{27}(1,3,8,10))$. 

Hence, $C_{27}(1,3,8,10)$ and $C_{27}(3,4,5,13)$ as well as  $C_{27}(1,3,8,10)$ and $C_{27}(2,3,7,11)$ are Type-2 isomorphic w.r.t. $m$ = 3.  
\\
(ii) $C_{27}(2(1,8,9,10))$ = $C_{27}(2(1,8,9,10, 17,18,19,26))$ = $C_{27}(2,16,18,20, 7,9,11,25)$

\hspace{2.5cm} = $C_{27}(2,7,9,11, 16,18,20,25)$ = $C_{27}(2,7,9,11)$,  

~ $C_{27}(4(1,8,9,10))$ = $C_{27}(4(1,8,9,10, 17,18,19,26))$ = $C_{27}(4,5,9,13, 14,18,22,23)$ = $C_{27}(4,5,9,13)$ and
\\
$Ad_{27}(C_{27}(1,8,9,10))$ = $\{C_{27}(1,8,9,10), C_{27}(2,7,9,11), C_{27}(4,5,9,13)\}$.

This implies,  $C_{27}(1,8,9,10)$,  $C_{27}(2,7,9,11)$ and $C_{27}(4,5,9,13)$ are Type-1 isomorphic.  \hfill $\Box$

\begin{prm}\quad \label{p5} {\rm Show that the following statements are true. 

(i) $C_{48}(1,2,23)$ and $C_{48}(2,11,13)$ are Type-2 isomorphic w.r.t. $m$ = 2;  

(ii) $C_{48}(1,6,23)$ and $C_{48}(6,11,13)$ are Type-2 isomorphic w.r.t. $m$ = 2;

(iii) $C_{48}(1,4,23)$ and $C_{48}(4,11,13)$ are Type-1 isomorphic.    }
\end{prm}
\noindent
{\bf Solution.}\quad (i)  Let $R$ = $\{1,2,23\}$, $S$ = $R \cup (48-R)$, $T$ = $\{2,11,13\}$, $n$ = 48, $r$ = 2 and $m$ = 2. This implies, $r$ = $2\in R,T$, $\gcd(n, r)$ = $\gcd(48, 2)$ = 2 and $m^3$ = 8 divides $n$ = 48. We have $\theta_{48,2,t}(s)$ = $s+2jt$, $s\in R$, $0 \leq j \leq 1$ and $0 \leq t \leq 23$. Consider,  
\\
 $\theta_{48,2,6}(C_{48}(1,2,23))$ = $\theta_{48,2,6}(C_{48}(1,2,23, 25,46,47))$ = $C_{48}(\theta_{48,2,6}(1,2,23,  25,46,47))$ 

\hspace{2.6cm}  = $C_{48}(13,2,35,  37,46,11)$ = $C_{48}(2,11,13,  35,37,46)$ = $C_{48}(2,11,13)$. 

$\Rightarrow$ $C_{48}(1,2,23)$ $\cong$ $C_{48}(2,11,13)$.
\\
$Ad_{48}(C_{48}(1,2,23))$ = $\{\varphi_{48, x}(C_{48}(1,2,23): x = 1,5,7,11,13,17,19,23\}$ 

\hspace{2.4cm}  = $\{C_{48}(1,2,23), C_{48}(5,10,19), C_{48}(7,14,17), C_{48}(11,13,22)\}$. 

 This implies, $C_{48}(2,11,13) \notin Ad_{48}(C_{48}(1,2,23))$. 

Hence, $C_{48}(1,2,23)$ and $C_{48}(2,11,13)$ are Type-2 isomorphic w.r.t. $m$ = 2.  
\\
(ii)  $\theta_{48,2,6}(C_{48}(1,6,23))$ = $\theta_{48,2,6}(C_{48}(1,6,23,  25,42,47))$ =  $C_{48}(13,2,35,  37,46,11)$ = $C_{48}(2,11,13)$. 

$\Rightarrow$ $C_{48}(1,6,23)$ $\cong$ $C_{48}(6,11,13)$.
\\
$Ad_{48}(C_{48}(1,6,23))$ = $\{\varphi_{48, x}(C_{48}(1,6,23): x = 1,5,7,11,13,17,19,23\}$ 

\hspace{2.4cm} = $\{C_{48}(x(1,6,23)): x = 1,5,7,11,13,17,19,23\}$ 

\hspace{2.4cm} = $\{C_{48}(1,6,23), C_{48}(5,18,19), C_{48}(6,7,17), C_{48}(11,13,18)\}$. 

 This implies, $C_{48}(6,11,13) \notin Ad_{48}(C_{48}(1,6,23))$. 

Hence, $C_{48}(1,6,23)$ and $C_{48}(6,11,13)$ are Type-2 isomorphic w.r.t. $m$ = 2.    
\\
(iii) $C_{48}(11(1,4,23))$ =  $C_{48}(11(1,4,23, 25,44,47))$ = $C_{48}(11,44,13, 35,4,37)$ = $C_{48}(4,11,13)$ and
\\
$Ad_{48}(C_{48}(1,4,23))$ = $\{C_{48}(1,4,23), C_{48}(5,18,19), C_{48}(6,7,17), C_{48}(11,13,18)\}$ 

~ This implies, $C_{48}(1,4,23)$ and $C_{48}(4,11,13)$ are Type-1 isomorphic.               \hfill $\Box$

\begin{prm}\quad \label{p6} {\rm Show that the following statements are true.

 (i) $C_{54}(1,3,17,19)$, $C_{54}(3,7,11,25)$ and $C_{54}(3,5,13,23)$ are isomorphic of Type-2 w.r.t. $m$ = 3;   

(ii) $C_{54}(1,6,17,19)$, $C_{54}(6,7,11,25)$ and $C_{54}(5,6,13,23)$ are isomorphic of Type-2 w.r.t. $m$ = 3; and 

(iii) $C_{54}(1,17,18,19)$, $C_{54}(5,13,18,23)$ and $C_{54}(7,11,18,25)$ are Type-1 isomorphic.   }
\end{prm}
\noindent
{\bf Solution.}\quad (i)  Let $R$ = $\{1,3,17,19\}$, $R'$ = $R \cup (54-R)$, $S$ = $\{3,7,11,25\}$, $T$ = $\{3,5,13,23\}$, $n$ = 54, $r$ = 3 and $m$ = 3. This implies, $r$ = $3\in R,S,T$, $\gcd(n, r)$ = $\gcd(54, 3)$ = 3 and $m^3$ = 27 divides $n$ = 54. We have $\theta_{54,3,t}(s)$ = $s+3jt$, $s\in R'$, $0 \leq j \leq 2$ and $0 \leq t \leq 17$. Consider,  
\\
 $\theta_{54,3,2}(C_{54}(1,3,17,19))$ = $\theta_{54,3,2}(C_{54}(1,3,17,19, 35,37,51,53))$ 

\hspace{3.1cm} = $C_{54}(\theta_{54,3,2}(1,3,17,19, 35,37,51,53))$  = $C_{54}(7,3,29,25,  47,43,51,11)$ 

\hspace{3.1cm} = $C_{54}(3,7,11,25,  29,43,47,51)$ = $C_{27}(3,7,11,25)$ and 
\\
 $\theta_{54,3,4}(C_{54}(1,3,17,19))$ = $\theta_{54,3,4}(C_{54}(1,3,17,19, 35,37,51,53))$ = $C_{27}(13,3,41,31,  5,49,51,23)$ 

\hspace{3.1cm} = $C_{54}(3,5,13,23,  31,41,49,51)$ = $C_{54}(3,5,13,23)$.

$\Rightarrow$ $C_{54}(1,3,17,19)$ $\cong$ $C_{54}(3,7,11,25)$ and $C_{54}(1,3,17,19)$ $\cong$ $C_{54}(3,5,13,23)$.
\\
$Ad_{54}(C_{54}(1,3,17,19))$ = $\{\varphi_{54, x}(C_{54}(1,3,17,19): x = 1,5,7,11,13,17,19,23,25\}$ 

\hspace{3cm} = $\{C_{54}(x(1,3,17,19)): x = 1,5,7,11,13,17,19,23,25\}$ 

\hfill  = $\{C_{54}(1,3,17,19), C_{54}(5,13,15,23), C_{54}(7,11,21,25)\}$ = $\{C_{54}(x(1,3,17,19)): x = 1,5,7\}$.

 This implies, $C_{54}(1,3,17,19),  C_{54}(1,3,17,19) \notin Ad_{54}(C_{54}(1,3,17,19))$. 

Hence, $C_{54}(1,3,17,19)$ and $C_{54}(3,7,11,25)$ as well as  $C_{54}(1,3,17,19)$ and $C_{54}(3,5,13,23)$ are Type-2 isomorphic w.r.t. $m$ = 3.  
\\
(ii)  $\theta_{54,3,2}(C_{54}(1,6,17,19))$ = $\theta_{54,3,2}(C_{54}(1,6,17,19, 35,37,48,53))$ 

\hspace{3.7cm} = $C_{54}(\theta_{54,3,2}(1,6,17,19, 35,37,48,53))$  = $C_{54}(7,6,29,25,  47,43,48,11)$ 

\hspace{3.7cm} = $C_{54}(6,7,11,25,  29,43,47,48)$ = $C_{27}(6,7,11,25)$ and 

 $\theta_{54,3,4}(C_{54}(1,6,17,19))$ = $\theta_{54,3,4}(C_{54}(1,6,17,19, 35,37,48,53))$ = $C_{27}(13,6,41,31,  5,49,48,23)$ 

\hspace{3.5cm} = $C_{54}(5,6,13,23,  31,41,48,49)$ = $C_{54}(5,6,13,23)$.

$\Rightarrow$ $C_{54}(1,6,17,19)$ $\cong$ $C_{54}(6,7,11,25)$ and $C_{54}(1,6,17,19)$ $\cong$ $C_{54}(5,6,13,23)$.
\\
$Ad_{54}(C_{54}(1,6,17,19))$ = $\{\varphi_{54, x}(C_{54}(1,6,17,19): x = 1,5,7,11,13,17,19,23,25\}$ 

\hspace{2.9cm} = $\{C_{54}(x(1,6,17,19)): x = 1,5,7,11,13,17,19,23,25\}$ 

\hfill  = $\{C_{54}(1,6,17,19), C_{54}(5,13,23,24), C_{54}(7,11,12,25)\}$ = $\{C_{54}(x(1,6,17,19)): x = 1,5,7\}$.

 This implies, $C_{54}(6,7,11,25),  C_{54}(5,6,13,23) \notin Ad_{54}(C_{54}(1,6,17,19))$. 

Hence, $C_{54}(1,6,17,19)$ and $C_{54}(6,7,11,25)$ as well as  $C_{54}(1,6,17,19)$ and $C_{54}(5,6,13,23)$ are Type-2 isomorphic w.r.t. $m$ = 3.   

(iii) $C_{54}(5(1,17,18,19))$ = $C_{54}(5(1,17,18,19, 35,36,37,53))$ = $C_{54}(5,31,36,41, 13,18,23,49)$

\hspace{3.7cm}  = $C_{54}(5,13,18,23)$ and 

\hspace{.6cm}   $C_{54}(7(1,17,18,19))$ = $C_{54}(7(1,17,18,19, 35,36,37,53))$ = $C_{54}(7,11,18,25, 29,36,43,47)$ 

\hspace{3.7cm}   = $C_{54}(7,11,18,25)$ and
\\
$Ad_{54}(C_{54}(1,17,18,19))$ = $\{C_{54}(1,17,18,19), C_{54}(5,13,23,24), C_{54}(7,11,12,25)\}$.

This implies, $C_{54}(1,17,18,19)$, $C_{54}(5,13,18,23)$ and $C_{54}(7,11,18,25)$ are Type-1 isomorphic.  \hfill $\Box$

\begin{prm}\quad \label{p7} {\rm Show that the following statements are true.

 (i) $C_{108}(3,5,31,41)$, $C_{108}(3,7,29,43)$ and $C_{108}(3,17,19,53)$ are Type-2 isomorphic w.r.t. $m$ = 3;   

(ii) $C_{108}(5,12,31,41)$, $C_{108}(7,12,29,43)$ and $C_{108}(12,17,19,53)$ are Type-2 isomorphic w.r.t. $m$ = 3; 

(iii) $C_{108}(5,18,31,41)$, $C_{108}(7,18,29,43)$ and $C_{108}(17,18,19,53)$ are Type-1 isomorphic.   }
\end{prm}
\noindent
{\bf Solution.}\quad (i)  Let $R$ = $\{3,5,31,41\}$, $R'$ = $R \cup (108-R)$, $S$ = $\{3,7,29,43\}$, $T$ = $\{3,17,19,53\}$, $n$ = 108, $r$ = 3 and $m$ = 3. This implies, $r$ = $3\in R,S,T$, $\gcd(n, r)$ = $\gcd(108, 3)$ = 3 and $m^3$ = 27 divides $n$ = 108. We have $\theta_{108,3,t}(s)$ = $s+3jt$, $s\in R'$, $0 \leq j \leq 2$ and $0 \leq t \leq 35$. Consider,  
\\
 $\theta_{108,3,4}(C_{108}(3,5,31,41))$ = $\theta_{108,3,4}(C_{108}(3,5,31,41, 67,77,103,105))$ 

\hspace{3.4cm} = $C_{108}(\theta_{108,3,4}(3,5,31,41, 67,77,103,105))$ 

\hspace{3.4cm}  = $C_{108}(3,29,43,65, 79,101,7,105)$ = $C_{108}(3,7,29,43)$ and 
\\
 $\theta_{108,3,8}(C_{108}(3,5,31,41))$ = $\theta_{108,3,8}(C_{108}(3,5,31,41, 67,77,103,105))$ 

\hspace{3.4cm}  = $C_{108}(3,53,55,89, 91,17,19,105)$ = $C_{108}(3,17,19,53)$.

$\Rightarrow$ $C_{108}(3,5,31,41)$ $\cong$ $C_{108}(3,7,29,43)$ and $C_{108}(3,5,31,41)$ $\cong$ $C_{108}(3,17,19,53)$.
\\
$Ad_{108}(C_{108}(3,5,31,41))$ 

\hfill = $\{\varphi_{108, x}(C_{108}(3,5,31,41): x = 1,5,7,11,13,17,19,23,25,29,31,35,37,41,43,47,49,53\}$ 

\hspace{1.6cm} = $\{C_{108}(x(3,5,31,41)): x = 1,5,7,11,13,17,19,23,25,29,31,35,37,41,43,47,49,53\}$ 

\hspace{1.6cm}  = $\{C_{108}(3,5,31,41), C_{108}(11,15,25,47), C_{108}(1,21,35,37)$, 

\hfill $C_{108}(17,19,33,53), C_{108}(7,29,39,43), C_{108}(13,23,49,51)\}$.

 This implies, $C_{108}(3,7,29,43),  C_{108}(3,17,19,53) \notin Ad_{108}(C_{108}(3,5,31,41))$. 

Hence, $C_{108}(3,5,31,41)$ and $C_{108}(3,7,29,43)$ as well as  $C_{108}(3,5,31,41)$ and $C_{108}(3,17,19,53)$ are Type-2 isomorphic w.r.t. $m$ = 3.  
\\
(ii)  $\theta_{108,3,4}(C_{108}(5,12,31,41))$ = $\theta_{108,3,4}(C_{108}(5,12,31,41, 67,77,96,103))$ 

\hspace{4.1cm} = $C_{108}(\theta_{108,3,4}(5,12,31,41, 67,77,96,103))$  

\hspace{4.1cm}  = $C_{108}(29,12,43,65, 79,101,96,7)$ = $C_{108}(7,12,29,43)$ and 

~ $\theta_{108,3,8}(C_{108}(5,12,31,41))$ = $\theta_{108,3,8}(C_{108}(5,12,31,41, 67,77,96,103))$ 

\hspace{4.2cm} = $C_{108}(12,53,55,89, 91,17,96,19)$ = $C_{108}(12,17,19,53)$.

$\Rightarrow$ $C_{108}(5,12,31,41)$ $\cong$ $C_{108}(7,12,29,43)$ and $C_{108}(5,12,31,41)$ $\cong$ $C_{108}(12,17,19,53)$.
\\
$Ad_{108}(C_{108}(5,12,31,41))$ 

\hfill = $\{\varphi_{108, x}(C_{108}(5,12,31,41): x = 1,5,7,11,13,17,19,23,25,29,31,35,37,41,43,47,49,53\}$ 

\hspace{1.4cm} = $\{C_{108}(x(5,12,31,41)): x = 1,5,7,11,13,17,19,23,25,29,31,35,37,41,43,47,49,53\}$ 

\hspace{1.4cm}  = $\{C_{108}(5,12,31,41), C_{108}(11,25,47,48), C_{108}(1,24,35,37)$, 

\hfill $C_{108}(17,19,24,53), C_{108}(7,29,43,48), C_{108}(12,13,23,49)\}$.

 This implies, $C_{108}(7,12,29,43),  C_{108}(12,17,19,53) \notin Ad_{108}(C_{108}(5,12,31,41))$. 

Hence, $C_{108}(5,12,31,41)$ and $C_{108}(7,12,29,43)$ as well as  $C_{108}(5,12,31,41)$ and $C_{108}(12,17,19,53)$ are Type-2 isomorphic w.r.t. $m$ = 3.    

(iii) $C_{108}(13(5,18,31,41))$ = $C_{108}(13(5,18,31,41, 67,77,90,103))$ = $C_{54}(65,18,79,101, 7,29,90,43)$

\hspace{3.9cm}  = $C_{108}(7,18,29,43)$ and

\hspace{.6cm}$C_{108}(11(5,18,31,41))$ = $C_{108}(11(5,18,31,41, 67,77,90,103))$ = $C_{54}(55,90,17,19, 89,91,18,53)$

\hspace{3.9cm}  = $C_{108}(17,18,19,53)$.  

This implies, $C_{108}(5,18,31,41)$, $C_{108}(7,18,29,43)$ and $C_{108}(17,18,19,53)$ are Type-1 isomorphic. Also, 
\\
$Ad_{108}(C_{108}(5,18,31,41))$ = $\{C_{108}(5,18,31,41), C_{108}(11,18,25,47), C_{108}(1,18,35,37)$, 

\hfill $C_{108}(17,18,19,53), C_{108}(7,18,29,43), C_{108}(13,18,23,49)\}$.    \hfill $\Box$
 
\begin{prm}\quad \label{p8} {\rm Show that the following statements are true.

 (i) $C_{108}(3,4,32,40)$, $C_{108}(3,16,20,52)$ and $C_{108}(3,8,28,44)$ are Type-2 isomorphic w.r.t. $m$ = 3;   

(ii) $C_{108}(4,12,32,40)$, $C_{108}(12,16,20,52)$ and $C_{108}(8,12,28,44)$ are Type-2 isomorphic w.r.t. $m$ = 3; 

(iii) $C_{108}(4,18,32,40)$, $C_{108}(16,18,20,52)$ and $C_{108}(8,18,28,44)$ are Type-1 isomorphic.   }
\end{prm}
\noindent
{\bf Solution.}\quad (i)  Let $R$ = $\{3,4,32,40\}$, $R'$ = $R \cup (108-R)$, $S$ = $\{3,16,20,52\}$, $T$ = $\{3,8,28,44\}$, $n$ = 108, $r$ = 3 and $m$ = 3. This implies, $r$ = $3\in R,S,T$, $\gcd(n, r)$ = $\gcd(108, 3)$ = 3 and $m^3$ = 27 divides $n$ = 108. We have $\theta_{108,3,t}(s)$ = $s+3jt$, $s\in R'$, $0 \leq j \leq 2$ and $0 \leq t \leq 35$. Consider,  
\\
 $\theta_{108,3,4}(C_{108}(3,4,32,40))$ = $\theta_{108,3,4}(C_{108}(3,4,32,40, 68,76,104,105))$ 

\hspace{3.4cm} = $C_{108}(\theta_{108,3,4}(3,4,32,40, 68,76,104,105))$ 

\hspace{3.4cm}  = $C_{108}(3,16,56,52, 92,88,20,105)$ = $C_{108}(3,16,20,52)$ and 
\\
 $\theta_{108,3,8}(C_{108}(3,4,32,40))$ = $\theta_{108,3,8}(C_{108}(3,4,32,40, 68,76,104,105))$ 

\hspace{3.4cm}  = $C_{108}(3,28,80,64, 8,100,44,105)$ = $C_{108}(3,8,28,44)$.

$\Rightarrow$ $C_{108}(3,4,32,40)$ $\cong$ $C_{108}(3,16,20,52)$ and $C_{108}(3,4,32,40)$ $\cong$ $C_{108}(3,8,28,44)$.
\\
$Ad_{108}(C_{108}(3,4,32,40))$ 

\hfill = $\{\varphi_{108, x}(C_{108}(3,4,32,40): x = 1,5,7,11,13,17,19,23,25,29,31,35,37,41,43,47,49,53\}$ 

\hspace{1.6cm} = $\{C_{108}(x(3,4,32,40)): x = 1,5,7,11,13,17,19,23,25,29,31,35,37,41,43,47,49,53\}$ 

\hspace{1.6cm}  = $\{C_{108}(3,4,32,40), C_{108}(15,16,20,52), C_{108}(8,21,28,44)$, 

\hfill $C_{108}(8,28,33,44), C_{108}(16,20,39,52), C_{108}(4,32,40,51)\}$.

 This implies, $C_{108}(3,16,20,52),  C_{108}(3,8,28,44) \notin Ad_{108}(C_{108}(3,4,32,40))$. 

Hence, $C_{108}(3,4,32,40)$ and $C_{108}(3,16,20,52)$ as well as  $C_{108}(3,4,32,40)$ and $C_{108}(3,8,28,44)$ are Type-2 isomorphic w.r.t. $m$ = 3.  
\\
(ii)  $\theta_{108,3,4}(C_{108}(4,12,32,40))$ = $\theta_{108,3,4}(C_{108}(4,12,32,40, 68,76,96,104))$ 

\hspace{4.1cm} = $C_{108}(\theta_{108,3,4}(4,12,32,40, 68,76,96,104))$ 

\hspace{4.1cm}  = $C_{108}(16,12,56,52, 92,88,96,20)$ = $C_{108}(12,16,20,52)$ and 
\\
 $\theta_{108,3,8}(C_{108}(4,12,32,40))$ = $\theta_{108,3,8}(C_{108}(4,12,32,40, 68,76,96,104))$ 

\hspace{3.55cm}  = $C_{108}(28,12,80,64, 8,100,96,44)$ = $C_{108}(8,12,28,44)$.

$\Rightarrow$ $C_{108}(4,12,32,40)$ $\cong$ $C_{108}(12,16,20,52)$ and $C_{108}(4,12,32,40)$ $\cong$ $C_{108}(8,12,28,44)$.
\\
$Ad_{108}(C_{108}(4,12,32,40))$ 

\hfill = $\{\varphi_{108, x}(C_{108}(4,12,32,40): x = 1,5,7,11,13,17,19,23,25,29,31,35,37,41,43,47,49,53\}$ 

\hspace{1.4cm} = $\{C_{108}(x(4,12,32,40)): x = 1,5,7,11,13,17,19,23,25,29,31,35,37,41,43,47,49,53\}$ 

\hspace{1.4cm}  = $\{C_{108}(4,12,32,40), C_{108}(16,20,48,52), C_{108}(8,24,28,44)\}$.

 This implies, $C_{108}(12,16,20,52),  C_{108}(8,12,28,44) \notin Ad_{108}(C_{108}(4,12,32,40))$. 

Hence, $C_{108}(4,12,32,40)$ and $C_{108}(12,16,20,52)$ as well as  $C_{108}(4,12,32,40)$ and $C_{108}(8,12,28,44)$ are Type-2 isomorphic w.r.t. $m$ = 3. 

(iii) $C_{108}(5(4,18,32,40))$ = $C_{108}(5(4,18,32,40, 68,76,90,104))$ = $C_{54}(20,90,52,92, 16,56,18,88)$

\hspace{3.75cm}  = $C_{108}(16,18,20,52)$ and

\hspace{.5cm} $C_{108}(7(4,18,32,40))$ = $C_{108}(7(4,18,32,40, 68,76,90,104))$ = $C_{54}(28,18,8,64, 44,100,90,80)$

\hspace{3.75cm}  = $C_{108}(8,18,28,44)$.  

This implies, $C_{108}(4,18,32,40)$, $C_{108}(16,18,20,52)$ and $C_{108}(8,18,28,44)$ are Type-1 isomorphic.  \hfill $\Box$
 
\section{Results on isomorphic circulant graphs of Type-2 w.r.t. $m$ = 2}

 In this section, we obtain new family of isomorphic circulant graphs of Type-2 w.r.t. $m$ = 2 and thereby new family of circulant graphs without CI-property. The following results on Type-2 isomorphic circulant graphs w.r.t. $m$ = 2 are obtained in \cite{v17,v20} but modified proofs are presented here. 

\begin{theorem} \label{t25} {\rm For $n,s \in \mathbb{N}$, $1 \leq 2s-1 \leq 2n-1$, $R$ = $\{2,2s-1,4n-(2s-1)\}$ and $S$ = $\{ 2,2n-(2s-1)$, $2n+2s-1 \}$,  $\theta_{8n,2,n}(C_{8n}(R))$ = $C_{8n}(S)$ = $\theta_{8n,2,3n}(C_{8n}(R))$, $\theta_{8n,2,n}(C_{8n}(S))$ = $C_{8n}(R)$ = $\theta_{8n,2,3n}(C_{8n}(S))$, $\theta_{8n,2,2n}(C_{8n}(R))$ = $C_{8n}(R)$ and $\theta_{8n,2,2n}(C_{8n}(S))$ = $C_{8n}(S)$.  }
\end{theorem}
\begin{proof}\quad Using the definition of $\theta_{n,m,t},$ we get, 

$\theta_{8n,2,n}(R \cup (8n-R))$ = $\theta_{8n,2,n}(2, 2s-1, 4n-(2s-1), 4n+(2s-1), 8n-(2s-1), 8n-2)$ 

\hspace{1.8cm} = $\{2, 2n+2s-1, 6n-(2s-1), 6n+(2s-1), 2n-(2s-1), 8n-2\}$ 

\hfill = $\{2, 2n-(2s-1), 2n+2s-1, 6n-(2s-1), 6n+(2s-1), 8n-2\}$ = $\{S \cup (8n-S)\}$; 

$\theta_{8n,2,2n}(R \cup (8n-R))$ = $\theta_{8n,2,2n}(2, 2s-1, 4n-(2s-1), 4n+(2s-1), 8n-(2s-1), 8n-2)$ 

\hspace{2.15cm} = $\{2, 4n+2s-1, 8n-(2s-1), 2s-1, 4n-(2s-1), 8n-2\}$ 

\hfill = $\{2, 2s-1, 4n-(2s-1), 4n+2s-1, 8n-(2s-1), 8n-2\}$ = $\{R \cup (8n-R)\}$; and

$\theta_{8n,2,3n}(R \cup (8n-R))$ = $\theta_{8n,2,3n}(2, 2s-1, 4n-(2s-1), 4n+(2s-1), 8n-(2s-1), 8n-2)$ 

\hfill = $\{2, 6n+2s-1, 2n-(2s-1), 2n+(2s-1), 6n-(2s-1), 8n-2\}$ = $\{S \cup (8n-S)\}$. 

$\Rightarrow$  $\theta_{8n,2,n}(C_{8n}(R))$ = $C_{8n}(S)$ = $\theta_{8n,2,3n}(C_{8n}(R))$ and $\theta_{8n,2,2n}(C_{8n}(R))$ = $C_{8n}(R)$. 

Similarly, it is easy to show that $\theta_{8n,2,n}(C_{8n}(S))$ = $C_{8n}(R)$ = $\theta_{8n,2,3n}(C_{8n}(S))$ and $\theta_{8n,2,2n}(C_{8n}(S))$ = $C_{8n}(S)$ where $R$ = $\{2,2s-1, 4n-(2s-1)\}$, $S$ = $\{2, 2n-(2s-1), 2n+2s-1\}$, $1 \leq 2s-1 \leq 2n-1$ and $n,s \in \mathbb{N}$. Hence the result. 
\end{proof}

\begin{theorem} \quad \label{t26} {\rm For $n \geq 2$, $1 \leq 2s-1 \leq 2n-1$, $n \neq 2s-1$, $R$ = $\{2,2s-1, 4n-(2s-1)\}$ and $S$ = $\{ 2$, $2n-(2s-1)$, $2n+2s-1 \}$, $\theta_{8n,2,n}(C_{8n}(R))$ = $C_{8n}(S)$ = $\theta_{8n,2,3n}(C_{8n}(R))$, $\theta_{8n,2,n}(C_{8n}(S))$ = $C_{8n}(R)$ = $\theta_{8n,2,3n}(C_{8n}(S))$, $\theta_{8n,2,2n}(C_{8n}(R))$ = $C_{8n}(R)$, $\theta_{8n,2,2n}(C_{8n}(S))$ = $C_{8n}(S)$ and circulant graphs $C_{8n}(R)$ and $C_{8n}(S)$ are Type-2 isomorphic  w.r.t. $m$ = 2. When $n$ = $2s-1$, the two circulant graphs are the same. }
\end{theorem}
\begin{proof}\quad When $R$ = $\{ 2,2s-1,4n-(2s-1)\}$ and $S$ = $\{ 2,2n-(2s-1),2n+2s-1 \}$, $\theta_{8n,2,n}(C_{8n}(R))$ = $C_{8n}(S)$ = $\theta_{8n,2,3n}(C_{8n}(R))$, $\theta_{8n,2,n}(C_{8n}(S))$ = $C_{8n}(R)$ = $\theta_{8n,2,3n}(C_{8n}(S))$, $\theta_{8n,2,2n}(C_{8n}(R))$ = $C_{8n}(R)$ and $\theta_{8n,2,2n}(C_{8n}(S))$ = $C_{8n}(S)$ using Theorem 
\ref{t25} where $1 \leq 2s-1 \leq 2n-1$ and $n,s \in \mathbb{N}$. This implies, $C_{8n}(R)$ $\cong$ $C_{8n}(S)$. The two circulant graphs $C_{8n}(R)$ and $C_{8n}(S)$ are the same when $2s-1$ = $2n-(2s-1)$ and $4n-(2s-1)$ = $2n+2s-1$, $1 \leq 2s-1 \leq 2n-1$. That is when $n$ = $2s-1$, the two circulant graphs become $C_{8n}(2,n,3n)$, $n \in \mathbb{N}$. When $n$ = $s$ = 1, the two graphs become $C_8(1,2,3)$. For $n \neq 2s-1$, $n,s\in \mathbb{N}$ and $n \geq 2$, the set of jump sizes of the two isomorphic circulant graphs $C_{8n}(R)$ and $C_{8n}(S)$ are different and thereby the two isomorphic circulant graphs are not same.   \\

\noindent
{\it \bf {Claim:}} $C_{8n}(R)$ and $C_{8n}(S)$ are of Type-2 isomorphic w.r.t. $m$ = 2 when $R$ = $\{2,2s-1,4n-(2s-1)\}$, $S$ = $\{2,2n-(2s-1),2n+2s-1\}$, $n \neq 2s-1$, $n \geq 2$ and $n,s \in \mathbb{N}$.
	
If not, graphs $C_{8n}(R)$ and $C_{8n}(S)$ are Type-1 isomorphic. Then, there exists $s' \in \mathbb{N}$ such that $\gcd(8n,2s'-1)$ = 1 and $C_{8n}((2s'-1)R)$ = $C_{8n}(S)$ which implies, 

$(2s'-1)\{ 2,2s-1,4n-(2s-1),4n+2s-1$, $8n-(2s-1), 8n-2 \}$ 

\hfill = $\{ 2,2n-(2s-1)$, $2n+2s-1, 6n-(2s-1), 6n+2s-1$, $8n-2 \}$, 
\\
using arithmetic modulo $8n$, $1 \leq 2s'-1 \leq 8n-1$. Now, $2(2s'-1)$, $(8n-2)(2s'-1), 2+8np_1$ and $8n-2+8np_2$ are the only even numbers in the two sets of the above relation for some $p_1, p_2 \in \mathbb{N}_0$. Then the following two cases arise.  
	
\noindent
{\bf Case (i).} $2(2s'-1)$ = $2+8np_1$, $1 \leq 2s'-1 \leq 8n-1$ and $p_1 \in \mathbb{N}_0$.  
	
In this case, the possible values of $p_1$ are 0 and 1 since $1 \leq 2s'-1 \leq 8n-1$ and $s',n \in \mathbb{N}$. When $p_1$ = 0, $2s'-1$ = 1 and so we get the same graph only. When $p_1$ = 1, $2s'-1$ = $4n+1$. The jump sizes of the circulant graph corresponding to Type-1 isomorphism when $2s'-1$ = $4n+1$ are given in Table 3. In this case, the two graphs are the same. See Table 3. 
 
\noindent
\textbf{Case (ii).} $2(2s'-1)$ = $8n-2+8np_2$, $1 \leq 2s'-1 \leq 8n-1$ and $p_2 \in \mathbb{N}_0$.  
	
	In this case, the possible values of $p_2$ are $0$ and $1.$ When $p_2$ = $0$, $2s'-1$ = $4n-1$. When $p_2$ = 1, $2s'-1$ = $8n-1$. The jump sizes of the circulant graph corresponding to Type-1 isomorphism when $2s'-1$ = $4n-1$ as well as when $2s'-1$ = $8n-(4n-1)$ = $4n+1$ are given in Table 3. In the case $2s'-1$ = $4n-1$ as well as when $2s'-1$ = $8n-1$ (which is same as when $2s'-1$ = $8n-(8n-1)$ = 1), we get the same circulant graph $C_{8n}(R)$ only. See Table 3.
	
	Thus, the isomorphic circulant graphs $C_{8n}(R)$ and $C_{8n}(S)$ for $R$ = $\{ 2, 2s-1, 4n-(2s-1) \}$, $S$ = $\{ 2$, $2n-(2s-1)$, $2n+2s-1 \}$, $1 \leq 2s-1 \leq 2n-1$, $n \geq 2$  and  $n \neq 2s-1$ (and thereby $R \neq S$, $2n-(2s-1) \neq 2s-1$ and $2n+2s-1 \neq 4n-(2s-1)$, $1 \leq 2s-1 \leq 2n-1$) are different, not of Type-1 and $\theta_{8n,2,n}(C_{8n}(R))$ = $C_{8n}(S)$. And hence the result follows. 

	\begin{table}
		\caption{{\small Calculation of $s(2s'-1)$ under arithmetic modulo $8n.$}}
		\begin{center}
			\scalebox{0.8}{
				\begin{tabular}{||c||*{5}{c|}c||}\hline \hline
					\backslashbox{Multiplier \\ $2s'-1$}{Jump size ~ $s$}
					& 2 & $2s-1$ & $4n-(2s-1)$ & $4n+2s-1$ & $8n-(2s-1)$ & $8n-2$\\\hline \hline
					& &  &   &  &  &  \\
					$4n-1$ & $8n-2$ & $4n-(2s-1)$ & $2s-1$ & $8n-(2s-1)$ & $4n+2s-1$ & 2  \\\hline
					& &  &   &  &  &  \\
					$4n+1$ & 2 & $4n+2s-1$ & $8n-(2s-1)$ & $2s-1$ & $4n-(2s-1)$ & $8n-2$ \\\hline
					& &  &   &  &  &  \\
					$8n-1$ & $8n-2$ & $8n-(2s-1)$ & $4n+2s-1$ & $4n-(2s-1)$ & $2s-1$ & 2  \\\hline \hline 
			\end{tabular}}
		\end{center}
	\end{table}
\end{proof}

\begin{theorem} \label{t27} {\rm Let $n \geq 2$, $k \geq 3$, $1 \leq 2s-1 \leq 2n-1$, $n \neq 2s-1$, $R$ = $\{ 2s-1,$ $4n-(2s-1),$ $2p_1,$ $2p_2,$ $\dots,$ $2p_{k-2} \}$, $S$ = $\{2n-(2s-1),$ $2n+2s-1,$ $2p_1,2p_2,\dots,2p_{k-2}\}$, $2y\in R,S$, $\gcd(4n,y)$ = 1, $p_1,p_2,\dots,p_{k-2} \in \mathbb{N}$ and $\gcd(p_1,p_2,\dots,p_{k-2})$ = 1. Then, (i) $\theta_{8n,2,n}(C_{8n}(R))$ = $C_{8n}(S)$ = $\theta_{8n,2,3n}(C_{8n}(R))$, $\theta_{8n,2,n}(C_{8n}(S))$ = $C_{8n}(R)$ = $\theta_{8n,2,3n}(C_{8n}(S))$, $\theta_{8n,2,2n}(C_{8n}(R))$ = $C_{8n}(R)$, $\theta_{8n,2,2n}($ $C_{8n}(S))$ = $C_{8n}(S)$ and (ii) for given values of $n,s,p_1,p_2,\dots,p_{k-2}$ and $y$, $C_{8n}(R)$ and $C_{8n}(S)$ are isomorphic of either Type-1 or Type-2 w.r.t. $m$ = 2. Moreover, for given $n, s$ and $k$ and for all such possible values of $p_1,p_2,\dots,p_{k-2}$ and $y$, the set $\{ C_n(S)$ = $\theta_{8n,2,n}(C_{8n}(R)):$  $\gcd(p_1,p_2,\dots,p_{k-2})$ = 1, $p_1,p_2,\dots,p_{k-2} \in \mathbb{N}\}$ contains all isomorphic circulant graphs of $C_{8n}(R)$ of Type-2 w.r.t. $m$ = 2. }
\end{theorem}
\begin{proof}\quad When $R$ = $\{2,2s-1,4n-(2s-1)\},$ $S$ = $\{ 2, 2n-(2s-1), 2n+2s-1 \},$ $n \neq 2s-1$, $n \geq 2$ and $n \in \mathbb{N}$, using Theorem 
\ref{t25}, $\theta_{8n,2,n}(C_{8n}(R))$ = $C_{8n}(S)$ = $\theta_{8n,2,3n}(C_{8n}(R))$, $\theta_{8n,2,n}(C_{8n}(S))$ = $C_{8n}(R)$ = $\theta_{8n,2,3n}(C_{8n}(S))$, $\theta_{8n,2,2n}(C_{8n}(R))$ = $C_{8n}(R)$, $\theta_{8n,2,2n}(C_{8n}(S))$ = $C_{8n}(S)$ and circulant graphs $C_{8n}(R)$ and $C_{8n}(S)$ are Type-2 isomorphic  w.r.t. $m$ = 2. And using Remark \ref{r22}, for a given set of values of $n,s,p_1,p_2,\dots,p_{k-2}$ and $y$,   $\theta_{8n,2,n}(C_{8n}(R))$ = $C_{8n}(S)$ = $\theta_{8n,2,3n}(C_{8n}(R))$, $\theta_{8n,2,n}(C_{8n}(S))$ = $C_{8n}(R)$ = $\theta_{8n,2,3n}(C_{8n}(S))$, $\theta_{8n,2,2n}(C_{8n}(R))$ = $C_{8n}(R)$ and $\theta_{8n,2,2n}(C_{8n}(S))$ = $C_{8n}(S)$ and thereby $C_{8n}(R)$ and $C_{8n}(S)$ are isomorphic. Moreover, $C_{8n}(R)$ and $C_{8n}(S)$ are either Type-1 isomorphic or Type-2 isomorphic w.r.t. $m$ = 2 when $R$ = $\{ 2s-1$, $4n-2s+1$, $2p_1$, $2p_2$, $\dots$, $2p_{k-2} \}$, $S$ = $\{2n-(2s-1),2n+2s-1, 2p_1,2p_2,\dots,2p_{k-2}\}$, $n \geq 2$, $k \geq 3$, $1 \leq 2s-1 \leq 2n-1$, $n \neq 2s-1$, $\gcd(8n, 2y)$ = $m$ = 2,  $\gcd(p_1,p_2,\dots,p_{k-2})$ = 1, $2y\in R,S$ and $n,s, y,p_1,p_2,\dots,p_{k-2} \in \mathbb{N}$.
	
Also, for given $n, s$ and $k$ and for all such possible values of $p_1,p_2,\dots,p_{k-2}$ and $y$, the set $\{ C_n(S)$ = $\theta_{8n,2,n}(C_{8n}(R)):$ $p_1,p_2,\dots,p_{k-2} \in \mathbb{N}$  and $\gcd(p_1,p_2,\dots,p_{k-2}) = 1\}$ contains all possible isomorphic circulant graphs of $C_{8n}(R)$ of Type-2 w.r.t. $m$ = 2.  Hence we get the result. 
\end{proof}

When the order of the circulant graph is $2^n$, corresponding to the above results, we obtain the following, $n \in \mathbb{N}$ and $2^n \geq 16$.

\begin{cor}{\rm \label{c28} For $n \geq 4$, $1 \leq 2s-1 \leq 2^{n-2}-1$, $R$ = $\{ 2, 2s-1, 2^{n-1}-(2s-1) \}$ and $S$ = $\{2, 2^{n-2}-(2s-1), 2^{n-2}+2s-1 \}$, $\theta_{2^n,2,2^{n-3}}(C_{2^n}(R))$ = $C_{2^n}(S)$ = $\theta_{2^n,2,3\times 2^{n-3}}(C_{2^n}(R))$, $\theta_{2^n,2,2^{n-3}}(C_{2^n}(S))$ = $C_{2^n}(R)$ = $\theta_{2^n,2,3\times 2^{n-3}}(C_{2^n}(S))$, $\theta_{2^n,2,2^{n-2}}(C_{2^n}(R))$ = $C_{2^n}(R)$, $\theta_{2^n,2,2^{n-2}}(C_{2^n}(S))$ = $C_{2^n}(S)$ and $C_{2^n}(R)$ and $C_{2^n}(S)$ are Type-2 isomorphic circulant graphs w.r.t. $m$ = 2. The two graphs are the same when $n$ = 3. }
\end{cor}
\begin{proof} The two circulant graphs are the same when $2s-1$ = $2^{n-2}-(2s-1)$ which implies, $2s-1$ = $2^{n-3}$. This implies, $n-3$ = 0 (and $2s-1$ = 1). That is when $n$ = 3 (and $s$ = 1). The remaining part of the proof follows from Theorem \ref{t26}. 
\end{proof}

\begin{cor}{\rm \label{c29} For $n \geq 4$, $k \geq 3$, $1 \leq 2s-1 \leq 2^{n-2}-1$, $R$ = $\{2s-1,2^{n-1}-(2s-1),2p_1,2p_2,\dots, 2p_{k-2}\}$ and $S$ = $\{2^{n-2}-(2s-1),2^{n-2}+2s-1,2p_1,2p_2,\dots, 2p_{k-2}\}$,  $C_{2^n}(R)$ and $C_{2^n}(S)$ are isomorphic of either Type-1 or Type-2 w.r.t. $m$ = 2 where $2y\in R,S$, $\gcd(2^{n}, 2y)$ = 2 = $m$, $\gcd(p_1,p_2,\dots,p_{k-2})$ = 1 and $n,s,p_1,p_2,\dots,p_{k-2} \in \mathbb{N}$. Moreover, for given $n, s$ and $k$ and for all such possible values of $p_1,p_2,\dots,p_{k-2}$ and $y$, the set $\{ C_{2^n}(S)$ = $\theta_{{2^n},2,{2^{n-3}}}(C_{2^n}(R)):$ $p_1,p_2,\dots,p_{k-2} \in \mathbb{N}\}$ contains all possible isomorphic circulant graphs of $C_{2^n}(R)$ of Type-2 w.r.t. $m$ = 2. }
\end{cor}
\begin{proof} The proof follows from Theorem \ref{t27}. 
\end{proof}

\begin{cor} \label{c30} {\rm For $n \geq 4$, $k \geq 3$, $1 \leq 2s-1 \leq 2^{n-2}-1$, $R$ = $\{2s-1,2^{n-1}-(2s-1),2p_1,2p_2,\dots, 2p_{k-2}\}$, $S$ = $\{2^{n-2}-(2s-1),2^{n-2}+2s-1$, $2p_1,2p_2,\dots, 2p_{k-2}\}$ and for a given set of values of $n,s,k,p_1,p_2,\dots,p_{k-2}$ and $y$, $C_{2^n}(R)$ and $C_{2^n}(S)$ are either Type-1 isomorphic or  Type-2 isomorphic w.r.t. $m$ = 2 where $2y\in R,S$, $\gcd(2^{n}, 2y)$ = 2 = $m$, $\gcd(p_1,p_2,\dots,p_{k-2})$ = 1 and $k,n,s,y,p_1,p_2,\dots,p_{k-2} \in \mathbb{N}$. \hfill $\Box$   }
\end{cor}

 The following is a more general result of Theorem \ref{t25} related to Type-2 isomorphism w.r.t. $m$ = 2.
	
\begin{theorem}{\rm  \quad \label{t31} For $n \geq 2$, $k \geq 3$, $1 \leq 2s-1 \leq 2n-1$, $n \neq 2s-1$, $R$ = $\{2s-1, 4n-(2s-1)$, $2p_1,2p_2,\dots,2p_{k-2}\}$ and $S$ = $\{ 2n-(2s-1)$, $2n+2s-1, 2p_1, 2p_2, \dots, 2p_{k-2} \}$, $\theta_{8n,2,t}(C_{8n}(R))$ = $\theta_{8n,2,2n+t}(C_{8n}(R))$ = $\theta_{8n,2,n+t}(C_{8n}(S))$ = $\theta_{8n,2,3n+t}(C_{8n}(S))$ and $\theta_{8n,2,t}(C_{8n}(S))$ = $\theta_{8n,2,2n+t}(C_{8n}(S))$ = $\theta_{8n,2,n+t}(C_{8n}(R))$ = $\theta_{8n,2,3n+t}(C_{8n}(R))$ where $n,s,p_1,p_2,\dots,p_{k-2} \in \mathbb{N}$, $\gcd(p_1,p_2,\dots,p_{k-2})$ = 1 and $0 \leq t \leq 4n-1$. }
\end{theorem}
\begin{proof}\quad To simplify our work, let $R_1$ = $\{2, 2s-1, 4n-(2s-1), 4n+2s-1, 8n-(2s-1), 8n-2\}$ and $S_1$ = $\{ 2, 2n-(2s-1), 2n+2s-1, 6n-(2s-1), 6n+2s-1, 8n-2 \}$. 
\\	
For $0 \leq t \leq 4n-1$, 
\\
$\theta_{8n,2,t}(R_1)$ = $\theta_{8n,2,t}(\{2, 2s-1, 4n-(2s-1), 4n+2s-1, 8n-(2s-1), 8n-2\})$ 

\hspace{1.2cm} = $\{2, 2s-1+2t, 4n-(2s-1)+2t, 4n+2s-1+2t, 8n-(2s-1)+2t, 8n-2\}$;
\\
$\theta_{8n,2,n+t}(R_1)$ = $\theta_{8n,2,n+t}(\{2, 2s-1, 4n-(2s-1), 4n+2s-1, 8n-(2s-1), 8n-2\})$ 

= $\{2, 2s-1+2n+2t, 4n-(2s-1)+2n+2t, 4n+2s-1+2n+2t, 8n-(2s-1)+2n+2t, 8n-2\}$ 

= $\{2, 2n+2s-1+2t, 6n-(2s-1)+2t, 6n+2s-1+2t, 2n-(2s-1)+2t, 8n-2\}$;
\\
$\theta_{8n,2,2n+t}(R_1)$ = $\theta_{8n,2,2n+t}(\{2, 2s-1, 4n-(2s-1), 4n+2s-1, 8n-(2s-1), 8n-2\})$ 

= $\{2, 2s-1+4n+2t, 4n-(2s-1)+4n+2t, 4n+2s-1+4n+2t, 8n-(2s-1)+4n+2t, 8n-2\}$ 

= $\{2, 4n+2s-1+2t, 8n-(2s-1)+2t, 2s-1+2t, 4n-(2s-1)+2t, 8n-2\}$ = $\theta_{8n,2,t}(R_1)$;
\\
$\theta_{8n,2,3n+t}(R_1)$ = $\theta_{8n,2,3n+t}(\{2, 2s-1, 4n-(2s-1), 4n+2s-1, 8n-(2s-1), 8n-2\})$ 

= $\{2, 2s-1+6n+2t$, $4n-(2s-1)+6n+2t, 4n+2s-1+6n+2t, 8n-(2s-1)+6n+2t, 8n-2\}$ 

= $\{2, 6n+2s-1+2t$, $2n-(2s-1)+2t, 2n+2s-1+2t, 6n-(2s-1)+2t, 8n-2\}$ = $\theta_{8n,2,n+t}(R_1)$;
\\
$\theta_{8n,2,t}(S_1)$ = $\theta_{8n,2,t}(\{2, 2n-(2s-1), 2n+2s-1, 6n-(2s-1), 6n+2s-1, 8n-2\})$ 

\hspace{1.15cm} = $\{2, 2n-(2s-1)+2t, 2n+2s-1+2t, 6n-(2s-1)+2t, 6n+2s-1+2t, 8n-2\}$;
\\
$\theta_{8n,2,n+t}(S_1)$ = $\theta_{8n,2,n+t}(\{2, 2n-(2s-1), 2n+2s-1, 6n-(2s-1), 6n+2s-1, 8n-2\})$ 

= $\{2, 2n-(2s-1)+2n+2t, 2n+2s-1+2n+2t, 6n-(2s-1)+2n+2t, 6n+2s-1+2n+2t, 8n-2\}$ 

= $\{2, 4n-(2s-1)+2t, 4n+2s-1+2t, 8n-(2s-1)+2t, 2s-1+2t, 8n-2\}$ = $\theta_{8n,2,2n+t}(R_1)$;
\\
$\theta_{8n,2,2n+t}(S_1)$ = $\theta_{8n,2,2n+t}(\{2, 2n-(2s-1), 2n+2s-1, 6n-(2s-1), 6n+2s-1, 8n-2\})$ 

= $\{2, 2n-(2s-1)+4n+2t, 2n+2s-1+4n+2t, 6n-(2s-1)+4n+2t, 6n+2s-1+4n+2t, 8n-2\}$ 

= $\{2, 6n-(2s-1)+2t, 6n+2s-1+2t, 2n-(2s-1)+2t, 2n+2s-1+2t, 8n-2\}$ 

= $\theta_{8n,2,t}(S_1)$ = $\theta_{8n,2,n+t}(R_1)$ = $\theta_{8n,2,3n+t}(R_1)$; and
\\
$\theta_{8n,2,3n+t}(S_1)$ = $\theta_{8n,2,3n+t}(\{2, 2n-(2s-1), 2n+2s-1, 6n-(2s-1), 6n+2s-1, 8n-2\})$ 

= $\{2, 2n-(2s-1)+6n+2t, 2n+2s-1+6n+2t, 6n-(2s-1)+6n+2t, 6n+2s-1+6n+2t, 8n-2\}$ 

= $\{2, 8n-(2s-1)+2t, 2s-1+2t, 4n-(2s-1)+2t, 4n+2s-1+2t, 8n-2\}$ 

= $\theta_{8n,2,n+t}(S_1)$ = $\theta_{8n,2,t}(R_1)$ = $\theta_{8n,2,2n+t}(R_1)$.

This implies, $\theta_{8n,2,t}(C_{8n}(R_2))$ = $\theta_{8n,2,2n+t}(C_{8n}(R_2))$ = $\theta_{8n,2,n+t}(C_{8n}(S_2))$ = $\theta_{8n,2,3n+t}(C_{8n}(S_2))$ and $\theta_{8n,2,t}(C_{8n}(S_2))$ = $\theta_{8n,2,2n+t}(C_{8n}(S_2))$ = $\theta_{8n,2,n+t}(C_{8n}(R_2))$ = $\theta_{8n,2,3n+t}(C_{8n}(R_2))$ where $R_2$ = $\{2,2s-1$, $4n-(2s-1)\}$ and $S_2$ = $\{ 2, 2n-(2s-1), 2n+2s-1 \}$. Then the result follows from Remark 
\ref{r22}.
\end{proof}   

For a given circulant graph $C_n(R)$ with $R$ = $\{2,2s-1, 2s'-1\}$, the following theorem gives conditions on $n, t$ and its jump sizes under which the transformed graph $\theta_{n,2,t}(C_n(R))$ and $C_n(R)$ are isomorphic of Type-2 w.r.t. $m$ = 2. Proof given here is different from the one given in \cite{v20}.

\begin{theorem}{\rm \cite{v20}}\quad label{t32} {\rm For $n \geq 2$, $1 \leq 2s-1 < 2s'-1 \leq [\frac{n}{2}]$, $0 \leq t \leq [\frac{n}{2}]$, $R$ = $\{2,2s-1, 2s'-1\}$ and $n,s,s'\in \mathbb{N}$, if $\theta_{n,2,t}(C_n(R))$ and $C_n(R)$ are  isomorphic circulant graphs of Type-2 w.r.t. $m$ = 2 for some $t$, then $n \equiv 0~(mod ~ 8)$, $2s-1+2s'-1$ = $\frac{n}{2}$, $2s-1 \neq \frac{n}{8}$, $t$ = $\frac{n}{8}$ or $\frac{3n}{8}$, $1 \leq 2s-1 \leq \frac{n}{4}$ and $n \geq 16$. In particular, when $R$ = $\{2, 2s-1, 4n-(2s-1)\}$, $S$ = $\{2, 2n-(2s-1), 2n+2s-1\}$, $n\geq 2$ and $n,s\in \mathbb{N}$, $\theta_{8n,2,n}(C_{8n}(R))$ = $C_{8n}(S)$ = $\theta_{8n,2,3n}(C_{8n}(R))$, $\theta_{8n,2,n}(C_{8n}(S))$ = $C_{8n}(R)$ = $\theta_{8n,2,3n}(C_{8n}(S))$, $\theta_{8n,2,2n}(C_{8n}(R))$ = $C_{8n}(R)$, $\theta_{8n,2,2n}(C_{8n}(S))$ = $C_{8n}(S)$ and $C_{8n}(R)$ and $C_{8n}(S)$ are Type-2 isomorphic w.r.t. $m$ = 2. }
\end{theorem}

\begin{proof}\quad Using the definition of $\theta_{n,2,t}$, we have $m$ = 2 divides $\gcd(n,r)$ = $\gcd(n,2)$ = $2$ and $m^3$ = $8$ divides $n$ which implies, $n \equiv 0~(mod ~ 8)$. Let $n$ = $8a$ and so the circulant graph $C_n(R)$ = $C_{8a}(R)$ and $\theta_{n,2,t}$ = $\theta_{8a,2,t}$ where $R$ = $\{ 2, 2s-1, 2s'-1 \}$. Also, $\theta_{8a,2,t}(2)$ = $2$, $\theta_{8a,2,t}(8a-2)$ = $8a-2$, $\theta_{8a,2,t}(2s-1)$ = $2s-1+2t$, $\theta_{8a,2,t}(2s'-1)$ = $2s'-1+2t$, $\theta_{8a,2,t}(8a-2s'+1)$ = $8a-2s'+1+2t$ and $\theta_{8a,2,t}(8a-2s+1)$ = $8a-2s+1+2t$, $0\leq t \leq \left[\frac{n}{2}\right]$ = $4a$ and $a\in\mathbb{N}$. 

Let Figure 15 corresponds to $C_{8a}(R)$ for $R$ = $ \{ 2, 2s-1, 2s'-1 \}$. Let $V(C_{8a}(R))$ = $\{v_0, v_1, \dots, v_{8a-1} \}$, cycles $C$ = $(v_0 v_1 \dots v_{8a-1})$, $C_0$ = $(v_0 v_2 \dots v_{8a-2})$ and $C_1$ = $(v_1 v_3 \dots v_{8a-1})$. Let $u_1$ = $v_{2s-1}$, $u_2$ = $v_{2s'-1}$, $u_3$ = $v_{8a-2s'+1}$ and $u_4$ = $v_{8a-2s+1}$. Clearly, $v_0,v_{2a},v_{4a},v_{6a}\in C_0$ and $u_1,u_2,u_3,u_4\in C_1$. Let $d(u,v)$ denote the distance measured from $u$ to $v$ w.r.t. the vertices $v_0, v_1, \dots, v_{8a-1}$ of the regular $8a$-gon, measured in the clockwise direction. $C_{8a}(R)$ may contain cycles $C$, $C_0$, $C_1$. Also, $d(v_{0}, v_{2a})$ = $d(v_{2a}, v_{4a})$ = $d(v_{4a}, v_{6a})$ = $d(v_{6a}, v_{0})$. See Figure 15. 

Under the transformation, $\theta_{8a,2,t}$ acting on $C_{8a}(R)$, cycle $C_0$ doesn't change but $C_1$ simply rotates so that the relative positions of $u_1, u_2, u_3$ and $u_4$ remain the same in $C_1$. Also, if $\theta_{8a,2,t}(C_{8a}(R))$ is circulant in its representation, then two out of the $4$ vertices $u_1, u_2, u_3$ and $u_4$ lie to the right of $v_0 v_{4a}$ and the other two to the left and correspondingly, the following 3 cases arise. Just to simplify our proof, we consider Figures 16, 17 and 18 corresponding to the three cases. In each case, we get at the most one possible (transformed) circulant graph of the form $C_{8a}(S)$ for some $S \subseteq [1, 4a]$ using the symmetric equidistance condition. See Figures 16, 17 and 18. 

Also, in Theorem \ref{t25}, we have $\theta_{8a,2,a}(C_{8a}(R))$ = $C_{8a}(S)$, $\theta_{8a,2,2a}(C_{8a}(R))$ = $C_{8a}(R)$, $\theta_{8a,2,3a}(C_{8a}(R))$ = $C_{8a}(S)$ and $\theta_{8a,2,4a}(C_{8a}(R))$ = $C_{8a}(R)$.

\medskip
\noindent
{\bf Case 1:} ~ Figure 16. 

From the figure if the transformed graph is circulant in its representation, then $d(v_0,u_4)$ = $d(u_3,v_0)$ (that is distance between $v_0$ and $u_4$ is equal to the distance between $u_3$ and $v_0$) and $d(v_0,u_1)$ = $d(u_2,v_0)$ and thereby $d(u_4,u_1)$ = $d(u_2,u_3)$, see Figure 16. This implies, $2(2s-1)$ = $2(4a-(2s'-1))$. This implies, $4a$ = $2s-1+2s'-1$ and $1 \leq 2s-1 < \frac{4a}{2}$ = $2a$. Also, when $2s-1$ = $a$, then $1 \leq 2s-1 = a < \frac{4a}{2}$ and the circulant graph $C_{8a}(R)$ is symmetric about $v_0v_{4a}$ (and $v_{\frac{4a}{2}}v_{\frac{3\times 4a}{2}})$. For $n$ = $8a$ and $a$ = $2s-1$, the two graphs are identical and for $a \geq 2$, $a \in \mathbb{N}$, $1 \leq 2s-1 \leq 2a$ and $a \neq 2s-1$, using Theorem \ref{t26}, $\theta_{n,2,\frac{n}{8}}(C_{n}(R))$ = $\theta_{8a,2,a}(C_{8a}(R))$ = $C_{8a}(S)$ and  $C_{8a}(R)$ and $C_{8a}(S)$ are Type-2 isomorphic w.r.t. $m$ = 2 where $S$ = $\{2$, $2a-(2s-1)$, $2a+2s-1 \}$, $1 \leq 2a-(2s-1) < 2a+2s-1 \leq 4a-1$. 

\begin{center}
\scalebox{0.8}{ \begin{tikzpicture}[thick, set/.style = { circle, minimum size = .02cm}]
		\node (0) at (-4,11.5) [circle,fill=red!30]{$v_{0}$};
		\node (1) at (-2.75,10.75) [circle][circle,fill=blue!30]{};
		\node (2) at (-2.25,10) [circle,fill=blue!30][label=2:]{};
		\node (3) at (-1.75,9) [circle,fill=blue!30][label=3:]{$u_{1}$};
		\node (4) at (-1.5,8) [circle,fill=blue!30][label=4:]{};
		\node (5) at (-2,5.5) [circle,fill=blue!30][label=5:]{$u_{2}$};
		\node (6) at (-3,4.5) [circle,fill=blue!30][label=6:]{};		
		
		\node (7) at (-4,3.75)[circle,fill=red!30][label=7:]{$v_{4a}$};
		\node (8) at (-4.75,4.5) [circle,fill=blue!30][label=8:]{};
		\node (9) at (-6,5.5) [circle,fill=blue!30][label=9:]{$u_{3}$};
		\node (10) at (-6.5,8) [circle,fill=blue!30][label=10:]{};
		\node (11) at (-6.25,9) [circle,fill=blue!30][label=11:]{$u_{4}$};
		\node (12) at (-6,10) [circle,fill=blue!30][label=12:]{};
		\node (13) at (-5.5,10.75) [circle,fill=blue!30][label=13:]{};
		
		\draw[->] [brown](-2.5,10.4) to [out=120,in=220] (-1.9,10.6);
		\node (14) at (-1.6,10.7) [label=14:]{$C_1$};

		\node (15) at (-.475,7.25) [circle,fill=red!30][label=15:]{$v_{2a}$};
		\node (16) at (-7.25,7.25) [circle,fill=red!30][label=16:]{$v_{6a}$};
		\node (17) at (-1.5,6.75) [circle,fill=blue!30][label=17:]{};
		\node (18) at (-6.35,6.75) [circle,fill=blue!30][label=18:]{};
		\draw [red][dashed](15) -- (16);
		\draw [dashed](4) -- (17);
		\draw [dashed](17) -- (5);
		\draw [dashed](10) -- (18);
		\draw [dashed](18) -- (9);
		\draw [dashed](11) -- (10);
		
		\draw [dashed](1) -- (2);
		\draw [dashed](2) -- (3);
		\draw [dashed](3) -- (4);
		\draw [dashed](5) -- (6);
		\draw [dashed](6) -- (8);
		\draw [dashed](8) -- (9);
		\draw [dashed](11) -- (12);
		\draw [dashed](12) -- (13);
		\draw [dashed](13) -- (1);
		
		\draw (0) -- (3);
		\draw (0) -- (5);
		\draw (0) -- (9);
		\draw (0) -- (11);
		
		\draw [red][dashed](0) -- (7);

		\node (00) at (5,11.5) [circle,fill=red!30]{$v_{0}$};
		\node (01) at (6.25,10.75) [circle][circle,fill=blue!30]{$u_{4}$};
		\node (02) at (6.75,10) [circle,fill=blue!30][label=02:]{};
		\node (03) at (7.25,9) [circle,fill=blue!30][label=03:]{};
		\node (04) at (7.5,7.75) [circle,fill=blue!30][label=04:]{};
		\node (05) at (7,5.5) [circle,fill=blue!30][label=05:]{};
		\node (06) at (6,4.5) [circle,fill=blue!30][label=06:]{$u_{1}$};		
		
		\node (07) at (5,3.75)[circle,fill=red!30][label=07:]{$v_{4a}$};
		\node (08) at (4.25,4.5) [circle,fill=blue!30][label=08:]{};
		\node (09) at (3,5.5) [circle,fill=blue!30][label=09:]{$u_{2}$};
		\node (010) at (2.5,7.75) [circle,fill=blue!30][label=010:]{};
		\node (011) at (2.5,9) [circle,fill=blue!30][label=011:]{};
		\node (012) at (2.5,10) [circle,fill=blue!30][label=012:]{$u_{3}$};
		\node (013) at (3.5,11.25) [circle,fill=blue!30][label=013:]{};

		\node (015) at (8.5,7.25) [circle,fill=red!30][label=015:]{$v_{2a}$};
		\node (016) at (1.5,7.25) [circle,fill=red!30][label=016:]{$v_{6a}$};
		\node (017) at (7.5,6.75) [circle,fill=blue!30][label=017:]{};
		\node (018) at (2.5,6.75) [circle,fill=blue!30][label=018:]{};

		\draw [red][dashed](015) -- (016);
		\draw [dashed](04) -- (017);
		\draw [dashed](017) -- (05);
		\draw [dashed](010) -- (018);
		\draw [dashed](018) -- (09);
		
		\draw [dashed](01) -- (02);
		\draw [dashed](02) -- (03);
		\draw [dashed](03) -- (04);
		\draw [dashed](05) -- (06);
		\draw [dashed](06) -- (08);
		\draw [dashed](08) -- (09);
		\draw [dashed](010) -- (011);
		\draw [dashed](011) -- (012);
		\draw [dashed](012) -- (013);
		\draw [dashed](013) -- (01);
		
		\draw (00) -- (01);
		\draw (00) -- (06);
		\draw (00) -- (09);
		\draw (00) -- (012);
		
		\draw [red][dashed](00) -- (07);
		
		\end{tikzpicture}}\\
	\vspace{0.2cm} 
	Figure $15$ \hspace{5.5cm} Figure $16$  \label{fig 9}			
\end{center}

\vspace{.1cm}
\noindent
{\bf Case 2:} ~ Figure 17. 

In this case, the transformed graph is circulant in its representation when $d(v_0,u_3)$ = $d(u_2,v_0)$ and $d(v_0,u_4)$ = $d(u_1,v_0)$. And let the transformed circulant graph be $C_{8a}(T)$ so that $C_{8a}(T)$ = $\theta_{8a,2,2n}(C_{8a}(R))$ where $T$ = $\{2, 4a-(2s'-1), 4a-(2s-1)\}$. From Table 4, we get $C_{8a}((4a-1)R))$ = $C_{8a}(T)$. This implies, $C_{8a}((4a-1)R)$  $\cong$ $C_{8a}(R)$, and the isomorphism is Type-1 isomorphism since $\gcd(8a, 4a-1)$ = $\gcd(4a, 4a-1)$ = 1. Already we have $C_{8a}(T)$ = $\theta_{8a,2,2a}(C_{8a}(R))$. Thus, in this case, the only possible transformed circulant graph is $C_{8a}(T)$ which is Adam's isomorphic to $C_{8a}(R)$. 

\begin{table}
\caption{Calculation of $(n-1)x$ and $\theta_{2n,2,\frac{n}{2}}(x)$ under arithmetic modulo $2n$.}
\begin{center}
\scalebox{0.8}{
\begin{tabular}{||c||*{6}{c|}}\hline \hline
\backslashbox{\\Operation}{Jump \\size $x$}
& 2 & $2r-1$ & $2s-1$ & $2n-(2s-1)$ & $2n-(2r-1)$ & $2n-2$\\\hline \hline
& &  &   &  &  &  \\
$(n-1)x$ & $2n-2$ & $n-(2r-1)$ & $n-(2s-1)$ & $n+2s-1$ & $n+2r-1$ & 2\\
& &  &   &  &  &  \\\hline
& &  &   &  &  &  \\
$\theta_{2n,2,\frac{n}{2}}(x)$ & 2 & $n+2r-1$ & $n+2s-1$ & $n-(2s-1)$ & $n-(2r-1)$ & $2n-2$\\
& &  &   &  &  &  \\\hline \hline
\end{tabular}}
\end{center}
\end{table}

\noindent
{\bf Case 3:} ~ Figure 18. 

This case is similar to case 1 and the only difference is that the transformed circulant graph $C_{n}(Y)$ = $\theta_{n,2,\frac{3n}{8}}(C_{8a}(R))$ = $\theta_{8a,2,3a}(C_{8a}(R))$ = $C_{8a}(Y)$ where $n$ = $8a$, $a \geq 2$, $R$ = $\{2, 2s-1, 4a-(2s-1)\}$ and $Y$ = $\{2, 2a-(2s-1), 2a+2s-1\}$, $1 \leq 2s-1 \leq 2a$, $a \neq 2s-1$ and $a\in \mathbb{N}$.

Hence the result follows. 
\end{proof}

\begin{center}
\scalebox{0.8}{ \begin{tikzpicture}[thick, set/.style = { circle, minimum size = .02cm}]
		\node (0) at (-4,11.5) [circle,fill=red!30]{$v_{0}$};
		\node (1) at (-2.75,10.75) [circle][circle,fill=blue!30]{};
		\node (2) at (-2.25,10) [circle,fill=blue!30][label=2:]{};
		\node (3) at (-1.75,9) [circle,fill=blue!30][label=3:]{$u_{3}$};
		\node (4) at (-1.5,8) [circle,fill=blue!30][label=4:]{};
		\node (5) at (-2,5.5) [circle,fill=blue!30][label=5:]{$u_{4}$};
		\node (6) at (-3,4.5) [circle,fill=blue!30][label=6:]{};		
		
		\node (7) at (-4,3.75)[circle,fill=red!30][label=7:]{$v_{4a}$};
		\node (8) at (-4.75,4.5) [circle,fill=blue!30][label=8:]{};
		\node (9) at (-6,5.5) [circle,fill=blue!30][label=9:]{$u_{1}$};
		\node (10) at (-6.25,8) [circle,fill=blue!30][label=10:]{};
		\node (11) at (-6.25,9) [circle,fill=blue!30][label=11:]{$u_{2}$};
		\node (12) at (-5.75,10) [circle,fill=blue!30][label=12:]{};
		\node (13) at (-5.25,10.75) [circle,fill=blue!30][label=13:]{};
		
		\node (15) at (-.775,7.25) [circle,fill=red!30][label=15:]{$v_{2a}$};
		\node (16) at (-7.25,7.25) [circle,fill=red!30][label=16:]{$v_{6a}$};
		\node (17) at (-1.5,6.75) [circle,fill=blue!30][label=17:]{};
		\node (18) at (-6.25,6.75) [circle,fill=blue!30][label=18:]{};
		\draw [red][dashed](15) -- (16);
		\draw [dashed](4) -- (17);
		\draw [dashed](17) -- (5);
		\draw [dashed](10) -- (18);
		\draw [dashed](18) -- (9);
		
		\draw [dashed](1) -- (2);
		\draw [dashed](2) -- (3);
		\draw [dashed](3) -- (4);
		\draw [dashed](5) -- (6);
		\draw [dashed](6) -- (8);
		\draw [dashed](8) -- (9);
		\draw [dashed](10) -- (11);
		\draw [dashed](11) -- (12);
		\draw [dashed](12) -- (13);
		\draw [dashed](13) -- (1);
		
		\draw (0) -- (3);
		\draw (0) -- (5);
		\draw (0) -- (9);
		\draw (0) -- (11);
		
		\draw [red][dashed](0) -- (7);

		\node (00) at (5,11.5) [circle,fill=red!30]{$v_{0}$};
		\node (01) at (6.25,10.75) [circle][circle,fill=blue!30]{$u_{2}$};
		\node (02) at (6.75,10) [circle,fill=blue!30][label=02:]{};
		\node (03) at (7.25,9) [circle,fill=blue!30][label=03:]{};
		\node (04) at (7.5,7.75) [circle,fill=blue!30][label=04:]{};
		\node (05) at (7,5.5) [circle,fill=blue!30][label=05:]{};
		\node (06) at (6,4.5) [circle,fill=blue!30][label=06:]{$u_{3}$};		
		
		\node (07) at (5,3.75)[circle,fill=red!30][label=07:]{$v_{4a}$};
		\node (08) at (4.25,4.5) [circle,fill=blue!30][label=08:]{};
		\node (09) at (3,5.5) [circle,fill=blue!30][label=09:]{$u_{4}$};
		\node (010) at (2.5,7.75) [circle,fill=blue!30][label=010:]{};
		\node (011) at (2.5,9) [circle,fill=blue!30][label=011:]{};
		\node (012) at (2.5,10) [circle,fill=blue!30][label=012:]{$u_{1}$};
		\node (013) at (3.5,11.25) [circle,fill=blue!30][label=013:]{};

		\node (015) at (8.5,7.25) [circle,fill=red!30][label=015:]{$v_{2a}$};
		\node (016) at (1.5,7.25) [circle,fill=red!30][label=016:]{$v_{6a}$};
		\node (017) at (7.5,6.75) [circle,fill=blue!30][label=017:]{};
		\node (018) at (2.5,6.75) [circle,fill=blue!30][label=018:]{};

		\draw [red][dashed](015) -- (016);
		\draw [dashed](04) -- (017);
		\draw [dashed](017) -- (05);
		\draw [dashed](010) -- (018);
		\draw [dashed](018) -- (09);
		
		\draw [dashed](01) -- (02);
		\draw [dashed](02) -- (03);
		\draw [dashed](03) -- (04);
		\draw [dashed](05) -- (06);
		\draw [dashed](06) -- (08);
		\draw [dashed](08) -- (09);
		\draw [dashed](010) -- (011);
		\draw [dashed](011) -- (012);
		\draw [dashed](012) -- (013);
		\draw [dashed](013) -- (01);
		
		\draw (00) -- (01);
		\draw (00) -- (06);
		\draw (00) -- (09);
		\draw (00) -- (012);
		
		\draw [red][dashed](00) -- (07);
		
		\end{tikzpicture}}\\
	\vspace{0.2cm} 
	Figure $17$ \hspace{5.5cm} Figure $18$  \label{fig9b}			
\end{center}

\begin{cor} \quad \label{c33} {\rm For $n \geq 2$, $k \geq 3$, $R$ = $\{2s-1,$ $2s'-1, 2p_1, 2p_2, \dots, 2p_{k-2}\}$, $r\in R$, $m$ = 2 divides $\gcd(n, r)$, $m^2$ = 8 divides $n$, $1 \leq 2s-1 < 2s'-1 \leq [\frac{n}{2}]$, $0 \leq t \leq [\frac{n}{2}]$, $\gcd(p_1,p_2,...,p_{k-2})$ = 1 and $n,r,s,s',p_1,p_2,\dots,p_{k-2} \in \mathbb{N}$, if for a given set of values of $p_1,p_2,\dots,p_{k-2},s$ and $s'$,  $\theta_{n,2,t}(C_n(R))$ and $C_n(R)$ are  isomorphic circulant graphs of Type-2 w.r.t. $m$ = 2 for some $t,$ then $n$ $\equiv$ $0~(mod ~ 8),$ $2s-1+2s'-1$ = $\frac{n}{2},$ $2s-1$ $\neq$ $\frac{n}{8},$ $t$ = $\frac{n}{8}$ or $\frac{3n}{8},$ $1 \leq 2s-1 \leq \frac{n}{4}$ and $n \geq 16.$  \hfill $\Box$ }
\end{cor}

\section {Visual Basic program POLY215.EXE}

In this section, we present a VB program POLY215.EXE that we developed to show visually how Type-2 isomorphism of $C_{8n}(R)$ for $m = 2$ and $R$ = $\{2, 2s-1$, $4n-(2s-1)\}$ is taking place, $n \geq 2$ and $n,s\in {\mathbb N}$. Using this program, we obtain Type-1 and Type-2 isomorphic circulant graphs, if exist, for given circulant graph $C_{8n}(R)$ and display its transformed isomorphic graphs $\theta_{8n,2,t}(C_{8n}(R))$ for $t$ = $0,1,2,\dots,4n-1$. 
In \cite{v24} computer programs are developed to generate circulant graphs of finite order and obtained their Type-1 and Type-2 isomorphic circulant graphs, if exist. The VB program is presented below.  

\medskip
\noindent
{\bf Visval Basic program POLY215.EXE:}

We present here Visval Basic program POLY215.exe that we developed to show visually how the transformation $\theta_{8n,2,t}$ acts on $C_{8n}(R)$ for various values of $t$ and also to obtain Type-1 and Type-2 isomorphic circulant graphs, if exist, of $C_{8n}(R)$ where $R$ = $\{2,2r-1,4n-(2r-1)\}$, $1 \leq 2r-1 \leq 2n-1$, $n \geq 2$, $0 \leq t \leq 4n-1$ and $n,r\in \mathbb{N}$.  

\medskip
\noindent
{\bf POLY215.EXE VB Program: $\theta_{8n,2,n}(C_{8n}(\{2,2r-1,4n-(2r-1)\}))$}\\\\
Dim tempx() As Integer\\
Dim tempy() As Integer\\
Dim mov(10) As Integer\\
Dim n\%,r\%,r1\%\\
Dim dis as Double

\vspace{.1cm}
\noindent
Private sub rotate ($n$ As Integer, $r_1$ As Integer)\\
Me.cls\\
$dx_1$ = $line1.x_2 - line1.x_1$\\
$dy_1$ = $line1.y_2 - line1.y_1$\\
For i = 0 To (16*n) - 1\\
If $i$ $\mod$ 2 = 0 Then 

Current $x$ = $line1.x_1+dx_1*cos((360/(16*n))*i*3.1428/180)$-

\hspace{2cm}	$dy_1*sin((360/(16*n))*i* 3.1428/180)$

Current $y$ = $line1.y_1+dx_1*sin ((360/(16*n))*i*3.1428/180)$+

\hspace{2cm}     $dy_1*cos ((360/(16*n))*i*3.1428/180)$\\
Else

Current $x$ = $line1.x_1+dx_1*cos((360/(16*n))*(i+dis)*3.1428/180)$-

\hspace{2cm}        $dy_1*sin((360/(16*n))*(i+dis)*3.1428/180)$

Current $y$ = $line1.y_1+dx_1*sin ((360/(16*n))*(i+dis)*3.1428/180)$+

\hspace{2cm}    $dy_1*cos ((360/(16*n))*(i+dis)*3.1428/180)$\\
End if\\
Current $x$ = current $x$ - 200

Current $y$ = current $y$ + 100\\
If $i$ $\mod$ 2 = 0 Then

Print $``v"$ + str\$(i)\\
Else

$Pr$ = $i$\\
If $Pr$ $<$ 0 Then

$Pr$ = $((16*n) + Pr) \mod (16*n)$\\
End If 

Print $``v"$ + str\$(Pr)\\
End If\\
Next

Call redraw($n,r_1$)

dis = dis + 0.2\\
End sub

\vspace{.2cm}
\noindent
Private Sub first ($n$ As Integer, $r$ As Integer)\\
$r_1$ = $(2*r - 1)$\\
Text1 = $``c"+str\$(16*n)+``("2, +Str\$(r_1)+``)", +Str\$(8*n-r_1)+``)"$\\
$dx_1$ = $line1.x_2 - line1.x_1$\\
$dy_1$ = $line1.y_2 - line1.y_1$\\
$dx_2$ = $line2.x_2 - line2.x_1$\\
$dy_2$ = $line2.y_2 - line2.y_1$\\
For $i$ = 0 To $(16*n) - 1$   \hspace{.5cm}      `For plot points $\&$ labels\\
Current $x$ = $line1.x_1 + dx_1*cos((360/(16*n))*i*3.1428/180)$ -

\hspace{2cm}                $dy_1*sin((360/(16*n))*i*3.1428/180)$\\
Current $y$ = $line1.y_1 + dx_1*sin ((360/(16*n))*i*3.1428/180)$ +

\hspace{2cm}               $dy_1*cos ((360/(16*n))*i*3.1428/180)$
\\
Current $x$ = current $x$ - 200\\
Current $y$ = current $y$ + 100\\
Print $"v"$ + str\$(i)\\
Current $x$ = $line2.x_1+dx_2*cos((360/(16*n))*i*3.1428/180)$ - 

\hspace{2cm}       $dy_2*sin((360/(16*n))*i*3.1428/180)$\\
Current $y$ = $line2.y_1 + dx_2*sin ((360/(16*n))*i*3.1428/180)$ +

\hspace{2cm}             $dy_2*cos ((360/(16*n))*i*3.1428/180)$\\
Temp $x(i)$ = current $x$\\
Temp $y(i)$ = current $y$\\
Print ”.”

`MsgBox str\$(current $x$) + ``~" + str\$(current y)\\
Next

For $j$ = 0 To $(16*n) - 1$ \hspace{.5cm}      `For first line

Me.line (temp $x(j)$ - 10, temp $y(j)$ + 200) - (temp $x(j + r_1)$ 

\hspace{.5cm} $\mod$ $(16*n))$ - 10, temp $y((j + r_1) \mod$ $(16*n))$ + 200, vb Red\\
Next\\
For $k$ = 0 To $(16*n)$ - 1  step 2 \hspace{.2cm}  `For second line

Me.line (temp $x(k)$ - 10, temp $y(k)$ + 200) - (temp $x((k + 2)$ 

\hspace{.5cm} $\mod$ $(16*n))$ - 10, temp $y((k + 2)$ $\mod$ $(16*n))$ + 200, vb Blue\\
Next\\
For $i$ = 0 To $(16*n)$ - 1  \hspace{.2cm}     `For Third line

Me.line (temp $x(i)$-10, temp $y(i)$+200)-(temp $x((i+(8*n-r_1))$  

\hspace{.2cm} $\mod$ $(16*n))$-10,	temp $y((i+(8*n-r_1))$ $\mod$ $(16*n))$+200, vb Green\\
Next\\
Line 1.visible = False

Line 2.visible = False\\
End Sub

\vspace{.2cm}
\noindent
Private Sub redraw ($n$ As Integer, $r_1$ As Integer)\\
$dx_2$ = $line2.x_2 - line2.x_1$\\
$dy_2$ = $line2.y_2 - line2.y_1$\\
For $i$ = 0 To $(16*n)$ - 1\\
If $i$ $\mod$ 2 = 0 Then

Current $x$ = $line2.x_1 + dx_2*cos((360/(16*n))*i*3.1428/180)$ - 

\hspace{2cm}	$dy_2*sin((360/(16*n))*i*3.1428/180)$

Current $y$ = $line2.y_1 + dx_2*sin ((360/(16*n))*i*3.1428/180)$ +

\hspace{2cm}	$dy_2*cos ((360/(16*n))*i*3.1428/180)$\\
Else

Current $x$ = $line2.x_1 + dx_2*cos((360/(16*n))*(i+dis)*3.1428/180)$  

\hspace{2cm}	-$dy_2*sin((360/(16*n))*(i+dis)*3.1428/180)$

Current $y$ = $line2.y_1 + dx_2*sin ((360/(16*n))*(i+dis)*3.1428/180)$

\hspace{2cm}	+ $dy_2*cos ((360/(16*n))*(i+dis)*3.1428/180)$\\
End if

Temp $x(i)$ = current $x$

Temp $y(i)$ = current $y$

Print ``.”\\
Next

If $r_1$ = 1 Then

Text1 = $``c" + str\$(16*n) + ``(" + Str\$(r_1) +``, 2$, 

\hspace{4cm} $" + Str\$(8*n-r_1) + ``)"$\\
Else

Text1 = $``c" + str\$(16*n) + ``(2," + Str\$(r_1) + ``$, 

\hspace{4cm} $" + Str\$(8*n -r_1) + ``)"$\\
End If \\
For $j$ = 0 To $(16*n)$ - 1       

Me.line (temp $x(j + r_1)$ $\mod$ $(16*n))$ - 10,  temp $y(j + r_1)$ $\mod$ $(16*n))$

\hspace{2cm}	+ 200) - (temp $x(j)$ - 10, temp $y(j)$ + 200), vb Red\\
Next\\
For $k$ = 0 To $(16*n)$ - 1 step 2     

Me.line (temp $x(k)$ - 10, temp $y(k)$ + 200) - (temp $x((k + 2)$ 

\hspace{2cm}  $\mod$ $(16*n))$ - 10, temp $y((k + 2)$ $\mod$ $(16*n))$ + 200), vb Blue

Me.line (temp $x(k+1)$-10, temp $y(k+1)$+200)-(temp $x((k+3)$  

\hspace{2cm} $\mod$ $(16*n))$ - 10, temp $y((k + 3)$ $\mod$ $(16*n))$ + 200), vb Blue\\
Next\\
For $i$ = 0 To $(16*n)-1$       

Me.line (temp $x((i+(8*n-r_1))$ $\mod$ $(16*n))$-10, temp $y(i+(8*n-r_1))$

\hspace{2cm}	$\mod$ $(16*n))$+200) - (temp $x((i)$-10, temp $y((i)$+200), vb Green\\
Next

Line 1.visible = False

Line 2.visible = False\\
End Sub

\vspace{.2cm}
\noindent      
Private Sub command 1\_click()\\
End\\
End sub

\vspace{.2cm}	
\noindent
Private Sub command 2\_click()\\
If Command 2.caption = ``stop" Then

Command 2.caption = ``continue"

Timer1.Enabled = False\\
Else

Command 2.caption = ``stop" 

Timer1.Enabled = True\\
End If\\
End sub

\vspace{.2cm}
\noindent
Private Sub command 3\_click()\\
On Error Resume Next

$n$ = Input Box (``Enter $n$", ``New Value", 1)

ReDim temp $x(n*16)$

ReDim temp $y(n*16)$\\
`MsgBox $r$

`call first $(n,r)$\\
End sub

\vspace{.2cm}
\noindent
Private Sub Form\_Activate ()\\ 
`$n$ = Input Box (``Enter $n$")\\  
ReDim temp $x(n*16)$ \\ 
ReDim temp $y(n*16)$  

$r$ = $2*n$

$r_1$ = $(2*r)-1$

HScroll1.Max = $r_1$

call first $(n,1)$\\
End sub

\vspace{.2cm}
\noindent
Private Sub Form\_Load()\\
$n$ = 1\\
dis = 0.2\\
End sub

\vspace{.2cm}
\noindent
Private Sub HScroll1\_Change ()\\
Call rotate ($n$, HScroll1.Value)\\
End Sub

\vspace{.2cm}
\noindent
Private Sub Timer1\_Timer ()\\
Call rotate ($n$, HScroll1.Value)\\
End Sub

\vspace{.1cm}
OUTPUT\\
-------------------------------------------------------------

\section {Conclusion}

\begin{enumerate}
\item Most of the results derived in this paper on Type-2 isomorphism are ralated to circulant graphs with even number of isomorphic circulant subgraphs ($m_i$ is even). Type-2 isomorphism on circulant graphs with odd number of isomorphic circulant  subgraphs ($m_i$ is odd) are obtained in \cite{v24}, \cite{vw1}-\cite{vw3} using the previous definition of Type-2 isomorphism and in \cite{v2-5} we use definition 
\ref{d16} to obtain Type-2 isomorphic circulant graphs w.r.t. $m$ = $p$ and of order $np^3$ where $p$ is a prime number and $n\in\mathbb{N}$.  

\item We could see that for $n \geq 16$, $0 \leq t \leq \frac{n}{m}-1$, $R = \{r_1,r_2,\dots,r_k\}$ and $x \in \varphi_n$, $\varphi_{n,x}(C_n(R))$ and $\theta_{n,m,t}(C_n(R))$ may represent different forms of circulant graph $C_n(R)$ where $V(C_{n}(R))$ $=$ $\{v_0$, $v_1, \dots, v_{n-1}\}$, $V(C_{n}(\{1,2,\dots,\left\lfloor \frac{n}{2} \right\rfloor))$ = $V(K_n)$ = $\{u_0, u_1, \dots,$ $u_{n-1}\}$, $r \in R$, $m > 1$ divides $\gcd(n,r)$ and $m^3$ divides $n$. Now, the question is,\lq\lq Is it possible to find out more forms of representation of graphs which are isomorphic to $C_n(R)?$\rq\rq. Yes, similar to $\theta_{n,m,t},$ we define the following mappings: 

\begin{enumerate}
\item Let $r \in \mathbb{Z}_n$, $m > 1$ be a diviser of $\gcd(n,r)$ and $m^3$ be a diviser of $n$. Define bijective mapping $\theta_{n,m,t_0,t_1,\dots,t_{m-1}}:$ $\mathbb{Z}_n$ $\rightarrow$ $\mathbb{Z}_n$ such that $\theta_{n,m,t_0,t_1,\dots,t_{m-1}}(x) = (q+jt_j)m+j$ = $\theta_{n,m,t_j}(x),$ under arithmetic modulo $n$  where $x$ = $qm+j$, $x \in \mathbb{Z}_n$, $0 \leq i,j \leq m-1$, $0 \leq q,t_i,t_i' \leq \frac{n}{m}-1$, $i,j,q,t_i,t_i' \in \mathbb{Z}_n$ and define $\circ$ in $\{\theta_{n,m,t_0,t_1,\dots,t_{m-1}}:$ $t_i\in [0,~\frac{n}{m}-1]$, $0 \leq i \leq m-1 \}$ such that $\theta_{n,m,t_0,t_1,\dots,t_{m-1}}$ $\circ$ $\theta_{n,m,t'_0,t'_1,\dots,t'_{m-1}}$ = $\theta_{n,m,t_0+t'_0,t_1+t'_1,\dots,t_{m-1}+t'_{m-1}}$ and $t_i+t'_i$ is calculated under addition modulo $\frac{n}{m}$, $0 \leq i \leq m-1$. 

Similarly, define one-to-one mapping $\theta_{n,m,t_0,t_1,\dots,t_{m-1}}:$ $V(C_n(R))$ $\rightarrow$ $V(K_n)$ for a set $R$ = $\{ r_1,r_2,\dots,r_k,n-r_k,n-r_{k-1} \dots, n-r_1 \}$ such that $\theta_{n,m,t_0,t_1,\dots,t_{m-1}}(v_x)$ = $u_{j+(q+jt_j)m}$ = $\theta_{n,m,t_j}(v_x)$ and $\theta_{n,m,t_0,t_1,\dots,t_{m-1}}$ $((v_x , v_{x+s}))$ = $(\theta_{n,m,t_0,t_1,\dots,t_{m-1}}(v_x)$, $\theta_{n,m,t_0,t_1,\dots,t_{m-1}}(v_{x+s}))$ under  arithmetic modulo $n$ where $x\in \mathbb{Z}_n,$ $x$ = $qm+j$, $0 \leq j \leq m-1$, $s\in R$, $0 \leq q,t_j \leq \frac{n}{m}-1$, $(v_x , v_{x+s})\in E(C_n(R))$, $V(C_n(R)) = \{v_0,v_1,v_2,\dots,v_{n-1}\}$ and $V(K_n) = \{u_0,u_1,u_2,\dots,u_{n-1}\}$. Under this mapping each $\Gamma_j$ of $C_n(R)$ rotates $jt_jm$ positions w.r.t. the regular $n$-gon in the clockwise direction and the mapping preserves adjacency, $j$ = $0,1,2,\dots,\frac{n}{m}-1$. By considering $V(C_n(R)) = A$ and $\Gamma_j$ = $A_j$ in Lemma \ref{l6}, we get $C_n(R)$ $\cong$ $\theta_{n,m,t_0,t_1,\dots,t_{m-1}}(C_n(R))$, $j = 0,1,\dots,m-1$.

\item Let $r \in \mathbb{Z}_n$, $m > 1$ divide $\gcd(n,r)$, $m^3$ divide $n$ and $\theta_{n,m,(x_0,t_0),(x_1,t_1),\dots,(x_{m-1},t_{m-1})}: \mathbb{Z}_n \rightarrow \mathbb{Z}_n$ and $\theta_{n,m,(x_j,t_j)}: \mathbb{Z}_n$ $\rightarrow \mathbb{Z}_n$  $\ni$ $\theta_{n,m,(x_0,t_0),(x_1,t_1),\dots,(x_{m-1},t_{m-1})}(x)$ = $\theta_{n,m,(x_j,t_j)}(x)$ = $j+(x_jq+t_j)m$ be bijective mappings under arithmetic modulo $n$ where $x$ = $qm+j$, $x \in \mathbb{Z}_n$, $0 \leq j \leq m-1$, $\gcd(\frac{n}{m}, x_j)$ = 1 and $0 \leq q,t_j \leq \frac{n}{m}-1$. Similarly, define 1-1 mappings, $\theta_{n,m,(x_0,t_0),(x_1,t_1),\dots,(x_{m-1},t_{m-1})}:$ $V(C_n(R)) \rightarrow$ $V(K_n))$ and $\theta_{n,m,(x_j,t_j)}: V(C_n(R)) \rightarrow V(K_n)$ for a set $R$ = $\{ r_1,r_2,\dots$, $r_k,n-r_k, \dots,n-r_1 \}$ $\ni$ $\theta_{n,m,(x_0,t_0),(x_1,t_1),\dots,(x_{m-1},t_{m-1})}(v_x) =$ $u_{j+(x_jq+t_j)m}$ = $\theta_{n,m,(x_j,t_j)}(v_x)$ and  $\theta_{n,m,(x_0,t_0),(x_1,t_1),\dots,(x_{m-1},t_{m-1})}((v_x, v_{x+s}))$ = $(\theta_{n,m,(x_j,t_j)}$ $(v_x)$, $\theta_{n,m,(x_i,t_i)}( v_{x+s}))$, under subscript arithmetic modulo $n$ where $m > 1$ divides $\gcd(n,r)$, $m^3$ divides $n$, $x$ = $qm+j$, $x+s$ = $q'm+i$, $\gcd(\frac{n}{m}, x_j)$ = 1 = $\gcd(\frac{n}{m}, x_i)$, $0 \leq i,j \leq m-1$, $0 \leq q,q',t_i,t_j \leq$ $\frac{n}{m}-1$, $x,x+s \in \mathbb{Z}_n$ and $r,s \in R$. Here, the mapping is 1-1 since under the above transformation, the set of vertices of each $\Gamma_j$ of $C_n(R)$ is mapped on to itself follows from $\gcd(\frac{n}{m}, x_j)$ = 1, $0 \leq j \leq m-1$ and the mapping preserves adjacency. And $\theta_{n,m,(x_0,t_0), (x_1,t_1),\dots, (x_{m-1},t_{m-1})}(C_n(R))$ $\cong$ $C_n(R)$, using Lemma 
\ref{l6}. Under the above transformation, for $j$ = $0,1,\dots,m-1$, the movement of each $\Gamma_{j+1}$ need not be uniform w.r.t. $\Gamma_j$ in the regular $n$-gon. 

\item Let $r \in \mathbb{Z}_n$, $m > 1$ divide $\gcd(n,r)$ and $m^3$ divide $n$. Define one-to-one onto mappings, $\mu_{n,m,(x_0,t_0),(x_1,t_1), \dots,(x_{m-1},t_{m-1})}:$ $\mathbb{Z}_n$ $\rightarrow$ $\mathbb{Z}_n$ and $\mu_{n,m,(x_j,t_j)}:$ $\mathbb{Z}_n$ $\rightarrow$ $\mathbb{Z}_n$ $\ni$ $\mu_{n,m,(x_j,t_j)}(x)$ = $\mu_{n,m,(x_0,t_0),(x_1,t_1),\dots,(x_{m-1},t_{m-1})}(x)$ = $j+(x_jq+jt_j)m$ under arithmetic modulo $n$ where $x = qm+j$, $x \in \mathbb{Z}_n$, $\gcd(\frac{n}{m}, x_j)$ = 1, $0 \leq j \leq m-1$ and $0 \leq q,t_j \leq \frac{n}{m}-1$. Similarly, define mappings, $\mu_{n,m,(x_0,t_0),(x_1,t_1),\dots,(x_{m-1},t_{m-1})}:$ $V(C_n(R))$ $\rightarrow$ $V(K_n)$ and $\mu_{n,m,(x_j,t_j)}:$ $V(C_n(R))$ $\rightarrow$ $V(C_n(1,2,\dots,n-1))$ which are 1-1 for a set $R$ = $\{r_1,r_2,\dots,r_k,n-r_k$, $n-r_{k-1},\dots,n-r_1\}$ $\ni$ $\mu_{n,m,(x_0,t_0),(x_1,t_1), \dots,(x_{m-1},t_{m-1})}(v_x)$ = $u_{j+(x_jq+jt_j)m}$ = $\mu_{n,m,(x_j,t_j)}(v_x)$ and $\mu_{n,m,(x_0,t_0),(x_1,t_1),\dots, (x_{m-1},t_{m-1})}((v_x, v_{x+s}))$ = $(\mu_{n,m, (x_j,t_j)}(v_x)$, $\mu_{n,m,(x_i,t_i)}(v_{x+s}))$ under arithmetic modulo $n$ where $m > 1$ divides $\gcd(n,r)$, $m^3$ divides $n$, $x = qm+j$, $x+s$ = $q'm+i$, $\gcd(\frac{n}{m}, x_j)$ = 1 = $\gcd(\frac{n}{m}, x_i)$, $0 \leq q,q',t_i,t_j \leq \frac{n}{m}-1$, $0 \leq i,j \leq m-1$, $x,x+s\in \mathbb{Z}_n$ and $r,s\in R$. Here the mapping is one-to-one since under the above transformation, the set of vertices of each $\Gamma_j$ of $C_n(R)$ is mapped onto itself follows from $\gcd(\frac{n}{m}, x_j)$ = 1, $0 \leq j \leq m-1$ and the mapping preserves adjacency. And $C_n(R)$ $\cong$ $\mu_{n,m,(x_0,t_0),(x_1,t_1),\dots,(x_{m-1},t_{m-1})}(C_n(R))$, using Lemma \ref{l6}. For $j$ = $0,1,\dots,m-1$, under the above transformation, the movement of each $\Gamma_{j+1}$ need not be uniform w.r.t. $\Gamma_j$ in the regular $n-gon$.  
\end{enumerate}

\item Transformations $\varphi_{n,x}$ defined on $\mathbb{Z}_n$ and $\theta_{n,m,t}$  defined on $C_n(R)$ act on their respective elements, elements of $\mathbb{Z}_n$ and $V(C_n(R))$, uniformly whereas transformations $\theta_{n,m,t_0,t_1,\dots,t_{m-1}}$, $\theta_{n,m,(x_0,t_0),(x_1,t_1),\dots,(x_{m-1},t_{m-1})}$  and  $\mu_{n,m,(x_0,t_0),(x_1,t_1),\dots,(x_{m-1},t_{m-1})}$ act uniformly when $t_0$ = $t_1$ = $\dots$ = $t_{m-1}$ = $t$, say, and $x_0$ = $x_1$ = $\dots$ = $x_{m-1}$ = 1. And in this case all these transformations become $\theta_{n,m,t}$. One can do extensive research work on the above transformations and try to find out more representations of circulant graphs and their algebraic properties. 
\end{enumerate}

\medskip
\noindent
\textbf{Declaration of competing interest}\quad 
The author declares that he has no conflict of interest.

\vspace{.2cm}
\noindent
\textbf{Acknowledgements}\quad I express my sincere thanks to Prof. S. Krishnan (late), Prof. V. Mohan and Prof. R. Aravamuthan (late), Thiyagarayar College of Engineering, Madurai, Tamil Nadu, India; Prof. M. I. Jinnah (late), Mr. Albert Joseph and Prof. L. John, University of Kerala, Trivandrum, Kerala, India; Prof. K. Vareethiah Konstantine and Prof. S. Amirthaiyan, St. Jude's College, Thoothoor, K. K. District, Tamil Nadu; Mr. R. Benadict, Headmaster (Rtd), Pius XI Higher Secondary School, Thoothoor; Dr. Oscar Fredy, Royal Liverpool University Hospital, Liverpool, U.K. and Prof. Lowell W Beineke, Purdue University, USA for their help and guidance and our sincere thanks to Dr. P. Wilson, Dr. A. Christopher and Mr. R. Satheesh of S.T. Hindu College, Nagercoil, India for their assistance to develop the VB programs. I also express my gratitude to the Central University of Kerala, Kasaragod, Kerala; St. Jude's College, Thoothoor and S. T. Hindu College, Nagercoil; and Lerroy Wilson Foundation, India (www.WillFoundation.co.in) for providing facilities to do this research work.


\end{document}